\documentclass{amsart}
\usepackage{amssymb,amsmath,amsthm,amsfonts}
\def\R {\mathbb{R}}
\def\C {\mathcal{C}}
\def\N {\mathbb{N}}
\def\S {\mathbb{S}}

\def\from {\colon}
\def\bv {\mathbf{v}}
\def\bu {\mathbf{u}}
\def\bw {\mathbf{w}}
\def\bW {\mathbf{W}}
\def\eps{\varepsilon}
\def\dist{{\rm dist}}

\newcommand{\loc}{\mathrm{loc}}
\newcommand{\sgn}{\mathrm{sgn}}
\renewcommand{\div}{\mathrm{div}}

\newcommand{\problem}[1] {(P)_{#1}}
\newcommand{\probbd}[1] {(PD)_{#1}}

\newcommand{\classG} {\mathcal{G}}
\newcommand{\classH} {\mathcal{H}}
\newcommand{\nuACF} {\nu^{\mathrm{ACF} }}
\newcommand{\nuLio} {\nu^{\mathrm{Liou} }}

\newcommand{\de}[1] {\mathrm{d} #1}

\newcommand{\ddfrac}[2] {\frac{\displaystyle #1 }{\displaystyle #2} }

\DeclareMathOperator*{\tsum}{\textstyle{\sum}}

\DeclareMathOperator*{\osc}{osc}

\newtheorem{proposition}{Proposition}[section]
\newtheorem{theorem}[proposition]{Theorem}
\newtheorem{corollary}[proposition]{Corollary}
\newtheorem{lemma}[proposition]{Lemma}

\theoremstyle{definition}
\newtheorem{definition}[proposition]{Definition}
\newtheorem{remark}[proposition]{Remark}
\numberwithin{equation}{section}

\title[Segregation problems involving half laplacians]{Uniform H\"older bounds for strongly competing systems  involving the square root of the laplacian}
\author{Susanna Terracini}
\email{susanna.terracini@unito.it}
\address{Dipartimento di Matematica "Giuseppe Peano", Universit\`a degli Studi
di Torino, Via Carlo Alberto 10, 10123 Torino, Italy}

\author{Gianmaria Verzini }
\email{gianmaria.verzini@polimi.it}
\address{Dipartimento di Matematica, Politecnico di Milano, p.za Leonardo da
Vinci 32,  20133 Milano, Italy}

\author{Alessandro Zilio}
\email{alessandro.zilio@mail.polimi.it}
\address{Dipartimento di Matematica, Politecnico di Milano, p.za Leonardo da
Vinci 32,  20133 Milano, Italy}

\thanks{Work partially supported by the PRIN2009 grant ``Critical Point Theory and
Perturbative Methods for Nonlinear Differential Equations''.}

\begin{document}

\begin{abstract}
For a class of competition-diffusion nonlinear systems involving the square root of the Laplacian, including the fractional Gross-Pitaevskii system
\[
(-\Delta)^{1/2} u_i=\omega_i u_i^3 + \lambda_i u_i
-\beta u_i\sum_{j\neq i}a_{ij}u_j^2,\qquad i=1,\dots,k,
\]
we prove that $L^\infty$ boundedness implies $\C^{0,\alpha}$ boundedness for every $\alpha\in[0,1/2)$, uniformly as $\beta\to +\infty$. Moreover we prove that the limiting profile is $\C^{0,1/2}$.
This system arises, for instance, in the relativistic Hartree-Fock approximation theory for
$k$-mixtures of Bose-Einstein condensates in different hyperfine states.
\end{abstract}
\maketitle

\section{Introduction}
Regularity issues involving fractional laplacians are very challenging, because of the genuinely non-local nature of such operators, and for this reason they have recently become the object of an intensive research, especially when associated with the asymptotic analysis and the study of free boundary problems, see for instance
\cite{css,atcafsal,garpetr,crs,MR2680400,MR2675483,DeSRoq} and references therein. The present paper is concerned with this topic,
when the creation of a free boundary is triggered by the interplay between fractional diffusion and \emph{competitive interaction}.

Several physical phenomena can be described by a certain number of densities (of mass, population,
probability, ...) distributed in a domain and subject to laws of diffusion, reaction, and
competitive interaction. Whenever the competition is the prevailing feature, the
densities tend to segregate, hence determining a partition of the domain. When anomalous diffusion
is involved, one is lead to consider the class of stationary systems of semilinear equations
\begin{equation*}
\begin{cases}
(-\Delta)^s u_i=f_i(x,u_i)-\beta u_i\sum_{j\neq i}g_{ij}(u_j)
\\u_i\in H^{s}(\R^N),
\end{cases}
\end{equation*}
thus focusing on the singular limit problem obtained when the (positive) parameter $\beta$,
accounting for the
competitive interactions, diverges to  $\infty$. Among the others, the cases
$f_i(s)=r_is(1-s/K_i)$, $g_{ij}(s)=a_{ij}s$ (logistic internal dynamics with Lotka-Volterra competition) and $f_i(s)=\omega_i s^3 + \lambda_i s$, $g_{ij}(s)=a_{ij}s^2$ (focusing-defocusing Gross-Pitaevskii system with competitive interactions, see for instance \cite{MR2090357, MR2435464}) are of the highest interest in the applications to population dynamics and theoretical physics, respectively.

For the standard Laplace diffusion operator (namely $s=1$),  the analysis of the qualitative properties of solutions to the corresponding systems
has been undertaken, starting from \cite{MR2090357, MR1939088, MR1962357}, in a series of recent papers \cite{ctv, ww,MR2529504, MR2393430, nttv}, also in the parabolic case \cite{wz,dwz1,dwz2,dwz3}.  In the singular limit one finds a vector $\bu=(u_1,\cdots,u_k)$ of limiting profiles  with mutually disjoint supports: indeed, \emph{the segregated states} $u_i$ satisfy $u_i\cdot u_j\equiv 0$, for $i\neq j$, and
\begin{equation*}
-\Delta u_i=f_i(x,u_i) \qquad\text{ whenever }\;u_i \neq 0\;, \ \ i=1,\ldots,k.
\end{equation*}
Natural questions concern the functional classes of convergence (a priori bounds), optimal regularity of the limiting profiles, equilibrium conditions at the interfaces, and regularity of the nodal set. In
\cite{ctv} (for the Lotka-Volterra competition) and \cite{nttv} (for the variational Gross-Pitaevskii one) it is proved that $L^\infty$ boundedness implies $\C^{0,\alpha}$ boundedness, uniformly as $\beta\to+\infty$, for every $\alpha\in(0,1)$. Moreover, in the second case, it is shown that the limiting profiles are Lipschitz continuous. The proof relies upon elliptic estimates,  the blow-up technique, the monotonicity formulae by Almgren \cite{almgren} and Alt-Caffarelli-Friedman \cite{acf}, and it reveals a subtle interaction between diffusion and competition aspects. This interaction mainly occurs at two levels: the validity and exactness of the  Alt-Caffarelli-Friedman monotonicity formula and, consequently, the validity of Liouville type theorems for entire solutions to semilinear systems.

In this paper we address the problem of a priori bounds and optimal regularity of the limiting profiles in the simplest case of anomalous diffusion, driven by the square root of the laplacian. As well known, anomalous diffusion arises when the Gaussian statistics of the classical Brownian motion is replaced by a different one, allowing for the L\'evy jumps (or flights).  In the light of already built theory  for the regular laplacian, we focus on the joint effect of diffusion and competition as the (non local) diffusion process acts on a longer range.

Our model problem will be the following:
\begin{equation}
 \begin{cases}\label{eq:initial_problem}
(-\Delta)^{1/2} u_i=f_{i,\beta}(u_i)-\beta u_i\sum_{j\neq i} u_j^2\\
u_i\in H^{1/2}(\R^N).
\end{cases}
\end{equation}
This class of problems includes the already mentioned Gross-Pitaevskii systems with focusing or defocusing nonlinearities
\begin{equation*}
\begin{cases}
(-\Delta + m_i^2)^{1/2}u_i=\omega_iu_i^3+\lambda_{i,\beta}u_i-\beta u_i\sum_{j\neq i} a_{ij}u^2_j\;\\
u_i\in H^{1/2}(\R^N),
\end{cases}
\end{equation*}
with $a_{ij}=a_{ji}>0$, which is the relativistic version of the  Hartree-Fock approximation theory for mixtures of Bose-Einstein condensates in different hyperfine states. Even though we will perform the proof in the case $m_i=0$ (and $a_{ij}=1$), the general case, allowing positive masses $m_i>0$, follows with minor changes and it is actually a bit simpler.

As it is well known (see e.g. \cite{cs}), the $N$-dimensional half laplacian can be interpreted as a Dirichlet-to-Neumann operator and solutions to problem \eqref{eq:initial_problem} as traces of harmonic functions on the $(N+1)$-dimensional half space having the right-hand side of  \eqref{eq:initial_problem} as normal derivative. For this reason, it is worth stating our main results for harmonic functions with nonlinear Neumann boundary conditions involving strong competition terms. We use the following notations: for any dimension $N\geq1$, we consider the half ball $B^+_r(x_0,0):= B_r(x_0,0)\cap\{y>0\}$, which boundary contains the spherical part $\partial^+B^+_r :=\partial B_r \cap\{y>0\}$ and the flat one $\partial^0B^+_r :=B_r \cap\{y=0\}$ (here $y$ denotes the $(N+1)$-th coordinate).
\begin{theorem}[Local uniform H\"older bounds]\label{thm: intro_local}
Let the functions $f_{i,\beta}$ be continuous and uniformly bounded (w.r.t. $\beta$) on bounded sets, and let $\{\bv_{\beta}=(v_{i,\beta})_{1\leq i\leq k}\}_{\beta}$ be a family of $H^1(B^+_1)$ solutions to the problems
\[
    \begin{cases}
    - \Delta v_i = 0 & \text{in } B^+_1\\
    \partial_{\nu} v_i = f_{i,\beta}(v_i) - \beta v_i \tsum_{j \neq i} v_j^2 & \text{on } \partial^0 B^+_1.
    \end{cases} \eqno \problem{\beta}
\]
Let us assume that
\[
    \| \bv_{\beta} \|_{L^{\infty}(B^+_1)} \leq M,
\]
for a constant $M$ independent of $\beta$. Then for every $\alpha \in (0,1/2)$ there exists a constant
$C = C(M,\alpha)$, not depending on $\beta$, such that
\[
    \| \bv_\beta\|_{\C^{0,\alpha}\left(\overline{B^+_{1/2}}\right)} \leq C(M,\alpha).
\]
Furthermore, $\{\bv_{\beta}\}_{\beta }$ is relatively compact in $H^1(B^+_{1/2}) \cap \C^{0,\alpha}\left(\overline{B^+_{1/2}}\right)$ for every $\alpha < 1/2$.
\end{theorem}
As a byproduct, up to subsequences we have convergence of the above solutions to a limiting profile, which components are segregated on the boundary $\partial^0 B^+$. If furthermore $f_{i,\beta}\to f_i$, uniformly on compact sets, we can prove that this limiting profile satisfies
\[
    \begin{cases}
    - \Delta v_i = 0 & \text{in } B^+_1\\
    v_i \partial_{\nu} v_i = f_{i}(v_i)v_i & \text{on } \partial^0 B^+_1.
    \end{cases}
\]
One can see that, for solutions of this type of equation, the highest possible regularity correspond to the H\"older exponent $\alpha=1/2$. As a matter of fact, we can prove that the limiting profiles do enjoy such optimal regularity.
\begin{theorem}[Optimal regularity of limiting profiles]\label{thm: intro_limiting_prof}
Under the assumptions above, assume moreover that the locally Lipschitz continuous functions $f_i$ satisfy $f_i(s) = f_i'(0)s + O(|s|^{1+\eps})$ as $s\to 0$, for some $\eps>0$. Then $\bv\in\C^{0,1/2}\left(\overline{B^+_{1/2}}\right)$.
\end{theorem}

Once local regularity is established, we can move from $\problem{\beta}$ and deal with global problems, adding suitable boundary conditions. An example of results that we can prove is the following.
\begin{theorem}[Global uniform H\"older bounds]\label{thm: intro_global}
Let the functions $f_{i,\beta}$ be continuous and uniformly bounded (w.r.t. $\beta$) on bounded sets, and let $\{\bu_{\beta}\}_{\beta }$ be a family of $H^{1/2}(\R^N)$ solutions to the problems
\[
    \begin{cases}
    ( -\Delta)^{1/2} u_i= f_{i,\beta}(u_i) - \beta u_i \tsum_{j \neq i} u_j^2 & \text{on } \;\Omega\\
    u_i\equiv 0 &\text{on} \; \R^N\setminus \Omega,
    \end{cases}
\]
where $\Omega$ is a bounded domain of $\R^N$, with sufficiently  smooth boundary. Let us assume that
\[
    \| \bu_{\beta} \|_{L^{\infty}(\Omega)} \leq M,
\]
for a constant $M$ independent of $\beta$. Then for every $\alpha \in (0,1/2)$ there exists a constant
$C = C(M,\alpha)$, not depending on $\beta$, such that
\[
    \| \bu_\beta\|_{\C^{0,\alpha}(\R^N)} \leq C(M,\alpha).
\]
\end{theorem}
Analogous results hold, for instance, when the square root of the laplacian is replaced with the spectral fractional laplacian with homogeneous Dirichlet boundary conditions on bounded domains (see \cite{ct}). Moreover, note that $L^\infty$ bounds can be derived from $H^{1/2}$ ones, once suitable restrictions are imposed on the growth rate (subcritical) of the nonlinearities and/or on the dimension $N$, by means of a Brezis-Kato type argument.

In order to pursue the program just illustrated, compared with the case of the standard laplacian,
a number of new difficulties has to be overcome. For instance, the polynomial decay of the
fundamental solution of $(-\Delta)^{1/2}+1$ already affects the rate of segregation. Furthermore,
since such segregation occurs only in  the $N$-dimensional space, it is natural to expect free
boundaries of codimension $2$. But, perhaps,  the most  challenging issue lies in the lack of the
validity of an exact Alt-Caffarelli-Friedman monotonicity formula. This reflects, at the spectral
level, the lack of convexity of the eigenvalues with respect to domain variations, see Remark
\ref{rem: non convexity eigenv} below. To attack these problems new tools are in order, involving
different extremality conditions and new monotonicity formulas (associated with trace spectral
problems).

Let us finally mention that general fractional laplacians arise in many models of enhanced anomalous
diffusion; such operators  are of real interest both in population dynamics and in relativistic
quantum electrodynamics. This strongly motivates the extension of the theory in this direction, for
any $s\in(0,1)$.

The paper is organized as follows:

{\tableofcontents}

\subsection{Notation}

Throughout the paper, we will agree that any $X\in\R^{N+1}$ can be written as $X=(x,y)$, with
$x\in\R^N$ and $y\in\R$, in such a way that $\R^{N+1}_+:=\R^{N+1}\cap\{y>0\}$. For any $D\subset
\R^{N+1}$ we write
\[
\begin{split}
D^+&:= D\cap\{y>0\},\\
\partial^+D^+&:=\partial D \cap\{y>0\},\\
\partial^0D^+ &:=D\cap\{y=0\}.
\end{split}
\]
In most cases, we use this notation with $D=B_r(x_0,0)$ (the $(N+1)$-dimensional ball centered at a point of $\R^N$). In such case, we denote
\[
S^{N-1}_r(x_0,0):=\{(x,0): x \in \R^{N},\, |x-x_0| = r\} = \partial B_r^+\setminus\left(\partial^+ B_r^+
\cup\partial^0 B_r^+\right).
\]
Beyond the usual functional spaces, we will write
\[
H^1_{\loc}\left(\overline{\R^{N+1}_+}\right) := \{v :
\forall D \subset \R^{N+1} \text{ open and bounded}, v|_{D^+}
\in H^1(D^+)\}.
\]
Finally, we write $B^+$ for $B_1^+$, and we denote with $C$ any constant we need not to specify (possibly assuming different values even in the same expression).

\section{Alt-Caffarelli-Friedman type monotonicity formulae}\label{sec:acf}

This section is devoted to the proof of some monotonicity formulae of
Alt-Caffarelli-Friedman (ACF) type.
\subsection{Segregated ACF formula}
The validity of ACF type formulae depends on optimal partition problems involving
spectral properties of the domain. In the present situation, the spectral problem we
consider involves a pair of functions defined on $\S^{N}_+ := \partial^+ B^+$.
As a peculiar fact, here such functions have not disjoint support
on the whole $\S^N_+$, but only on its boundary $\S^{N-1}$. In this way we are lead to
consider the following optimal partition problem on $\S^{N-1}$.
\begin{definition}\label{def: lambda}
For each open subset $\omega$ of $\S^{N-1}:=\partial\S^N_+$ we define the first eigenvalue
associated to $\omega$ as
\[
    \lambda_1(\omega) :=
    \inf\left\{
    \frac{\int_{\S^{N}_+} |\nabla_{T} u|^2 \, \de{\sigma}
    }{\int_{\S^{N}_+}  u^2\, \de{\sigma}} :
    u \in H^1(\S^{N}_+), \, u\equiv0 \text{ on }\S^{N-1}\setminus\omega \right\}.
\]
Here $\nabla_{T} u$ stands for the (tangential) gradient of $u$ on $\S^{N}_+$.
\end{definition}

\begin{definition}\label{def: nu}
On $\S^{N-1}$ we define the set of 2-partition $\mathcal{P}^2$ by
\[
    \mathcal{P}^2 := \left\{(\omega_1, \omega_2)\from \omega_i \subset \S^{N-1}
\text{ open}, \, \omega_1 \cap \omega_2 = \emptyset \right\}
\]
and the number, only depending on $N$,
\[
\begin{split}
    \nuACF :&= \frac{1}{2} \inf_{(\omega_1, \omega_2) \in \mathcal{P}^2}
\sum_{i=1}^{2} \left( \sqrt{ \left(\frac{N-1}{2} \right)^2 + \lambda_1(\omega_i)
} - \frac{N-1}{2} \right)\\ & = \frac{1}{2} \inf_{(\omega_1, \omega_2) \in
\mathcal{P}^2} \sum_{i=1}^{2} \gamma({\lambda_1(\omega_i)}).
\end{split}
\]
\end{definition}
\begin{remark}\label{rem: gamma lambda}
As it is well known, $u$ achieves $\lambda_1(\omega)$ if and only if it is one signed, and its $ \gamma (\lambda_1(\omega))$-homogeneous extension to $\R^{N+1}_+$ is harmonic.
\end{remark}
\begin{remark}\label{rem: non convexity eigenv}
By symmetrization arguments, one may try to restrict the study of the above optimal partition problem to the case when both $\omega_i$ are spherical caps. In such a situation, writing $\Gamma(\vartheta):=\gamma(\lambda_1(\omega_\vartheta))$ for the spherical cap $\omega_\vartheta$ with opening $\vartheta$, one is lead to minimize the quantity
\[
\varphi(\vartheta):=\frac12\left[\Gamma(\vartheta)+\Gamma(\pi-\vartheta) \right],
\qquad \vartheta\in[0,\pi].
\]
It is worthwhile noticing that the function $\varphi$ is not convex, indeed one can prove
that
\[
\varphi(0)=\varphi\left(\frac{\pi}{2}\right)=\varphi(\pi)=\frac12
\]
(for details, see the proofs of Lemma \ref{lem: nuacf>0} and Proposition \ref{prp: ACF sym} below).
Thus, in particular, it is not clear whether the minimum of $\varphi$ may be strictly less that $1/2$.
As already mentioned, this marks a notable difference with respect to the standard diffusion case.
\end{remark}
\begin{lemma}\label{lem: nuacf>0}
For every dimension $N$, it holds $0<\nuACF\leq \frac12$.
\end{lemma}
\begin{proof}
The bound from above easily follows by comparing with the value corresponding
to the partition $(\S^{N-1},\emptyset)$: indeed, it holds $\lambda_1(\S^{N-1})=0$,
achieved by $u(x,y)\equiv1$, and $\lambda_1(\emptyset)=2N$,
achieved by $u(x,y)=y$. In order to prove the estimate from below,
let us first observe that, for each pair $(\omega_1,\omega_2) \in
\mathcal{P}^2$, there exist two functions $u_1$ and $u_2$ in $H^1(\S^{N}_+)$
such that $u_i\equiv 0$ on $\S^{N-1}\setminus\omega_i$,
\[
    \lambda_1(\omega_i) = \int\limits_{\S^{N}_+} |\nabla_{T} u_i|^2 \, \de{\sigma}
\quad \text{and} \quad \int\limits_{\S^{N}_+}  u_i^2 \, \de{\sigma} = 1.
\]
This claim is a consequence of the compactness both of the embedding
$H^1(\S^N_+)\hookrightarrow L^2(\S^N_+)$ and of the trace operator from
$H^1(\S^N_+)$ to $L^2(\S^{N-1})$ (recall that the constraint is continuous with
respect to the $L^2(\S^{N-1})$ topology).

We proceed by contradiction, supposing that there exists a sequence of
2-partition $(\omega_1^n, \omega_2^n) \in \mathcal{P}^2$ such that
\[
    \gamma\left( {\lambda_1(\omega_1^n)}\right) +
\gamma\left( {\lambda_1(\omega_2^n)}\right) \to 0.
\]
Since the function $\gamma$ is non negative and increasing, it must be that
$\lambda_1(\omega_i^n) \to 0$ for $i =1,2$, that is, there exist two sequences
of functions $u_1^n$ and $u_2^n$ in $H^1(\S^{N}_+)$ such that
$u_i\equiv 0$ on $\S^{N-1}\setminus\omega_i$,
\[
    \int\limits_{\S^{N}_+} |\nabla_{T} u_i|^2 \, \de{\sigma} \to 0  \quad \text{while}
\quad \int\limits_{\S^{N}_+}  u_i^2 \, \de{\sigma} = 1.
\]
Therefore, up to a subsequence, it holds
\[
    u_1^n, u_2^n \rightharpoonup |\S^{N}_+|^{1/2} \text{ in } H^1(\S^{N}_+)
\quad \text{and} \quad \int\limits_{\S^{N-1}}  u_1^n u_2^n \, \de{\sigma} = 0
\]
which are incompatible.
\end{proof}
Under the previous notations, we can prove the following monotonicity formula.
\begin{theorem}\label{thm: ACF}
Let $v_1, v_2 \in H^1(B_R^+(x_0,0))$ be continuous functions such that
\begin{itemize}
   \item  $v_1 v_2|_{\{y=0\}} = 0$, $v_i(x_0,0) = 0$;
   \item  for every non negative $\phi \in \C_0^{\infty}(B_R(x_0,0))$,
   \[
   \int\limits_{\R^{N+1}_+}(-\Delta v_i)v_i\phi \, \de x \de y + \int\limits_{\R^N}(\partial_\nu v_i)v_i\phi \,
   \de x = \int\limits_{\R^{N+1}_+}\nabla v_i\cdot\nabla(v_i\phi) \, \de x \de y \leq0.
   \]
\end{itemize}
Then the function
\[
    \Phi(r) := \prod_{i=1}^{2} \frac{1}{r^{2\nuACF}} \int\limits_{B_r^+(x_0,0)}
\frac{|\nabla v_i|^2}{|X-(x_0,0)|^{N-1} } \,\de{x} \de{y}
\]
is monotone non decreasing in $r$ for $r \in (0,R)$.
\end{theorem}
\begin{remark}\label{rem:chang_sign_ACF_segr}
Since
\begin{equation}\label{eqn: acf_per_moduli}
\int\limits_{\R^{N+1}_+}\nabla v_i\cdot\nabla(v_i\phi) \, \de x \de y =
\int\limits_{\R^{N+1}_+}\left[|\nabla v_i|^2\phi + \frac12\nabla (v_i)^2\cdot\nabla\phi\right] \de x \de y,
\end{equation}
we have that if $v_1,v_2$ satisfy the assumptions of Theorem \ref{thm: ACF} then
also $|v_1|,|v_2|$ do.
\end{remark}
By the above remark, we can assume without loss of generality that $v_1$ and $v_2$ are non negative.
Since the theorem is trivial if either $v_1\equiv 0$ or $v_2 \equiv 0$, we will
prove it when both $v_1$ and $v_2$ are non zero. Moreover, by translating and
scaling, the theorem can be proved under the assumption that $x_0 = 0$ and $R=
1$. We will need the following technical lemmas.
\begin{definition}\label{def:Gamma_1}
We define ${\Gamma_1} \in \C^1(\R^{N+1}_+; \R^+)$ as
\[
    {\Gamma_1}(X) :=
    \begin{cases}
        \frac{1}{|X|^{N-1}} & |X|\geq 1\\
        \frac{N+1}{2} - \frac{N-1}{2} |X|^2 & |X| < 1.
    \end{cases}
\]
We let also $\Gamma_{\eps}(X) = {\Gamma_1}(X/\eps) \eps^{1-N}$, so that
$\Gamma_{\eps} \nearrow \Gamma = |X|^{1-N}$, a multiple of the fundamental
solution of the half-laplacian, as $\eps \to 0$.
\end{definition}

\begin{remark}
Let us observe that each ${\Gamma_\eps}$ is radial and, in
particular, $\partial_{\nu} {\Gamma_\eps} = 0$ on $\R^N$. Moreover, they are
superharmonic on $\R^{N+1}_+$.
\end{remark}

\begin{lemma}\label{lem: phi is well defined}
Let $v_1, v_2$ be as in Theorem \ref{thm: ACF}. The
function
\begin{equation}\label{eqn: improper integral}
   r\mapsto\int\limits_{B_r^+} \frac{|\nabla v_i|^2}{|X|^{N-1}}
\,\de{x} \de{y}
\end{equation}
is well defined and bounded in any compact subset of $(0,1)$.
\end{lemma}
\begin{proof}
We proceed as follows: let $\eps > 0$, $\delta>0$ and let 
$\eta_\delta\in \C^\infty_0 (B_{r+\delta})$ be a smooth, radial cutoff function such that
$0\leq\eta_\delta\leq1$ and $\eta_\delta=1$ on $B_r$.
Choosing $\phi=\eta_\delta \Gamma_{\eps}$ in the second assumption of the theorem,
and recalling equation \eqref{eqn: acf_per_moduli}, we obtain
\begin{multline*}
\int\limits_{\R^{N+1}_+}\left[|\nabla v_i|^2\Gamma_{\eps} +
\frac12\nabla (v_i)^2\cdot\nabla\Gamma_{\eps}\right]\eta_\delta
\de x \de y
\leq -\int\limits_{\R^{N+1}_+}\frac12\Gamma_{\eps}\nabla (v_i)^2\cdot\nabla\eta_\delta \de x \de y\\
= \int\limits_{r}^{r+\delta} \left[-\eta_\delta'(\rho)\int\limits_{\partial^+B^+_\rho}\Gamma_{\eps}v_i\nabla v_i
\cdot \frac{X}{|X|}\de\sigma \right]\de\rho.
\end{multline*}
Passing to the limit as $\delta\to0$ we obtain, for almost every $r\in(0,1)$,
\[
\int\limits_{B^+_r}\left[|\nabla v_i|^2\Gamma_{\eps} +
\frac12\nabla (v_i)^2\cdot\nabla\Gamma_{\eps}\right]\de x \de y
\leq
\int\limits_{\partial^+B^+_r}\Gamma_{\eps}v_i\partial_\nu v_i\de\sigma,
\]
which, combined with the inequality $-\Delta \Gamma_{\eps} \geq 0$ tested with
$v_i^2/2$ leads to
\[
    \int\limits_{B_r^+} |\nabla v_i|^ 2 \Gamma_{\eps} \, \de{x}\de{y} \leq
\int\limits_{\partial^+ B_r^+}\left(\Gamma_{\eps} v_i {\partial_\nu v_i }
 - \frac{v_i^2}{2} {\partial_\nu} \Gamma_{\eps}
\right)\, \de{\sigma}.
\]
Letting $\eps \to 0^+$, by monotone convergence we infer
\begin{equation}\label{eqn: estim from above denom}
    \int\limits_{B_r^+} \frac{|\nabla v_i|^ 2}{|X|^{N-1}} \, \de{x}\de{y} \leq
\frac{1}{r^{N-1}} \int\limits_{\partial^+ B_r^+} v_i \frac{\partial v_i
}{\partial \nu} \, \de{\sigma} + \frac{N-1}{2r^{N}} \int\limits_{\partial^+
B_r^+} v_i^2 \, \de{\sigma}
\end{equation}
and this, in turns, proves the lemma.
\end{proof}

\begin{lemma}\label{lem: bound from below ACF}
Let $v_1, v_2$ be two non trivial functions satisfying the assumptions of
Theorem \ref{thm: ACF}. It holds
\begin{equation}\label{eqn: bound from below nu}
    \sum_{i=1}^2 \ddfrac{\int\limits_{\partial^+ B_r^+}\tfrac{|\nabla
v_i|^2}{|X|^{N-1}} \,\de{\sigma} }{\int\limits_{B_r^+} \tfrac{|\nabla
v_i|^2}{|X|^{N-1}} \,\de{x} \de{y}} \geq \frac{4}{r}\nuACF.
\end{equation}
\end{lemma}
\begin{proof}
First we use the estimate \eqref{eqn: estim from above denom} to bound from below the left hand side of
\eqref{eqn: bound from below nu}:
\begin{multline*}
    \dfrac{\int\limits_{\partial^+ B_r^+}\frac{|\nabla v_i|^2}{|X|^{N-1}
}\,\de{\sigma} }{\int\limits_{B_r^+} \frac{|\nabla v_i|^2 }{|X|^{N-1}} \,\de{x}
\de{y}} \geq  \dfrac{\int\limits_{\partial^+ B_r^+}|\nabla v_i|^2 \,\de{\sigma}
}{\int\limits_{\partial^+ B_r^+} v_i \partial_{\nu} v_i \, \de{\sigma} + (N-1)
\frac{r}{2} \int\limits_{\partial^+ B_r^+} v_i^2 \, \de{\sigma} }\\ =  \frac{1}{r}
\dfrac{\int\limits_{\S^{N}_+}|\nabla v_i^{(r)}|^2 \,\de{\sigma} }{\int\limits_{\S^{N}_+}
v_i^{(r)} \partial_{\nu} v_i^{(r)} \, \de{\sigma} + \frac{N-1}{2}
\int\limits_{\S^{N}_+} (v_i^{(r)})^2 \, \de{\sigma} },
\end{multline*}
where $v_i^{(r)}\from \S^{N-1}_+ \to \R$ is defined as $v_i^{(r)}(\xi) =
v_i(r\xi)$. We now
estimate the right hand side as follows: the numerator writes
\[
    \begin{split}
    \int\limits_{\S^{N}_+} |\nabla v_i^{(r)}|^2 \, \de{\sigma}  &= \int\limits_{\S^{N}_+}
|\partial_{\nu} v_i^{(r)}|^2 \, \de{\sigma}  + \int\limits_{\S^{N}_+} |\nabla_T
v_i^{(r)}|^2 \, \de{\sigma}  \\
    &= \int\limits_{\S^{N}_+} |v_i^{(r)}|^2 \, \de{\sigma}  \left( \underbrace{\frac{
\int\limits_{\S^{N}_+} |\partial_{\nu} v_i^{(r)}|^2 \, \de{\sigma} }{\int\limits_{\S^{N}_+}
|v_i^{(r)}|^2 \, \de{\sigma} } }_{t^2} + \underbrace{\frac{\int\limits_{\S^{N}_+}
|\nabla_T v_i^{(r)}|^2 \, \de{\sigma} }{\int\limits_{\S^{N}_+} |v_i^{(r)}|^2 \,
\de{\sigma}} }_{\mathcal{R}} \right).
    \end{split}
\]
where $\mathcal{R}$ stands for the Rayleigh quotient of $v_i^{(r)}$ on
$\S^{N}_+$. On the other hand, by the Cauchy-Schwarz inequality, the denominator
may be estimated from above by
\begin{multline*}
    \int\limits_{\S^{N}_+} v_i^{(r)} \partial_{\nu} v_i^{(r)} \, \de{\sigma} +
r\frac{N-1}{2} \int\limits_{\S^{N}_+} |v_i^{(r)}|^2 \, \de{\sigma} \\
\leq
\left(\int\limits_{\S^{N}_+} |v_i^{(r)}|^2 \, \de{\sigma}\right)^{1/2}
\left(\int\limits_{\S^{N}_+} \partial_{\nu} v_i^{(r)} \, \de{\sigma}\right)^{1/2} +
r\frac{N-1}{2} \int\limits_{\S^{N}_+} |v_i^{(r)}|^2 \, \de{\sigma}\\
    \leq\int\limits_{\S^{N}_+} |v_i^{(r)}|^2 \, \de{\sigma} \left[
\underbrace{\left(\frac{ \int\limits_{\S^{N}_+} |\partial_{\nu} v_i^{(r)}|^2 \,
\de{\sigma} }{\int\limits_{\S^{N}_+} |v_i^{(r)}|^2 \, \de{\sigma} }\right)^{1/2}}_{t} +
\frac{N-1}{2} \right].
\end{multline*}
As a consequence
\[
    \dfrac{\int\limits_{\partial^+ B_r^+}\frac{|\nabla v_i|^2}{|X|^{N-1}
}\,\de{\sigma} }{\int\limits_{B_r^+} \frac{|\nabla v_i|^2 }{|X|^{N-1}} \,\de{x}
\de{y}} \geq \frac{1}{r} \min_{t \in \R^+} \frac{\mathcal{R}+t^2}{ t +
\frac{N-1}{2}}.
\]
A simple computation shows that the minimum is achieved when
\[
    t = \gamma( {\mathcal{R}}) = \sqrt{ \left(\frac{N-1}{2}\right)^2 +
\mathcal{R} } - \frac{N-1}{2},
\]
and it is equal to $2\gamma( {\mathcal{R}})$. Summing over $i = 1, 2$, we obtain
\[
   \sum_{i=1}^2 \ddfrac{\int\limits_{\partial^+ B_r^+}\tfrac{|\nabla v_i|^2}{|X|^{N-1}}
\,\de{\sigma} }{\int\limits_{B_r^+} \tfrac{|\nabla v_i|^2}{|X|^{N-1}} \,\de{x} \de{y}}
\geq \frac{2}{r} \inf_{(\omega_1, \omega_2) \in \mathcal{P}^2} \sum_{i=1}^{2}
\gamma\left(\lambda_1(\omega_i) \right)
= \frac{4}{r}\nuACF
\]
where the inequality follows by substituting each $\mathcal{R}$ with their
optimal value, that is, the eigenvalue $\lambda_1(\omega_i)$.
\end{proof}

\begin{proof}[Proof of Theorem \ref{thm: ACF}]
As already noticed, we may assume that $x_0 = 0$ and $R =1$ and that both $v_1$
and $v_2$ are non trivial and non negative. We start observing that the function $\Phi(r)$ is
positive and absolutely continuous for $r \in (0,1)$, since it is the product of
functions which are positive and absolutely continuous in $(0,1)$. Therefore,
the theorem follows once we prove that $\Phi'(r) \geq 0$ for almost every $r \in
(0,1)$. A direct computation of the logarithmic derivative of $\Phi$ shows that
\[
    \frac{\Phi'(r)}{\Phi(r)} = -\frac{4\nuACF}{r} + \sum_{i=1}^2
\dfrac{\int\limits_{\partial^+ B_r^+}|\nabla v_i|^2 /|X|^{N-1} \,\de{\sigma}
}{\int\limits_{B_r^+} |\nabla v_i|^2 / |X|^{N-1} \,\de{x} \de{y}} \geq 0
\]
where the last inequality follows by Lemma \ref{lem: bound from below ACF}.
\end{proof}

As we mentioned, Theorem \ref{thm: ACF} will be crucial in proving interior regularity
estimates. We now provide a related result, suitable to treat regularity up to the boundary.
Differently from before, in this case we can show that the optimal exponent in the corresponding monotonicity formula is exactly $\gamma=1/2$.
\begin{proposition}\label{prp: ACF sym}
Let $v \in H^1(B_R^+)$ be a continuous function such that
\begin{itemize}
   \item  $v_1(x,0) = 0$ for $x_1 \leq 0$;
   \item  for every non negative $\phi \in \C_0^{\infty}(B_R)$,
   \[
   \int\limits_{\R^{N+1}_+}(-\Delta v)v\phi \, \de x \de y + \int\limits_{\R^N}(\partial_\nu v)v\phi \,
   \de x = \int\limits_{\R^{N+1}_+}\nabla v\cdot\nabla(v\phi) \, \de x \de y \leq0.
   \]
\end{itemize}
Then the function
\[
    \Phi(r) := \frac{1}{r} \int\limits_{B_r^+}
\frac{|\nabla v|^2}{|X|^{N-1} } \,\de{x} \de{y}
\]
is monotone non decreasing in $r$ for $r \in (0,R)$.
\end{proposition}
\begin{proof}
Let $\bar\omega:=\S^{N-1}\cap\{x_1>0\}$, and let $v$ denote the
$1/2$ homogeneous, harmonic extension of $v(x,0)=\sqrt{x_1^+}$ to $\R^{N+1}_+$, that is
\[
v(x,y)= \sqrt{\frac{\sqrt{x_1^2+y^2}+x_1}{2}}.
\]
Since $v$ is positive for $y>0$, Remark \ref{rem: gamma lambda} implies that $v|_{\S^N_+}$ is an eigenfunction associated to $\lambda_1(\bar\omega)$, providing
\[
\gamma(\lambda_1(\bar\omega))=\frac12.
\]
But then, reasoning as in the proofs of  Lemma \ref{lem: bound from below ACF} and
Theorem \ref{thm: ACF}, we readily obtain that
\[
\frac{\Phi'(r)}{\Phi(r)}\geq \frac{2}{r}\left[-\frac12+\gamma(\lambda_1(\bar\omega))\right]
=0. \qedhere
\]
\end{proof}

\subsection{Perturbed ACF formula}

We now move from Theorem \ref{thm: ACF} and introduce a perturbed version of the
monotonicity formula, suitable for functions which coexist on the boundary, rather
than having disjoint support.
\begin{theorem}\label{thm: ACF perturbed}
Let $\nuACF$ be as in Definition \ref{def: nu}, and let $v_1,v_2 \in H^1_{\loc}\left(\overline{\R^{N+1}_+}\right)$ be continuous functions such that, for every non negative $\phi \in \C^{\infty}_0\left(\overline{\R^{N+1}_+}\right)$ and $j\neq i $,
\begin{multline*}
\int\limits_{\R^{N+1}_+}(-\Delta v_i)v_i\phi \, \de x \de y + \int\limits_{\R^N}(\partial_\nu v_i + v_iv_j^2)
v_i\phi \,\de x \\
= \int\limits_{\R^{N+1}_+}\nabla v_i\cdot\nabla(v_i\phi) \, \de x \de y + \int\limits_{\R^N}v_i^2v_j^2\phi \, \de x  \leq0.
\end{multline*}
For any $\nu' \in (0, \nuACF)$ there exists $\bar{r} >1$ such that the function
\[
    \Phi(r) := \prod_{i=1}^{2} \Phi_i(r)
\]
is monotone non decreasing in $r$ for $r \in (\bar{r}, \infty)$, where
\[
    \Phi_i(r) := \frac{1}{r^{2\nu'}}\left(\int\limits_{B_r^+} |\nabla v_i|^2 {\Gamma_1}
\,\de{x} \de{y} +  \int\limits_{\partial^0 B_r^+} v_i^2 v_j^2 {\Gamma_1} \,\de{x}
\right), \quad \text{ for } j \neq i.
\]
\end{theorem}
\begin{remark}\label{rem:chang_sign_ACF_pert}
We observe that, analogously to Remark \ref{rem:chang_sign_ACF_segr}, the main assumption of
Theorem \ref{thm: ACF perturbed} can be equivalently rewritten as
\begin{equation*}
\int\limits_{\R^{N+1}_+}\left[|\nabla v_i|^2\phi + \frac12\nabla (v_i)^2\cdot\nabla\phi\right] \de x \de y + \int\limits_{\R^N}v_i^2v_j^2\phi \, \de x \leq0,
\end{equation*}
for every compactly supported $\phi\geq0$. In particular, if $v_1, v_2$ satisfy such assumption, so $|v_1|,|v_2|$ do. Moreover, reasoning as in Lemma \ref{lem: phi is well defined}, we obtain that,
for every $\phi\geq0$ and almost every $r$,
\begin{equation}\label{eqn: acf_perturbed per_moduli}
\int\limits_{B^+_r}\left[|\nabla v_i|^2\phi + \frac12\nabla (v_i)^2\cdot\nabla\phi\right]\, \de x \de y + \int\limits_{\partial^0B^+_r}v_i^2v_j^2\phi \, \de x \leq \int\limits_{\partial^+B^+_r}(\partial_\nu v_i )v_i\phi
\, \de \sigma,
\end{equation}
\end{remark}
The proof of Theorem \ref{thm: ACF perturbed} follows the lines of the one of
Theorem \ref{thm: ACF}.
\begin{lemma}\label{lem: bound from below ACF perturbed}
Let $v_1, v_2$ be two non trivial functions satisfying the assumptions of
Theorem \ref{thm: ACF perturbed}. Then, for any $r > 1$, it holds
\begin{equation}\label{eqn: bound from below perturbed}
    \sum_{i=1}^2 \ddfrac{\int\limits_{\partial B_r^+} |\nabla v_i|^2 {\Gamma_1}
\,\de{\sigma} +  \int\limits_{r \S^{N-1} } v_i^2 v_j^2 {\Gamma_1} \,\de{\sigma}
}{\int\limits_{B_r^+} |\nabla v_i|^2 {\Gamma_1} \,\de{x} \de{y} +  \int\limits_{\partial^0
B_r^+} v_i^2 v_j^2 {\Gamma_1} \,\de{x} } \geq \frac{2}{r} \sum_{i=1}^{2}
\gamma\left(  {\Lambda_i(r) } \right),
\end{equation}
where
\[
    \Lambda_i(r) = \frac{\int\limits_{\S^{N}_+} |\nabla_T v_i^{(r)}|^2 \, \de{\sigma} +
r\int\limits_{\S^{N-1} } (v_i^{(r)} v_j^{(r)})^2 \,\de{\sigma}  }{\int\limits_{\S^{N}_+}
|v_i^{(r)}|^2 \, \de{\sigma}}
\]
(again, $v_i^{(r)}\from \S^{N-1}_+ \to \R$ is such that $v_i^{(r)}(\xi) = v_i(r\xi)$). 
\end{lemma}
\begin{proof}
By choosing $\phi=\Gamma_1$ (Definition \ref{def:Gamma_1}) in equation \eqref{eqn: acf_perturbed per_moduli} we obtain, for a.e. $r > 0$,
\[
    \int\limits_{B_r^+} \left[ |\nabla v_i|^2 {\Gamma_1} + \frac12 \nabla (v_i)^2 \cdot \nabla \Gamma_1 \right] \,\de{x} \de{y} +  \int\limits_{\partial^0
B_r^+} v_i^2 v_j^2 {\Gamma_1} \,\de{x} \leq \int\limits_{\partial^+ B_r^+} v_i {\partial_\nu}v_i {\Gamma_1} \, \de{\sigma}.
\]
The superharmonicity of $\Gamma_1$ yields then
\[
    \int\limits_{B_r^+} |\nabla v_i|^2 {\Gamma_1} \,\de{x} \de{y} +  \int\limits_{\partial^0
B_r^+} v_i^2 v_j^2 {\Gamma_1} \,\de{x} \leq \int\limits_{\partial^+ B_r^+}\left(
v_i {\partial_\nu}v_i {\Gamma_1} - \frac{v_i^2}{2}
{\partial_\nu} {\Gamma_1} \right)\, \de{\sigma}.
\]
Recalling that $r>1$ we can use the previous estimate to bound from below the
left hand side of equation \eqref{eqn: bound from below perturbed}, obtaining
\[
    \ddfrac{\int\limits_{\partial B_r^+} |\nabla v_i|^2 {\Gamma_1} \,\de{\sigma} +
\int\limits_{r \S^{N-1} } v_i v_j^2 {\Gamma_1} \,\de{\sigma} }{\int\limits_{B_r^+} |\nabla
v_i|^2 {\Gamma_1} \,\de{x} \de{y} + \int\limits_{\partial^0 B_r^+} v_i v_j^2 {\Gamma_1}
\,\de{x} }  \geq  \frac{1}{r} \dfrac{\int\limits_{\S^{N}_+} |\nabla v_i^{(r)}|^2
\,\de{\sigma} +  r\int\limits_{\S^{N-1} } (v_i^{(r)} v_j^{(r)})^2 \,\de{\sigma}
}{\int\limits_{\S^{N}_+} v_i^{(r)} \partial_{\nu} v_i^{(r)} \, \de{\sigma} +
\frac{N-1}{2} \int\limits_{\S^{N}_+} (v_i^{(r)})^2 \, \de{\sigma} }.
\]
We now estimate the right hand side as follows: the numerator writes
\begin{multline*}
    \int\limits_{\S^{N}_+} |\nabla v_i^{(r)}|^2 \, \de{\sigma} + r\int\limits_{\S^{N-1} }
(v_i^{(r)} v_j^{(r)})^2 \,\de{\sigma} \\ = \int\limits_{\S^{N}_+} |\partial_{\nu}
v_i^{(r)}|^2 \, \de{\sigma}  + \int\limits_{\S^{N}_+} |\nabla_T v_i^{(r)}|^2 \,
\de{\sigma} + r\int\limits_{\S^{N-1} } (v_i^{(r)} v_j^{(r)})^2 \,\de{\sigma}   \\
    = \int\limits_{\S^{N}_+} |v_i^{(r)}|^2 \, \de{\sigma}  \left( \underbrace{\frac{
\int\limits_{\S^{N}_+} |\partial_{\nu} v_i^{(r)}|^2 \, \de{\sigma} }{\int\limits_{\S^{N}_+}
|v_i^{(r)}|^2 \, \de{\sigma} } }_{t^2} + \underbrace{\frac{\int\limits_{\S^{N}_+}
|\nabla_T v_i^{(r)}|^2 \, \de{\sigma} + r\int\limits_{\S^{N-1} } (v_i^{(r)}
v_j^{(r)})^2 \,\de{\sigma}  }{\int\limits_{\S^{N}_+} |v_i^{(r)}|^2 \, \de{\sigma}}
}_{\mathcal{R}} \right).
\end{multline*}
We may bound the denominator as in Lemma \eqref{lem: bound from below ACF}. As a
consequence
\[
    \ddfrac{\int\limits_{\partial B_r^+} |\nabla v_i|^2 {\Gamma_1}\,\de{\sigma} +
\int\limits_{r \S^{N-1} } v_i^2 v_j^2 {\Gamma_1} \,\de{\sigma} }{\int\limits_{B_r^+} |\nabla
v_i|^2 {\Gamma_1}  \,\de{x} \de{y} + \int\limits_{\partial^0 B_r^+} v_i^2 v_j^2
{\Gamma_1} \,\de{x} } \geq \frac{1}{r} \min_{t \in \R^+} \frac{\mathcal{R}+t^2}{
t + \frac{N-1}{2}}.
\]
Minimizing with respect to $t$ as in Lemma \eqref{lem: bound from below ACF} and
summing over $i = 1, 2$, we obtain equation \eqref{eqn: bound from below
perturbed}.
\end{proof}

\begin{proof}[Proof of Theorem \ref{thm: ACF perturbed}]
Without loss of generality, we assume that both $v_1$ and $v_2$ are non trivial.
As in Theorem \ref{thm: ACF}, we will prove that the logarithmic
derivative of $\Phi$ is non negative for any $\nu' \in (0,\nuACF)$ and $r$
sufficiently large. Again, a direct computation shows that
\[
    \begin{split}
    \frac{\Phi'(r)}{\Phi(r)} &= -\frac{4\nu'}{r} + \sum_{i=1}^2
\ddfrac{\int\limits_{\partial B_r^+} |\nabla v_i|^2 {\Gamma_1}  \,\de{\sigma} +
\int\limits_{r \S^{N-1} } v_i^2 v_j^2 {\Gamma_1} \,\de{\sigma} }{\int\limits_{B_r^+} |\nabla
v_i|^2 {\Gamma_1}  \,\de{x} \de{y} +  \int\limits_{\partial^0 B_r^+} v_i^2 v_j^2
{\Gamma_1} \,\de{x} }\\
    &\geq \frac{4}{r} \left[ -\nu' +\frac12 \sum_{i=1}^{2}
    \gamma \left(  \Lambda(v_i^{(r)}) \right)
\right]
    \end{split}
\]
and thus it is sufficient to prove that there exists $\bar{r} > 1$ such that, for every
$r > \bar{r}$, the last term is nonnegative. Of course if $\Lambda_i(r)\to +\infty$ for
some $i$ then there is nothing to prove; thus we can suppose that each $\Lambda_i(r)$ is
bounded uniformly. To begin with, we see that, for $r$ large,
\begin{equation}\label{eqn: bound of L2 from below}
    H(r):=\|v_{i}^{(r_n)}\|_{L^2(\S^N_+)}^2 = \int\limits_{\S^N_+} (v_i^{(r)})^2 \de{\sigma}
\geq C > 0.
\end{equation}
Indeed, the choice of $\phi\equiv1$ in equation \eqref{eqn: acf_perturbed per_moduli}
yields
\[
    H'(r) = \int\limits_{\S^N_+}  r \partial_{\nu} (v_i^2)(r\xi) \de{\sigma} \geq 0,
\]
and, since the functions are non trivial, $H$ cannot be identically 0.

Let us suppose by contradiction that there exists a sequence $r_n \rightarrow \infty$ such that
\begin{equation}\label{eqn: bound of character}
    \frac12\sum_{i=1}^{2} \gamma\left( \Lambda_i(r_n) \right) \leq \nu' < \nuACF.
\end{equation}
We introduce the renormalized sequence
\begin{equation*}
    w_{i,n} = \frac{ v_i^{(r_n)} }{ \left( \int\limits_{\S^N_+} (v_i^{(r_n)})^2
\de{\sigma} \right)^{1/2} }, \quad \text{so that} \quad
\|w_{i,n}\|_{L^2(\S^N_+)} = 1.
\end{equation*}
Recall that $\Lambda_i(r_n)$ is uniformly bounded, that is
\[
    K \geq \Lambda_i(r_n) = \int\limits_{\S^N_+}  |\nabla_{T} w_{i,n}|^2 \de{\sigma} +
\int\limits_{\S^{N-1}} r_n w_{i,n} ^2 w_{i,n}^2 \|v_{i}^{(r_n)}\|_{L^2(\S^N_+)}
\de{\sigma}
\]
and, together with \eqref{eqn: bound of L2 from below}, this yields
\begin{equation}\label{eqn: separated bounds}
    \int\limits_{\S^N_+} |\nabla_{T} w_{i,n}|^2 \, \de{\sigma} \leq K \quad \text{and}
\quad  \int\limits_{\S^{N-1}} w_{i,n} ^2 w_{i,n}^2 \, \de{\sigma} \leq \frac{1}{r_n}
K'.
\end{equation}
Hence there exist functions $\bar{w}_i \in H^1(\S^N_+)$ such that, up to
subsequences, $w_{i,n_k} {\rightharpoonup} \bar{w}_i$, weakly in $H^1(\S^N_+)$, with
$\|\bar{w}_i\|_{L^2(\S^N_+)} = 1$. Moreover, from the weak lower semi-continuity of the norm,
\begin{equation*}
    \liminf_{k \to \infty} \Lambda_i(r_{n_k}) \geq \int\limits_{\S^N_+} |\nabla_{T}
\bar{w}_{i}|^2 \de{\sigma}_N \geq \lambda_1(\{ \bar{w}_i|_{y=0} > 0\}).
\end{equation*}
From \eqref{eqn: separated bounds} we have that $w_i^{(r)} w_j^{(r)} \rightarrow
0$ a.e. on $\S^{N-1}$ and $\bar{w}_{i} \bar{w}_{j} = 0$ on $\S^{N-1}$. This
means that the limit configuration $(w_1,w_2)$ induces a partition of
$\S^{N}_+$, for which we have
\begin{equation*}
    \liminf_{k \rightarrow \infty} \frac12 \sum_{i=1}^{2}
\gamma\left( {\Lambda_i(r_{n_k})}\right) \geq \nuACF
\end{equation*}
in contradiction with \eqref{eqn: bound of character}.
\end{proof}

\section{Almgren type monotonicity formulae}\label{sec:almgren}

In the following, we will be concerned with a number of entire profiles, that is $k$-tuples
of functions defined on the whole $\R^{N+1}_+$, which will be obtained from
solutions to problem $\problem{\beta}$, by suitable limiting procedures.
This section is devoted to the proof of some monotonicity formulae
of Almgren type, related to such profiles.

\subsection{Almgren's formula for segregation entire profiles}\label{subsec: Alm seg prof}
To start with, we consider $k$-tuples $\bv$ having components with segregated traces on $\R^N$.
In such a situation, on one hand each component of $\bv$, when different from zero, satisfies a limiting version of $\problem{\beta}$, where the internal dynamics are trivialized; on the other hand, the interaction
between different components is now described by the validity of some Pohozaev type identity.
We recall that, in order to prove the Almgren formula, it is sufficient to require the Pohozaev
identity to hold only in spherical domains. Nonetheless, we prefer to assume its validity in the broader
class of cylindrical domains, that is domains which are products of spherical and cubic ones. This choice will be useful in classifying the possible limiting profiles,
when we will be involved in a procedure of dimensional reduction.

More precisely, let $C_{r,l}^+(x_0,0)\subset \R^{N+1}_+$ be any set such that there exists $h \in \N$, $h \leq N$,
and a decomposition $\R^{N+1}_+ = \R^{h+1}_+ \oplus \R^{N-h}$ such that, writing
\[
    \R^{N+1}_+\ni X = (x',x'',y), \text{ with } (x',y) \in \R^{h+1}_+, \, x'' \in \R^{N-h},
\]
it holds
\[
C_{r,l}^+(x_0,0) = B_r^+(x'_0,0) \times Q_l(x''_0).
\]
Here, $B_r^+ \subset \R^{h+1}_+$ denotes an half ball of radius $r$,
and $Q_l \subset \R^{N-h}$ a cube of edge length equal to $2l$.
\begin{definition}[Segregation entire profiles]\label{def:segr ent prof}
We denote with $\classG_s$ the set of functions $\bv \in H^1_\loc\left(\overline{\R^{N+1}_+}; \R^k\right)$,
$\bv=(v_1,\dots,v_k)$ continuous, which satisfy the following assumptions:
\begin{enumerate}
 \item $v_i v_j |_{y = 0} = 0$ for every $j\neq i$;
 \item for every $i$,
\begin{equation}\label{eqn: equation of classG_s}
    \begin{cases}
    - \Delta v_i = 0 & \text{ in } \R^{N+1}_+\\
    v_i \partial_{\nu} v_i = 0 & \text{ on } \R^N \times \{0\};
    \end{cases}
\end{equation}
\item for any $x_0\in\R^N$ and a.e. $r>0$, $l>0$,
\begin{multline}\label{eqn: Pohozaev for classGs}
        \int\limits_{C_{r,l}^+} \tsum_{i} 2 |\nabla_{(x',y)} v_i|^2 -(h+1) |\nabla v_i|^2 \,\de{x}\de{y}  +
    r\int\limits_{\partial^+B^+_r \times Q_l } \tsum_{i} |\nabla v_i|^2 \,\de{\sigma} + \\
     = 2r \int\limits_{\partial^+B^+_r \times Q_l }  \tsum_{i} |\partial_{\nu} v_i|^2 \,\de{\sigma}- 2 \int\limits_{B^+_r \times \partial^+Q_l }  \tsum_{i} \partial_{\nu}    v_i \nabla_{(x',y)} v_i \cdot (x'-x_0',y) \,\de{\sigma},
\end{multline}
where $\nabla_{(x',y)}$ is the gradient with respect to the directions in $\R^{h+1}_+$.
\end{enumerate}
\end{definition}
\begin{remark}
Let $\bv\in\classG_s$. By choosing $h=N$ in the above definition, we obtain that the
spherical Pohozaev identity holds, namely
\begin{equation}\label{eqn: Pohozaev spheres for classGs}
    (1-N) \int\limits_{ B^+_r} \tsum_{i} |\nabla v_i|^2 \, \de x \de y + r \int\limits_{\partial^+ B^+_r} \tsum_{i} |\nabla v_i|^2 \, \de \sigma = 2 r \int\limits_{\partial^+ B^+_r} \tsum_{i} |\partial_{\nu} v_i|^2 \, \de \sigma
\end{equation}
for a.e. $r>0$.
\end{remark}
Let us define, for every $x_0 \in \R^N$ and $r > 0$,
\[
    \begin{split}
    E(x_0, r) &:= \frac{1}{r^{N-1}} \int\limits_{B^+_r(x_0,0)} \tsum_{i}
|\nabla v_i|^2 \, \de{x} \de{y}\\
    H(x_0,r) &:= \frac{1}{r^{N}} \int\limits_{\partial^+ B^+_r(x_0,0)}
\tsum_{i} v_i^2 \, \de{\sigma}.
    \end{split}
\]
Let $x_0$ be fixed. Since $\bv \in H^1_{\loc}\left(\overline{\R^{N+1}_+}, \R^k\right)$, both $E$ and $H$ are locally absolutely continuous functions on $(0,+\infty)$, that is, both $E'$ and $H'$ are $L^1_{\loc}(0,\infty)$ (here, $'=\de/\de r$).
\begin{theorem}\label{thm:_Almgren_for_classG_s}
Let $\bv \in \classG_s$, $\bv \not\equiv 0$. For every $x_0 \in \R^N$ the function (Almgren frequency function)
\[
    N(x_0,r) := \frac{E(x_0,r)}{H(x_0, r)}
\]
is well defined on $(0,\infty)$, absolutely continuous, non decreasing, and it satisfies the identity
\begin{equation}\label{eqn: logarithmic derivative of H}
    \frac{\mathrm{d}}{\mathrm{d}r} \log H(r) = \frac{2N(r)}{r}.
\end{equation}
Moreover, if $N(r)\equiv \gamma$ on an open interval, then $N\equiv\gamma$ for every $r$, and
$\bv$ is a homogeneous function of degree $\gamma$.
\end{theorem}
\begin{proof}
Up to a translation, we may suppose that $x_0 = 0$. Obviously $H \geq 0$, and $H > 0$ on a nonempty interval
$(r_1,r_2)$, otherwise $\bv \equiv 0$. As a consequence, either $\bv$ is a nontrivial constant, and the theorem easily follows; or, by harmonicity, $\bv$ is not constant in the whole $B^+_{r_2}$, and also $E > 0$ for $r<r_2$. Passing to the logarithmic derivatives, the monotonicity of $N$ will be a consequence of the claim
\[
    \frac{N'(r)}{N(r)} = \frac{E'(r)}{E(r)} - \frac{H'(r)}{H(r)} \geq 0 \qquad \text{ for } r \in (r_1,r_2).
\]
Deriving $E$ and using the Pohozaev identity \eqref{eqn: Pohozaev spheres for classGs},  we have that
\[
\begin{split}
    E'(r) &=\frac{1-N}{r^N} \int\limits_{ B^+_r} \tsum_{i} |\nabla v_i|^2 \, \de{x} \de{y} + \frac{1}{r^{N-1}} \int\limits_{\partial^+ B^+_r} \tsum_{i} |\nabla v_i|^2 \, \de\sigma\\
     & =  \frac{2}{r^{N-1}} \int\limits_{\partial_+ B^+_r} \tsum_{i} |\partial_{\nu} v_i|^2 \, \de\sigma,
\end{split}
\]
while testing equation \eqref{eqn: equation of classG_s} with $v_i$ in $B_r^+$ and summing over $i$, we obtain
\[
    E(r) = \frac{1}{r^{N-1}} \int\limits_{B^+_r} \tsum_{i} |\nabla v_i|^2 \, \de x \de y = \frac{1}{r^{N-1}} \int\limits_{\partial^+ B^+_r} \tsum_{i} v_i \partial_{\nu} v_i \, \de \sigma.
\]
As far as $H$ is concerned, we find
\[
    H'(r) =  \frac{2}{r^N} \int\limits_{\partial^+ B_r^+} \tsum_{i}  v_i \partial_{\nu} v_i \, \de\sigma.
\]
As a consequence, by the Cauchy-Schwarz inequality, we have
\begin{equation}\label{eqn: Cauchy Schwarz for N}
      \frac12\,\frac{N'(r)}{N(r)} = \ddfrac{\int\limits_{\partial^+ B^+_r} \tsum_{i} |\partial_{\nu} v_i|^2 \, \de \sigma }{ \int\limits_{\partial^+ B^+_r} \tsum_{i} v_i \partial_{\nu} v_i \, \de \sigma } - \ddfrac{ \int\limits_{\partial^+ B_r^+}  \tsum_{i} v_i \partial_{\nu} v_i \, \de{\sigma} }{ \int\limits_{\partial^+ B_r^+} \tsum_{i} v_i^2 \, \de\sigma }\geq 0
      \qquad \text{ for } r \in (r_1,r_2).
\end{equation}
Moreover, on the same interval,
\[
    \frac{\mathrm{d}}{\mathrm{d}r} \log H(r) = \frac{H'(r)}{H(r)} =
\frac{2E(r)}{rH(r)} = \frac{2N(r)}{r}.
\]
Let us show that we can choose $r_1=0$, $r_2=+ \infty$. On one hand, the above equation provides that,
if $\log H(\bar r)>-\infty$, then $\log H(r)>-\infty$ for every $r>\bar  r$, so that $r_2=+\infty$. On
the other hand, let us assume by contradiction that
\[
    r_1 := \inf \{ r: H(r) > 0\text{ on }(r,+\infty)\} > 0.
\]
By monotonicity, we have that $N(r) < N(2r_1)$ for every $r_1 < r \leq 2r_1$. It follows that
\[
    \frac{\mathrm{d}}{\mathrm{d}r} \log H(r) \leq \frac{2N(2r_1)}{r} \implies \frac{H(2r_1)}{H(r)} \leq \left(\frac{2r_1}{r}\right)^{2N(2r_1)}
\]
and, since $H$ is continuous, $H(r_1) >0$, a contradiction.

Now, let us assume $N(r)\equiv \gamma$ on some interval $I$.
Recalling equation \eqref{eqn: Cauchy Schwarz for N}, we see that
\[
    \left( \int\limits_{\partial^+ B_r^+}  \tsum_{i} v_i \partial_{\nu} v_i \,
\de\sigma \right)^2 = \int\limits_{\partial^+ B_r^+} \tsum_{i} v_i^2 \, \de\sigma
\int\limits_{\partial^+ B^+_r} \tsum_{i} |\partial_{\nu} v_i|^2 \, \de\sigma,
\]
which is true, by the Cauchy-Schwarz inequality, if and only if $ \bv$ and $\partial_{\nu} \bv$ are parallel, that is
\[
    v_i = \lambda(r) \partial_{\nu} v_i = \frac{\lambda(r)}{r} X \cdot \nabla v_i,\quad\text{for every } r \in I.
\]
Using the definition of $N$, we have $\gamma= r/\lambda(r)$ for every $r \in I$, so that
\[
    \gamma v_i = X \cdot \nabla v_i \quad \forall i = 1, \dots, k.
\]
But this is the Euler equation for homogeneous functions, and it implies that $\bv$
is homogeneous of degree $\gamma$. Since each $v_i$ is also
harmonic in $\R^{N+1}_+$, the homogeneity extends to the whole of $\R^{N+1}_+$, yielding
$N(r) \equiv \gamma$ for every $r > 0$.
\end{proof}
In a standard way, from Theorem \ref{thm:_Almgren_for_classG_s} we infer that the growth properties of the
elements of $\classG_s$ are related with their Almgren quotient.
\begin{lemma}\label{lem: first consequence classG_s}
Let $\bv \in \classG_s$, and let $\gamma$, $\bar r$ and $C$ denote positive constants.
\begin{enumerate}
 \item If $|\bv(X)| \leq C|X-(x_0,0)|^{\gamma}$ for every $X\not\in B^+_{\bar r}(x_0,0)$, then $N(x_0,r) \leq \gamma$ for every $r>0$.
 \item If $|\bv(X)| \leq C|X-(x_0,0)|^{\gamma}$ for every $X\in B^+_{\bar r}(x_0,0)$, then $N(x_0,r) \geq \gamma$ for every $r>0$.
\end{enumerate}
\end{lemma}
\begin{proof}
Let $\bv \in \classG_s$, and let us assume the growth condition for $r\geq\bar r$.
We observe that it implies, for $r$ large, $H(r) \leq C r^{2\gamma}$.
Arguing by contradiction, let us suppose that there exists $R > \bar r$ such that $N(x_0, R) \geq \gamma + \eps$. By monotonicity of $N$ we have
\[
    \frac{\mathrm{d}}{\mathrm{d}r} \log H(r) \geq \frac{2}{r} (\gamma +\eps) \quad \forall r \geq R
\]
and, integrating in $(R,r)$, we find
\[
    C r^{2(\gamma+\eps)} \leq H(r) \leq C r^{2\gamma},
\]
a contradiction for $r$ large enough. On the other hand, if the growth condition holds
for $r<\bar r$, we can argue in an analogous way, assuming that
\[
    \frac{\mathrm{d}}{\mathrm{d}r} \log H(r) \leq \frac{2(\gamma - \eps)}{r}
\]
for $r$ small enough and obtaining again a contradiction.
\end{proof}
\begin{corollary}\label{cor:4.5}
If $\bv \in \classG_s$ is globally H\"older continuous of exponent $\gamma$ on $\R^{N+1}_+$,
then it is homogeneous of degree $\gamma$ with respect to any of its (possible) zeroes, and
\[
\mathcal{Z}:=\{x \in \R^N: \bv(x,0) = 0\}\quad\text{is an affine subspace of }\R^N.
\]
Furthermore, if $\gamma<1$, then
\[
\mathcal{Z}=\emptyset\qquad\iff\qquad \bv\text{ is a (nontrivial) constant.}
\]
\end{corollary}
\begin{proof}
On one hand, if $(x_0,0)\in\mathcal{Z}$, Lemma \ref{lem: first consequence classG_s}
implies $N(x_0,r)= \gamma$ for every $r$, and the first part easily follows.
On the  other hand, let $\mathcal{Z}=\emptyset$. By continuity, up to a relabeling,
we have that $v_1(x,0)=\dots=v_{k-1}(x,0)=0$ on $\R^N$, so that their odd extension across $\{y=0\}$ are harmonic
and globally H\"older continuous  of exponent $\gamma<1$ on the whole of $\R^{N+1}$; but then the classical Liouville Theorem implies that they are all trivial. Finally, by continuity, $v_k(x,0)$ is always
different from zero, so that $\partial_\nu v_k(x,0)\equiv0$ on $\R^N$. As a consequence, Liouville Theorem applies also to the even extension of $v_k$ across $\{y=0\}$, concluding the proof.
\end{proof}
\begin{remark}
We observe that $\bv=(1,y,0,\dots,0)$ belongs to $\classG_s$ and it is globally Lipschitz
continuous, but it is not homogeneous. This does not contradict the previous Corollary
\ref{cor:4.5}, indeed its zero set is empty.
\end{remark}
To conclude this section, we observe that the monotonicity of $N(x,r)$
implies that both for $r$ small and for $r$ large the corresponding limits are well defined.
\begin{lemma}\label{lem: second consequence classG_s}
Let $\bv\in\classG_s$. Then
\begin{enumerate}
 \item $N(x,0^+)$ is a non negative upper semicontinuous function on $\R^N$;
 \item $N(x,\infty)$ is constant (possibly $\infty$).
\end{enumerate}
\end{lemma}
\begin{proof}
The first assertion follows because $N(x,0^+)$ is the infimum of continuous functions.
On the other hand, let
\[
    \nu := \lim_{r \to \infty} N(0,r) > 0;
\]
we prove the second assertion in the case $\nu<\infty$, the other case following
 with minor changes. Let us suppose by contradiction that there exists $x_0 \in \R^N$ such that
\(
    \sup_{r > 0} N(x_0,r) = \nu - 2\eps
\)
for some $\eps > 0$. Let moreover $r_0 > 0$ be such that
\(
    N(0,r_0) \geq \nu - \eps.
\)
Reasoning as in the proof of Lemma \ref{lem: first consequence classG_s} we see that, when
$R_1$, $R_2$ are sufficiently large, both $H(x_0,R_1) \leq C R_1^{2(\nu-2\eps)}$ and $H(0,R_2) \geq C R_2^{2(\nu-\eps)}$. By definition
\[
    \int\limits_{B_{R_1}^+(x_0,0) \setminus B_{r_0}^+(x_0,0)} \tsum_i v_i^2 \, \de x \de y = \int\limits_{r_0}^{R_1} H(x_0,s) s^{N}\de s  \leq C R_1^{N + 2(\nu-2\eps)}
\]
and
\[
    \int\limits_{B_{R_2}^+(0,0) \setminus B_{r_0}^+(0,0)} \tsum_i v_i^2 \, \de x \de y = \int\limits_{r_0}^{R_2} H(0,s) s^{N}\de s \geq C R_2^{N + 2(\nu-\eps)}.
\]
Now, if we let $R_1 = R_2 + |x_0|$, we obtain
\begin{multline*}
    C R_2^{N + 2(\nu-\eps)} \leq \int\limits_{B_{R_2}^+(0,0) \setminus B_{r_0}^+(0,0)} \tsum_i v_i^2 \, \de x \de y \\
    \leq \int\limits_{B_{r_0}^+(x_0,0)} \tsum_i v_i^2 \, \de x \de y - \int\limits_{B_{r_0}^+(0,0)} \tsum_i v_i^2 \, \de x \de y + \int\limits_{B_{R_1}^+(x_0,0) \setminus B_{r_0}^+(x_0,0)} \tsum_i v_i^2 \, \de x \de y \\
    \leq C + C' (R_2 + |x_0|)^{N + 2(\nu-2\eps)}
\end{multline*}
and we find a contradiction for $R_2$ sufficiently large. Exchanging the role of $0$ and $x_0$ we can conclude.

\end{proof}

\subsection{Almgren's formula for coexistence entire profiles}
We now shift our attention to the case in which $\bv$ is a $k$-tuple of functions which a priori are not segregated, but satisfy a boundary equation on $\R^N$. In this setting, the validity of the Pohozaev identities is a consequence of the boundary equation.
\begin{definition}[Coexistence entire profiles]\label{def:coex ent prof}
We denote with $\classG_c$ the set of functions $\bv \in H^1_\loc\left(\overline{\R^{N+1}_+}\right)$ which are solutions to
\begin{equation}\label{eqn: equation of classG_c}
    \begin{cases}
    - \Delta v_i = 0 & \text{ in } \R^{N+1}_+\\
    \partial_{\nu} v_i + v_i \tsum_{j \neq i} v_j^2 = 0 & \text{ on } \R^N \times \{0\},
    \end{cases}
\end{equation}
for every $i = 1, \dots, k$.
\end{definition}
\begin{remark}
Of course, if $\bv \in H^1_\loc\left(\overline{\R^{N+1}_+}\right)$ solves
\begin{equation*}
    \begin{cases}
    - \Delta v_i = 0 & \text{ in } \R^{N+1}_+\\
    \partial_{\nu} v_i + \beta v_i \tsum_{j \neq i} v_j^2 = 0 & \text{ on } \R^N \times \{0\},
    \end{cases}
\end{equation*}
for some $\beta>0$, then a suitable multiple of $\bv$ belongs to $\classG_s$.
\end{remark}
\begin{lemma}\label{lem: Poho_G_c}
Let $\bv\in \classG_c$. For any $x_0\in\R^N$ and $r>0$, the following identity holds
\begin{multline*}
    (1-N) \int\limits_{B^+_r } \tsum_{i} |\nabla v_i|^2 \,\de{x}\de{y}  +
r\int\limits_{\partial^+B^+_r } \tsum_{i} |\nabla v_i|^2 \,\de{\sigma} - N
\int\limits_{\partial^0 B^+_r } \tsum_{i, j < i} v_i^2 v_j^2 \, \de{x}
\\+ r \int\limits_{S_r^{N-1} } \tsum_{i, j < i } v_i^2 v_j^2 \,
\de{\sigma} = 2r \int\limits_{\partial^+B^+_r }  \tsum_{i} |\partial_{\nu}
v_i|^2 \,\de{\sigma}.
\end{multline*}
\end{lemma}
\begin{proof}
The proof follows by testing equation \eqref{eqn: equation of classG_c} with $\nabla v_i \cdot X$ in $B_r^+$ and exploiting some standard integral identities (see also Lemma \ref{lem: Pohozaev identity} for a similar proof in a more general case).
\end{proof}

As before, we introduce the functions
\[
    \begin{split}
    E(x_0, r) &:= \frac{1}{r^{N-1}}   \int\limits_{B^+_r(x_0,0)}
\tsum_{i} |\nabla v_i|^2 \, \de{x} \de{y} + \frac{1}{r^{N-1}}  \int\limits_{\partial^0 B^+_r(x_0,0)} \tsum_{i,j < i} v_i^2 v_j^2 \, \de{x} \\
    H(x_0,r) &:= \frac{1}{r^{N}} \int\limits_{\partial^+ B^+_r(x_0,0)}
\tsum_{i} v_i^2 \, \de{\sigma}.
    \end{split}
\]
%
\begin{theorem}\label{thm:_Almgren_for_classG_c}
Let $\bv \in \classG_c$. For every $x_0 \in \R^N$ the function
\[
    N(x_0,r) := \frac{E(x_0,r)}{H(x_0, r)}
\]
is non decreasing, absolutely continuous and strictly positive for $r > 0$. Moreover it holds
\[
    \frac{\mathrm{d}}{\mathrm{d}r} \log H(r) \geq \frac{2N(r)}{r}.
\]
\end{theorem}
\begin{proof}
The proof runs exactly as the one of Theorem \ref{thm:_Almgren_for_classG_s}, by using
Lemma \ref{lem: Poho_G_c} instead of equation \eqref{eqn: Pohozaev spheres for classGs}.
\end{proof}
As in the case of entire profiles of segregation, we can state some first consequence of Theorem \ref{thm:_Almgren_for_classG_c}.
\begin{lemma}\label{lem:_first consequence classG_c}
Let $\bv \in \classG_c$, and let $\gamma$ and $C$ denote positive constants. If
\(
|\bv(X)| \leq C(1+|X|^{\gamma})
\)
for every $X$, then $N(x,\infty)$ is constant and less than $\gamma$.
\end{lemma}
\begin{proof}
The proof follows reasoning as in the ones of Lemmas \ref{lem: first consequence classG_s}
and \ref{lem: second consequence classG_s}.
\end{proof}

\section{Liouville type theorems}

By combining the results obtained in Sections \ref{sec:acf} and \ref{sec:almgren}, we
are in a position to prove that nontrivial entire profiles, both of segregation and of
coexistence, exhibit a minimal rate of growth connected with the Alt-Caffarelli-Friedman
exponent $\nuACF$. To be precise, the result concerning coexistence entire profiles
only relies on the arguments developed in Section \ref{sec:acf}.
\begin{proposition}\label{prp: liouville system}
Let $\bv \in \classG_c$ and $\nuACF$ be defined according to Definitions
\ref{def:coex ent prof} and \ref{def: nu}. If for some $\gamma \in (0,\nuACF)$ there exists
$C$ such that
\begin{equation*}
    |\bv(X)| \leq C\left(1 + |X|^{\gamma}\right),
\end{equation*}
for every $X$, then $k-1$ components of $\bv$ annihilate and the last is constant.
\end{proposition}
\begin{proof}
We start by proving that only one component of $\bv$ can be different from zero.
Let us suppose by contradiction that two components, say $v_1$
and $v_2$, are non trivial: indeed, we observe that $|v_1|$, $|v_2|$ fit in the setting of
Theorem \ref{thm: ACF perturbed} (recall Remark \ref{rem:chang_sign_ACF_pert}).
Let $r$ be large accordingly, and let $\eta$ be a non negative, smooth and radial
cut-off function supported in  $B_{2r}^+$ with $\eta = 1$ in $B_r^+$ and
$|\nabla \eta| \leq C r^{-1}$, $|\Delta \eta| \leq C r^{-2}$.
Moreover, let $\Gamma_1$ be defined as in
\ref{def:Gamma_1} (in particular, it is radial and superharmonic).
Testing the equation for $v_i$ with ${\Gamma_1} v_i\eta$ we obtain
\[
    \int\limits_{ B_{2r}^+ } |\nabla v_i|^2  {\Gamma_1} \eta \de{x} \de{y} +
    \int\limits_{\partial^0 B_{2r}^{+} } v_i^2 v_j^2 {\Gamma_1} \eta \de{x}
    \leq \int\limits_{ B_{2r}^+ \setminus B_{r}^+} \frac{1}{2} v_i^2
\left[ {\Gamma_1}  \Delta \eta + 2 \nabla \eta \cdot \nabla {\Gamma_1} \right]
\de{x} \de{y},
\]
where in the last step we used that $\eta$ is constant in $B_r^+$. Since
$\Gamma_1(X)=|X|^{1-N}$ outside $B_1$, and $|v_i(X)| \leq C r^{\gamma}$ outside
a suitable $B_{\bar r}$, using the notations of Theorem \ref{thm: ACF perturbed} we infer
\[
\Phi_i(r) = \frac{1}{r^{2\nu'}}\left(\int\limits_{B_r^+} |\nabla v_i|^2 {\Gamma_1}
\,\de{x} \de{y} +  \int\limits_{\partial^0 B_r^+} v_i^2 v_j^2 {\Gamma_1} \,\de{x}
\right)\leq \frac{1}{r^{2\nu'}} \cdot C r^{2\gamma},
\]
with $C$ independent of $r>\bar r$. Fixing $\gamma<\nu'<\nuACF$ and possibly
taking $\bar{r}$ larger, Theorem \ref{thm: ACF perturbed} states that
\begin{equation*}
0 <  \Phi(\bar{r}) \leq \Phi(r) = \prod_{i=1}^{2} \Phi_i(r) \leq C
r^{4(\gamma-\nu')},
\end{equation*}
a contradiction for $r$ large enough. Finally, if $v_1$ is the unique non
trivial component of $\bv$, an even extension of $v_1$ through $\R^{N}$ is
harmonic in $\R^{N+1}$ and bounded everywhere by a function growing less than
linearly, implying that $v_1$ is constant.
\end{proof}
Turning to segregation entire profiles, the results of Section \ref{sec:almgren}
become crucial.
\begin{proposition}\label{prp: liouville inequalities}
Let $\bv \in \classG_s$ and $\nuACF$ be defined according to Definitions
\ref{def:segr ent prof} and \ref{def: nu}.
\begin{enumerate}
 \item If for some $\gamma \in (0,\nuACF)$ there exists $C$ such that
\begin{equation*}
    |\bv(X)| \leq C\left(1 + |X|^{\gamma}\right),
\end{equation*}
for every $X$, then $k-1$ components of $\bv$ annihilate;
 \item if furthermore $\bv \in \C^{0,\gamma}\left(\overline{\R^{N+1}_+}\right)$ then the only
 possibly nontrivial component is constant.
\end{enumerate}
\end{proposition}
\begin{remark}
We notice that the uniform H\"older continuity of exponent $\gamma$ required in 2. readily
implies the growth condition in 1., which we may not require explicitly. On the other hand,
from the proof it will be clear that, once $k-1$ components annihilate,
2. follows by assuming uniform H\"older continuity of \emph{any} exponent $\gamma'\in(0,1)$,
not necessarily related to $\nuACF$.
\end{remark}
\begin{proof}[Proof of Proposition \ref{prp: liouville inequalities}]
To prove 1., we start as above by assuming by contradiction that there exist two components,
$v_1$ and $v_2$, which are non trivial.
We deduce that they must have a common zero on $\R^N$. As a consequence, we can reason as
in the proof of Proposition \ref{prp: liouville system}, using Theorem \ref{thm: ACF}
(and Remark \ref{rem:chang_sign_ACF_segr}) instead of Theorem \ref{thm: ACF perturbed}, and obtain a contradiction. Turning to 2.,
let $v$ denote the only non trivial component. By Corollary \ref{cor:4.5}, we
have that the set
\[
\mathcal{Z} = \{x \in \R^N: v(x,0) = 0\}
\]
is an affine subspace of $\R^N$. Now, if $\mathcal{Z}=\R^N$, then $v$ satisfies
\[
    \begin{cases}
    - \Delta v = 0 & \text{in } \R^{N+1}_+ \\
    v = 0 & \text{on } \R^{N},
    \end{cases}
\]
so that the odd extension of $v$ through $\{ y = 0 \}$ is harmonic in
$\R^{N+1}$ and bounded everywhere by a function growing less than linearly,
implying that $v$ is constant. On the other hand, if $\dim \mathcal{Z} \leq N-1$, then
\[
    \begin{cases}
    - \Delta v = 0 & \text{in } \R^{N+1}_+ \\
    \partial_{\nu} v = 0 & \text{on } \R^{N}\setminus \mathcal{Z},
    \end{cases}
\]
and the even reflection of $v$ through $\{ y = 0 \}$ is harmonic in $\R^{N+1} \setminus \mathcal{Z}$; since $\mathcal{Z}$ has null capacity with respect to $\R^{N+1}$, we infer that $v$ is actually harmonic in $\R^{N+1}$, and the conclusion follows again since, by assumption, $v$ is bounded everywhere by a function growing less than linearly.
\end{proof}
In the same spirit of the previous theorems, we provide now a result
concerning single functions, rather than $k$-tuples.
\begin{proposition}\label{prp: liouville_boundary}
Let $v\in H^1_\loc\left(\overline{\R^{N+1}_+}\right)$ be continuous and satisfy
\begin{equation*}
    \begin{cases}
    - \Delta v = 0 & \text{in } \R^{N+1}_+ \\
    v\partial_{\nu} v \leq 0 & \text{on } \R^{N}\\
    v(x,0)=0       & \text{on } \{x_1\leq0\},
    \end{cases}
\end{equation*}
and let us suppose that for some $\gamma \in [0,1/2)$, $C >0$ it holds
\[
    |v(X)| \leq C(1+|X|^{\gamma})
\]
for every $X$. Then $v$ is constant.
\end{proposition}
\begin{proof}
It is trivial to check that $v$ as above fulfills the assumptions of
Proposition \ref{prp: ACF sym}. Now, assuming that $v$ is not constant, we can argue
as in the proof of Proposition \ref{prp: liouville system} obtaining a contradiction.
\end{proof}
To conclude the section, we provide other two theorems of Liouville type concerning
single functions. The first one relies on the construction
of a supersolution of a suitable problem, as done in the following lemma.
\begin{lemma}\label{eqn: decay with perturbations}
Let $M>0$ and $\delta > 0$ be fixed and let $h \in L^{\infty}(\partial^0 B^+_1)$
with $\|h\|_{L^{\infty}} \leq \delta$. Any $v \in H^1(B^+_1) \cap
\C\left(\overline{B^+_1}\right)$ non negative solution to
\[
    \begin{cases}
    - \Delta v \leq 0, & \text{in }B^+_1\\
    \partial_{\nu} v \leq -M v + h, & \text{on }\partial^0 B^+_1
    \end{cases}
\]
verifies
\[
    \sup_{\partial^0 B^+_{1/2}} v \leq \frac{1+\delta}{M} \sup_{\partial^+ B^+_1} v.
\]
\end{lemma}
\begin{proof}
The proof follows from a simple comparison argument, once one notices that, for any $\delta > 0$, the function
\begin{equation*}
    w_{\delta} := \delta \frac{1}{M} + \frac{1}{N} \sum_{i=1}^{N} \frac{2}{\pi} \left[ \frac{\pi}{2} -
\arctan\left(\frac{x_i+1}{y + \frac{2}{M}}\right) + \frac{\pi}{2} -
\arctan\left(\frac{1-x_i}{y + \frac{2}{M}}\right) \right]
\end{equation*}
satisfies the following system
\[
    \begin{cases}
    - \Delta w_{\delta} = 0 & \text{in }B^+_1\\
    \partial_{\nu} w_{\delta}\geq -M w_{\delta} + \delta & \text{on }\partial^0
B^+_1\\
    w_{\delta} \geq 1 & \text{on }\partial^+ B^+_1\\
    w_{\delta} \leq \frac{1+\delta}{M} &\text{in }\partial^0 B^+_{1/2}.
    \end{cases}
\]
The claim can be proved by direct checking. For the reader's convenience, we
sketch it in the case $N=1$, $\delta=0$.

For notation convenience, let us denote $w_M$ by $w$. It is a straightforward
computation to verify that $w$ is positive and harmonic in $\R^2_+$. Using
the elementary inequality $\frac{\pi}{2} - \arctan t \geq \frac{1}{1+t}$ for all
$t \geq 0$, we can estimate
\[
    \begin{split}
    w(x,0) \geq  \frac{2}{\pi} \left[
\frac{1}{ 1 + \frac{M}{2}(x+1)} + \frac{1}{ 1 + \frac{M}{2}(1-x)} \right].
    \end{split}
\]
On the other hand, using the inequality $\frac{t}{1+t^2} \leq \frac{2}{1+t}$ for
$t \geq 0$, we have
\[
w_y(x,0)\leq \frac{2}{\pi} M \left[ \frac{1}{ 1 + \frac{M}{2}(x+1)} + \frac{1}{ 1 +
\frac{M}{2}(1-x)} \right].
\]
Therefore, $\partial_{\nu} w(x,0) = -w_y(x,0)\geq -M w(x,0)$. For $(x,y) \in
\overline{B^+_1}$ we have
\[
    \arctan\left(\frac{x+1}{y + \frac{2}{M}}\right) + \arctan\left(\frac{1-x}{y
+ \frac{2}{M}}\right) \leq \frac{\pi}{2},
\]
that is $w(x,y) \geq 1$ in $B^+_1$. Finally, we observe that $w(x,0)$, as a
function of $x$, is strictly convex and even in $(-1,1)$. Consequently,
if $|x|\leq\frac{1}{2}$, using the elementary inequality $\frac{\pi}{2} - \arctan
t \leq \frac{1}{t}$ for $t \geq 0$, we obtain
\[
    w(x,0) \leq  \frac{1}{M}.
\qedhere
\]
\end{proof}

\begin{remark}
One of the peculiar difficulties in dealing with fractional operators
with respect to the standard local case is due to the slow decay of supersolutions.
Indeed, in the pure laplacian case, it is well known that positive solutions of
\[
    -\Delta u \leq - M u \quad \text{ in } B \subset \R^N
\]
exhibit an exponential decay, that is $u|_{B_{1/2}} \leq
e^{-\frac{1}{2}\sqrt{M}} \sup_{\partial B} u$; see, for instance, \cite{ctv,
nttv}. In great contrast with this result, in the previous lemma we proved that non
negative solutions of
\[
    (-\Delta)^{1/2} u \leq - M u \quad \text{ in } B \subset \R^N
\]
exhibit only polynomial decay, that is $u|_{B_{1/2}} \leq \frac{1}{M}
\sup_{\R^N\setminus B} u$. This estimate is sharp, since
\[
    \begin{cases}
    - \Delta v = 0 &\text{ in } B^+\\
    v\geq 0 &\text{ in } B^+\\
    \partial_{\nu} v = - M v &\text{ on } \partial^0 B^+
    \end{cases}
\]
implies
\[
    \inf_{\partial^0 B^+_{1/2}} v \geq \frac{1}{1+M} \inf_{\partial^+ B^+} v.
\]
This last fact follows by a comparison between $v$ and the subsolution
$w =\frac{1}{1+M}(1+My) \inf_{\partial^+ B^+} v$.
\end{remark}
The previous estimate allows to prove the following.
\begin{proposition}\label{prp: global eigenfunction}
Let $v$ satisfy
\begin{equation*}
    \begin{cases}
    - \Delta v = 0 & \text{in } \R^{N+1}_+ \\
    \partial_{\nu} v = - \lambda v & \text{on } \R^{N}
    \end{cases}
\end{equation*}
for some $\lambda \geq 0$ and let us suppose that
for some $\gamma \in [0,1)$, $C >0$ it holds
\[
    |v(X)| \leq C(1+|X|^{\gamma})
\]
for every $X$. Then $v$ is constant.
\end{proposition}
\begin{proof}
If $\lambda = 0$, using an even reflection through $\{y=0\}$, we
extend $v$ to a harmonic function in all $\R^{N+1}$, and we conclude as usual using
the growth assumption. If $\lambda > 0$ let either $z=v^+$ or $z=v^-$. In both cases,
\[
    \begin{cases}
    - \Delta z \leq 0, & \text{in }\R^{N+1}_+\\
    \partial_{\nu} z \leq -M z, & \text{on }\R^N.
    \end{cases}
\]
By translating and scaling, Lemma \ref{eqn: decay with perturbations} implies that
\[
z(x_0,0) \leq \sup_{\partial^0 B_{r/2}(x_0,0)} z \leq \frac{1}{\lambda r}\sup_{\partial^+ B_r(x_0,0)} z
\leq C \frac{1+r^{\gamma}}{r}.
\]
Letting $r\to\infty$ the proposition follows.
\end{proof}
Finally, we have the following.
\begin{proposition}\label{prp: prescribed constant normal derivative}
Let $v$ satisfy
\begin{equation*}
    \begin{cases}
    - \Delta v = 0 & \text{in } \R^{N+1}_+ \\
    \partial_{\nu} v = \lambda  & \text{on } \R^{N}
    \end{cases}
\end{equation*}
for some $\lambda \in \R$ and let us suppose that
for some $\gamma \in [0,1)$, $C >0$ it holds
\[
    |v(X)| \leq C(1+|X|^{\gamma})
\]
for every $X$. Then $v$ is constant.
\end{proposition}
\begin{proof}
For $h \in \R^N$, let $w(x,y) := v(x+h,y) - v(x,y)$. Then $w$ solves
\[
    \begin{cases}
    - \Delta w = 0 & \text{in } \R^{N+1}_+ \\
    \partial_{\nu} w = 0 & \text{on } \R^{N}
    \end{cases}
\]
and, as usual, we can reflect and use the growth condition to infer that $w$ has to be constant, that is $v(x+h,y) = c_h + v(x,y)$. Deriving the previous expression in $x_i$, we find that
\[
    v(x,y) = \sum_{i = 1}^k c_i(y) x_i + c_0(y).
\]
Using again the growth condition, we see that $c_i \equiv 0$ for $i = 1, \dots, k$, while $c_0$ is constant. We observe that, consequently, $\lambda = 0$.
\end{proof}

\section{Some approximation results}
In the following, we want to apply the Liouville type theorems obtained in the previous section to
suitable limiting profiles, obtained from solutions to the problem
\[\label{prova}
    \begin{cases}
    - \Delta v_i = 0 & \text{in } B^+\\
    \partial_{\nu} v_i = f_{i, \beta}(v_i) - \beta v_i \tsum_{j \neq i} v_j^2 & \text{on } \partial^0 B^+,
    \end{cases} \eqno \problem{\beta}
\]
through some blow up and blow down procedures. From this point of view we have seen that, in the case of entire profiles of segregation, the key property is the validity of some Pohozaev identities, which imply
that the Almgren formula holds. In this section we prove that such identities can be obtained by
passing to the limit in the corresponding identities for $\problem{\beta}$, under suitable assumptions about the convergence. To be more precise, we will prove the following.
\begin{proposition}\label{prp: approx classG_s}
Let $\bv_n \in H^1\left(B^+_{r_n}\right)$ solve problem
$\problem{\beta_n}$ on $B^+_{r_n}$, $n\in\N$, and $\bv \in H^1_{\loc}\left(\overline{\R^{N+1}_+}\right)$, be such that,
as $n\to\infty$,
\begin{enumerate}
	\item $\beta_n \to \infty$;
    \item $r_n \to \infty$;
    \item for every compact $K \subset \R^{N+1}_+$, $\bv_n \to \bv$ in $H^1(K) \cap C(K)$;
    \item the continuous functions $f_{i,\beta_n}$ are such that, for every $\bar m> 0$,
        \[
        |f_{i,\beta_n}(s)| \leq C_{n}(\bar m)\quad\text{ for }|s| < \bar m,
        \]
        where $C_{n}(\bar m) \to 0$.
\end{enumerate}
Then $\bv \in \classG_s$.
\end{proposition}
We start by stating the basic identities for problem $\problem{\beta}$. We recall that
$S_{r}^{N-1}$ denotes the $(N-1)$-dimensional boundary of $\partial^0 B^+_r$ in $\R^N$.
\begin{lemma}[Pohozaev identity]\label{lem: Pohozaev identity}
Let $\bv$ solve problem $\problem{\beta}$ on $B^+$. For every
$B^+_r:=B^+_r(x_0,0) \subset B^+$ the following Pohozaev identity holds:
\begin{multline*}
    (1-N) \int\limits_{B^+_r } \tsum_{i} |\nabla v_i|^2 \,\de{x}\de{y}  +
r\int\limits_{\partial^+B^+_r } \tsum_{i} |\nabla v_i|^2 \,\de{\sigma} + \\ + 2N
\int\limits_{\partial^0 B^+_r } \tsum_{i} F_{i,\beta}(v_i) \, \de{x} - N \beta
\int\limits_{\partial^0 B^+_r } \tsum_{i, j < i} v_i^2 v_j^2 \, \de{x}
      - 2r \int\limits_{S_r^{N-1} } \tsum_{i} F_{i,\beta}(v_i) \, \de{\sigma}+
\\+ r\beta \int\limits_{S_r^{N-1} } \tsum_{i, j < i } v_i^2 v_j^2 \,
\de{\sigma} = 2r \int\limits_{\partial^+B^+_r }  \tsum_{i} |\partial_{\nu}
v_i|^2 \,\de{\sigma}.
\end{multline*}
\end{lemma}
\begin{proof}
Let the functions $v_i$ solve problem $\problem{\beta}$. Up to translations, we assume that $x_0=0$.
By multiplying the equation with $X \cdot \nabla v_i$ and integrating by parts
over $B^+_r $, we obtain
\[
\int\limits_{B^+_r } \nabla v_i \cdot \nabla(X \cdot \nabla v_i) \,\de{x}\de{y}
= r \int\limits_{\partial^+B^+_r } |\partial_{\nu} v_i|^2 \,\de{\sigma} +
\int\limits_{\partial^0B^+_r } (\partial_{\nu} v_i) (x\cdot \nabla_x v_i) \,
\de{x}.
\]
Using the identity
\[
\nabla v_i \cdot \nabla(X \cdot \nabla v_i) = |\nabla v_i|^2 + X \cdot \nabla
\left(\frac{1}{2} |\nabla v_i|^2\right)
\]
and integrating again by parts, we can write the right hand side as
\[
\int\limits_{B^+_r } \nabla v_i \cdot \nabla(X \cdot \nabla v_i) \,\de{x}\de{y}
= \frac{1-N}{2} \int\limits_{B^+_r } |\nabla v_i|^2 \,\de{x}\de{y}  +
\frac{r}{2}\int\limits_{\partial^+B^+_r } |\nabla v_i|^2 \,\de{\sigma}
\]
and this yields
\begin{multline*}
    \frac{1-N}{2} \int\limits_{B^+_r } |\nabla v_i|^2 \,\de{x}\de{y}  +
\frac{r}{2}\int\limits_{\partial^+B^+_r } |\nabla v_i|^2 \,\de{\sigma} -
\int\limits_{\partial^0B^+_r } f_{i,\beta}(v_i) (x\cdot \nabla_x v_i) \, \de{x}
+ \\
    + \frac{\beta}{2} \int\limits_{\partial^0B^+_r } (x\cdot \nabla_x v_i^2)
\tsum_{j\neq i} v_j^2 \, \de{x} = r \int\limits_{\partial^+B^+_r }
|\partial_{\nu} v_i|^2 \,\de{\sigma}.
\end{multline*}
Summing the identities for $i =1, \dots, k$ we obtain
\begin{multline}\label{eqn:_passaggio_Pohozaev}
    \frac{1-N}{2} \int\limits_{B^+_r } \tsum_{i} |\nabla v_i|^2 \,\de{x}\de{y}
+ \frac{r}{2}\int\limits_{\partial^+B^+_r } \tsum_{i} |\nabla v_i|^2
\,\de{\sigma} \\- \int\limits_{\partial^0B^+_r } (x\cdot \nabla_x )\tsum_{i} F_{i,\beta}(v_i)
 \, \de{x} 
    + \frac{\beta}{2} \int\limits_{\partial^0B^+_r } (x\cdot \nabla_x) \tsum_{i,
j < i } v_i^2 v_j^2 \, \de{x} = r \int\limits_{\partial^+B^+_r }  \tsum_{i}
|\partial_{\nu} v_i|^2 \,\de{\sigma}.
\end{multline}
The terms on $\partial^0 B^+_r $ can be further simplified: by an application of
the divergence theorem on $\R^N$ we have
\[
    \begin{split}
    \int\limits_{\partial^0B^+_r } (x\cdot \nabla_x) \tsum_{i, j < i } v_i^2
v_j^2 \, \de{x} &= \int\limits_{\partial^0B^+_r } \div \left(x \tsum_{i, j < i }
v_i^2 v_j^2 \right) \, \de{x} - \int\limits_{\partial^0B^+_r } \div x  \tsum_{i,
j < i} v_i^2 v_j^2 \, \de{x}\\
    &= r \int\limits_{S_r^{N-1} } \tsum_{i, j < i } v_i^2 v_j^2 \, \de{\sigma} -
N \int\limits_{\partial^0 B^+_r } \tsum_{i, j < i } v_i^2 v_j^2 \, \de{x}
    \end{split}
\]
and
\[
    \begin{split}
    \int\limits_{\partial^0B^+_r } (x\cdot \nabla_x) \tsum_{i} F_{i,\beta}(v_i)
\, \de{x} &= \int\limits_{\partial^0B^+_r } \div \left(x \tsum_{i}
F_{i,\beta}(v_i) \right) \de{x} - \int\limits_{\partial^0B^+_r } \div x
\tsum_{i} F_{i,\beta}(v_i) \de{x}\\
    &= r \int\limits_{r S^{N-1} } \tsum_{i} F_{i,\beta}(v_i) \, \de{\sigma} - N
\int\limits_{\partial^0 B^+_r } \tsum_{i} F_{i,\beta}(v_i) \, \de{x};
    \end{split}
\]
the lemma  follows by substituting into equation \eqref{eqn:_passaggio_Pohozaev}.
\end{proof}
In a similar way, it is possible to prove the validity of the Pohozaev identities in cylinders
(we use the notations introduced in the discussion at the beginning of Section \ref{subsec: Alm seg prof}).
\begin{lemma}[Pohozaev identity in cylinders]\label{lem: Pohozaev identity cylinder}
Let $\bv \in H^1(B^+)$ be a solution to problem
$\problem{\beta}$. For every $x \in \partial^0 B^+$ and $r>0$, $l>0$ such that
$C_{r,l}^+ \subset B^+$ the following Pohozaev identity holds:
\begin{multline*}
        \int\limits_{C_{r,l}^+} \left(\tsum_{i} 2 |\nabla_{(x',y)} v_i|^2 -(h+1) |\nabla v_i|^2 \right)\de{x}\de{y}  +
    r\int\limits_{\partial^+B^+_r \times Q_l } \tsum_{i} |\nabla v_i|^2 \,\de{\sigma} + \\ + 2h
    \int\limits_{\partial^0 C_{r,l}^+ } \tsum_{i} F_{i,\beta}(v_i) \, \de{x} - h \beta
    \int\limits_{\partial^0 C_{r,l}^+ } \tsum_{i, j < i} v_i^2 v_j^2 \, \de{x}\\
          - 2r \int\limits_{S_r^{h-1} \times Q_l } \tsum_{i} F_{i,\beta}(v_i) \, \de{\sigma}+
     r\beta \int\limits_{S_r^{h-1} \times Q_l } \tsum_{i, j < i } v_i^2 v_j^2 \,
    \de{\sigma}\\ = 2r \int\limits_{\partial^+B^+_r \times Q_l }  \tsum_{i} |\partial_{\nu} v_i|^2 \,\de{\sigma}- 2 \int\limits_{B^+_r \times \partial^+Q_l }  \tsum_{i} \partial_{\nu}    v_i \nabla_{(x',y)} v_i \cdot (x',y) \,\de{\sigma},
\end{multline*}
where $\nabla_{(x',y)}$ is the gradient with respect to the directions in $\R^{h+1}_+$.
\end{lemma}
\begin{remark}
Even though the mentioned Pohozaev identities are enough for our purposes, we would like to point
out that they are nothing but special cases of a more general class of identities, namely the
domain variation formulas, see for instance \cite{dwz3}. They may be obtained by testing the equation of
$\problem{\beta}$ by $\nabla \bv \cdot Y$ in a smooth domain $\omega \subset \R^{N+1}_+$, where $Y
\in \C^{1}(\R^{N+1}_+; \R^{N+1}_+)$ is a smooth vector field such that $Y|_{y=0} \in \C^{1}(\R^{N};
\R^{N})$.
\end{remark}
To proceed, we need the following standard result.
\begin{lemma}\label{lem: measure inequality}
Let $f, \lambda \in L^{\infty}(\partial^0 B^+)$. If $w \in H^1(B^+)$ is a solution to
\[
    \begin{cases}
    - \Delta w = 0 &\text{ in } B^+\\
    \partial_{\nu} w = f - \lambda w &\text{ on } \partial^0 B^+,
    \end{cases}
\]
then $|w| \in H^1(B^+)$ and for any $\phi \in H^1(B^+)$, $\phi|_{\partial^+ B^+} = 0$, $\phi \geq 0$ it holds
\[
    \int\limits_{B^+} \nabla |w| \cdot \nabla \phi \,\de x \de y - \int\limits_{\partial^0B^+} (|f| -\lambda |w|) \phi \de x \leq 0
\]
\end{lemma}
\begin{proof}
Let $g_{\eps} (s) = \sqrt{s^2+\eps} \in \C^{\infty}(\R)$ such that $g_{\eps} (s) \to |s|$ and $g_{\eps}' (s) \to \sgn(s)$. By Stampacchia Lemma,
\[
	g_{\eps}(w) \to |w| \quad \text{ in } H^1(B^+)
\]
while, by Lebesgue theorem
\[
	g_{\eps}'(w)w \to |w| \quad \text{ in } L^2(\partial^0 B^+).
\]
Thus, for any $\phi \in H^1(B^+)$, $\phi|_{\partial^+ B^+} = 0$, $\phi \geq 0$, we have
\begin{multline*}
    \int\limits_{B^+} \nabla g_{\eps}(w) \cdot \nabla \phi \,\de x \de y - \int\limits_{\partial^0B^+} g_{\eps}'(v) (f -\lambda w) \phi \de x
\\= \int\limits_{B^+} g_{\eps}'(w) \nabla w \cdot \nabla \phi \,\de x \de y - \int\limits_{\partial^0B^+} g_{\eps}'(w) \partial_{\nu} v \phi \de x \\
= \int\limits_{B^+} - \div \left( g_{\eps}'(w) \nabla w \right) \phi \,\de x \de y
= \int\limits_{B^+} \left( - g_{\eps}''(w) |\nabla w|^2 - g_{\eps}'(w) \Delta w \right) \phi\,\de x \de y \leq 0.
\end{multline*}
Passing to the limit for $\eps \to 0$ we obtain the lemma.
\end{proof}
Going back to the notations of Proposition \ref{prp: approx classG_s}, we have the following lemma.
\begin{lemma}\label{lem:_competition_term_vanishes_in_strong_blowup}
For every $K$ compact subset of $\R^N$, it holds
\[
    \lim_{n \to \infty} \beta_n \int\limits_{K} v_{i,n}^2 \tsum_{j \neq i} v_{j,n}^2 \, \de{x} = 0.
\]
Moreover, for every $x_0\in\R^N$, and for almost every $r>0$,
\[
\beta_n \int\limits_{S_r^{N-1}}v_{i,n}^2 \tsum_{j \neq i} v_{j,n}^2\de{\sigma} \to 0.
\]
\end{lemma}
\begin{proof}
Let $\eta \in \C^{\infty}_0(B_{r})$ be a
positive smooth cutoff function with the property that $\eta \equiv 1$ on $K$.
Taking into account Lemma \ref{lem: measure inequality}, we obtain that
\[
    0\leq \beta_n \int\limits_{K} |v_{i,n}| \tsum_{j \neq i} v_{j,n}^2 \, \de{x} \leq
\int\limits_{\partial^0 B^+_{r} } \left( |f_{i,n}|\eta  - |v_{i,n}| \partial_{\nu}\eta  \right)\de{x} +
\int\limits_{B^+_{r}} |v_{i,n}| \Delta \eta \,\de{x} \de{y} \leq C.
\]
In particular, on one hand this implies that
\[
\beta_n \int\limits_{K} |v_{i,n}| \tsum_{j \neq i} v_{j,n}^2 \, \de{x} \leq C,
\]
while on the other hand, by passing to the limit, we infer that $\left\{v_{i}=0\right\} \cup
\left\{v_{j}=0\right\}$ contains $K$, for every $i\neq j$. As a consequence, each term in the
sum can be estimated as follows
\[
    \begin{split}
	\beta_n \int\limits_{K} v_{i,n}^2 v_{j,n}^2 \,\de{x} \leq& \beta_n
\int\limits_{K\cap \{v_{i} = 0\}} v_{i,n}^2 v_{j,n}^2 \,\de{x}+ \beta_n
\int\limits_{K\cap \{v_{j} = 0\}} v_{j,n}^2 v_{i,n}^2 \,\de{x} \\
    \leq& \|v_{i,n}\|_{L^{\infty}(K \cap \{v_{i} = 0\})}\beta_n
\int\limits_{K\cap \{v_{i} = 0\}} |v_{i,n}| v_{j,n}^2 \,\de{x}\\
    &+  \|v_{j,n}\|_{L^{\infty}(K \cap \{v_{j} = 0\})}\beta_n
\int\limits_{K \cap \{v_{j} = 0\}} |v_{j,n}| v_{i,n}^2 \,\de{x} \to 0,
    \end{split}
\]
and the first conclusion follows by summing over all $j \neq i$. As far as
the second one is concerned, it follows by applying Fubini's Theorem to the previous
conclusion when $K=\partial^0 B^+_R$.
\end{proof}

\begin{proof}[Proof of Proposition \ref{prp: approx classG_s}]
First we notice that, by Lemma \ref{lem:_competition_term_vanishes_in_strong_blowup},
it holds $v_i v_j \equiv 0$ for every $i\neq j$. Moreover, since the uniform limit
of harmonic functions is harmonic itself, $\Delta v_i=0$ on $\R^{N+1}_+$. In order to
obtain \eqref{eqn: equation of classG_s}, we observe that, for any $\eta \in \C^{\infty}_0(\R^N)$,
we have
\[
\int\limits_{\R^N} v_{i,n} \partial_{\nu} v_{i,n} \phi \, \de{x} =
\int\limits_{\R^N} \left( v_{i,n} f_{i, \beta_n}(v_{i,n}) - \beta_n v_{i,n}^2
\tsum_{j \neq i} v_{j,n}^2 \right) \phi \, \de{x}\to 0
\]
by assumption 4. and Lemma \ref{lem:_competition_term_vanishes_in_strong_blowup}. Finally,
to prove that \eqref{eqn: Pohozaev for classGs} holds, we are going to show that, for every
$x_0 \in \R^N$ and almost every $r > 0$, the Pohozaev identity of Lemma
\ref{lem: Pohozaev identity} passes to the limit (the general case following
by analogous arguments). Let us recollect the terms of the identity as
\begin{multline*}
    \underbrace{(1-N) \int\limits_{B^+_r } \tsum_{i} |\nabla v_{i,n}|^2 \,\de{x}\de{y}}_{A_n} +
\underbrace{r\int\limits_{\partial^+B^+_r } \tsum_{i} |\nabla v_{i,n}|^2 \,\de{\sigma}}_{B^1_n} + \\
\underbrace{+ 2N \int\limits_{\partial^0 B^+_r } \tsum_{i} F_{i,n}(v_{i,n}) \, \de{x} - 2r \int\limits_{S_r^{N-1} } \tsum_{i} F_{i,n}(v_{i,n}) \, \de{\sigma}}_{I_n}+
\\ \underbrace{- N \beta_n \int\limits_{\partial^0 B^+_r } \tsum_{i, j < i} v_{i,n}^2 v_{i,n}^2 \, \de{x} + r\beta \int\limits_{S_r^{N-1} } \tsum_{i, j < i } v_{i,n}^2 v_{j,n}^2 \, \de{\sigma}}_{C_n}
= \underbrace{2r \int\limits_{\partial^+B^+_r }  \tsum_{i} |\partial_{\nu} v_{i,n}|^2 \,\de{\sigma}.}_{B^2_n}
\end{multline*}
On one hand, by strong $H^1_{\loc}$ convergence,
\[
    A_n \to (1-N) \int\limits_{B^+_r } \tsum_{i} |\nabla v_i|^2 \,\de{x}\de{y}.
\]
Moreover, both $I_n \to 0$ (by assumption 4.) and $C_n \to 0$ for a.e. $r$ (by Lemma
\ref{lem:_competition_term_vanishes_in_strong_blowup}). We claim that
\[
    \lim_{n \to \infty} B^1_n = r\int\limits_{\partial^+B^+_r } \tsum_{i} |\nabla v_i|^2 \,\de{\sigma} \quad \text{ and } \quad \lim_{n \to \infty} B^2_n = 2r \int\limits_{\partial^+B^+_r }  \tsum_{i} |\partial_{\nu} v_i|^2 \,\de{\sigma}
\]
in $L^1_{\loc}[0,\infty)$: in particular, this will imply convergence for a.e. $r$.
Let us prove the former limit, which implies also the latter. The
strong convergence $\bv_n \to \bv$ in $H^1_\loc\left(\overline{\R^{N+1}_+}\right)$ implies that
\[
    \int\limits_0^{R} \int\limits_{\partial^+ B^+_r} \tsum_{i} |\nabla v_{i,n} - \nabla
v_i |^2 \, \de{\sigma} \mathrm{d}r \to 0,
\]
so that $\int\limits_{\partial^+ B^+_r} |\nabla v_{i,n}|^2 \de{\sigma} \to \int\limits_{\partial^+ B^+_r} |\nabla
v_{i,n}|^2 \de{\sigma}$ for a.e. $r$ and there exists an integrable function $f \in L^1(0,R)$ such
that, up to a subsequence
\[
    \int\limits_{\partial^+ B^+_r} |\partial_{\nu} v_{i,n_k}|^2 \de{\sigma} \leq
\int\limits_{\partial^+ B^+_r} |\nabla v_{i,n_k}|^2 \de{\sigma} \leq f(r) \quad
\text{a.e. } r \in (0,R)
\]
for every $i = 1, \dots, k$. We can then use the Dominated Convergence Theorem.
Since every subsequence of $\{\bv_{n}\}_{n \in \N}$ admits a convergent
sub-subsequence, and the limit is the same, we conclude the convergence for the
entire approximating sequence.
\end{proof}
%
%
\section{Local $\C^{0,\alpha}$ uniform bounds, $\alpha$ small}\label{section:_uniform_local}
%
%
%
In this section we begin our regularity analysis with a first partial result. We will obtain a localized version of the uniform H\"older regularity for solutions to problem $\problem{\beta}$ (introduced at page \pageref{prova}), when the H\"older exponent is sufficiently small. We recall that, here and in the following, the functions $f_{i,\beta}$ are assumed to be continuous and uniformly bounded, with respect to $\beta$, on bounded sets.
\begin{remark}\label{rem: troianiello}
By standard regularity results (see for instance the book \cite{trobook}), we already know that
for every $r<1$, $\alpha\in(0,1)$, $\bar m>0$ and $\bar\beta>0$, there exists a constant
$C=C(r,\alpha,\bar m,\bar\beta)$ such that
\[
    \| \bv_\beta\|_{\C^{0,\alpha}\left(\overline{B^+_{r}}\right)} \leq C,
\]
for every $\bv_\beta$ solution of problem $\problem{\beta}$ on $B^+_1$, satisfying
\[
\beta\leq\bar\beta\quad\text{ and }\quad  \| \bv_{\beta} \|_{L^{\infty}(B^+_1)} \leq \bar m.
\]
\end{remark}
The main result of this section is the following.
\begin{theorem}\label{thm:_local_holder}
Let $\{\bv_{\beta}\}_{\beta > 0}$ be a family of solutions to problem $\problem{\beta}$ on $B^+_1$
such that
\[
    \| \bv_{\beta} \|_{L^{\infty}(B^+_1)} \leq \bar m,
\]
with $\bar m$ independent of $\beta$. Then for every $\alpha \in (0,\nuACF)$ there exists a constant
$C = C(\bar m,\alpha)$, not depending on $\beta$, such that
\[
    \| \bv_\beta\|_{\C^{0,\alpha}\left(\overline{B^+_{1/2}}\right)} \leq C.
\]
Furthermore, $\{\bv_\beta\}_{\beta > 0}$ is relatively compact in $H^1(B^+_{1/2}) \cap \C^{0,\alpha}\left(\overline{B^+_{1/2}}\right)$ for every $\alpha < \nuACF$.
\end{theorem}
\begin{remark}
Even though we prove it in $\overline{B^+_{1/2}}$, Theorem \ref{thm:_local_holder} holds also
when replacing $\overline{B^+_{1/2}}$ with $K\cap B^+_1$, for every compact set $K\subset B_1$.
\end{remark}
For easier notation, we write $B^+=B^+_{1}$. Inspired by \cite{nttv, wang}, we proceed by
contradiction and develop a blow up analysis. First, let $\eta$ denote a smooth function such that
\begin{equation}\label{eqn: eta_blowup}
 \begin{cases}
 \eta(X) = 1   &  0\leq |X|\leq \frac12\\
 0<\eta(X) \leq 1   &  \frac12\leq |X|\leq 1\\
 \eta(X) = 0   &  |X|=1\\
 \end{cases}
\end{equation}
(in particular, $\eta$ vanishes on $\partial^+B^+$ but is strictly positive $\partial^0B^+$).
We will prove that
\[
\| \eta\bv\|_{\C^{0,\alpha}\left(\overline{B^+}\right)} \leq C,
\]
and the theorem will follow by the regularity of $\eta$.

Let us assume by contradiction the existence of sequences
$\{\beta_n\}_{n \in \N}$, $\{\mathbf{v}_n\}_{n\in \N}$,
solutions to $\problem{\beta_n}$, such that
\[
    L_n := \max_{i = 1, \dots, k} \max_{X'\neq X'' \in \overline{B^+}} \frac{|(\eta v_{i,n})(X')-(\eta v_{i,n})(X'')|}{|X'-X''|^{\alpha}} \rightarrow \infty,
\]
for some $\alpha \in (0,\nuACF)$, which from now on we will consider as fixed. By Remark
\ref{rem: troianiello} we readily infer that
$\beta_n \rightarrow \infty$. Moreover, up to a relabelling, we may assume that $L_n$ is achieved
by $i = 1$ and a sequence of points $(X'_n, X''_n) \in \overline{B^+} \times
\overline{B^+}$. We start showing some first properties of such sequences.
\begin{lemma}\label{lem: acc non a part+}
Let $X'_n \neq X''_n$ and $r_n := |X'_n-X''_n|$ satisfy
\begin{equation*}
    L_n = \frac{|(\eta v_{1,n})(X'_n)-(\eta v_{i,n})(X''_n)|}{r_n^{\alpha}}.
\end{equation*}
Then, as $n\to\infty$,
\begin{enumerate}
 \item $r_n \rightarrow 0$;
 \item $\dfrac{\dist(X'_n,\partial^+ B^+)}{r_n}\to \infty$,
 $\dfrac{\dist(X''_n,\partial^+ B^+)}{r_n}\to \infty$.
\end{enumerate}
\end{lemma}
\begin{proof}
By the uniform control on $\|\bv_n\|_{L^\infty}$ we have
\[
L_n\leq \frac{\bar m}{r_n^\alpha}\left(\eta(X'_n) + \eta(X''_n)\right),
\]
which immediately implies $r_n \to 0$. Since $\eta$ vanishes on $\partial^+B^+$, we have that, for every $X\in\overline{B^+}$, it holds
\[
\eta(X) \leq \ell \dist(X,\partial^+ B^+),
\]
where $\ell$ denotes the Lipschitz constant of $\eta$.
As a consequence, the first inequality becomes
\[
\dfrac{\dist(X'_n,\partial^+ B^+)}{r_n} + \dfrac{\dist(X''_n,\partial^+ B^+)}{r_n}
\geq \frac{L_n r_n^{\alpha-1}}{\bar m\ell}\to\infty
\]
(recall that $\alpha<1$), and the lemma follows by recalling that $\dist(X'_n,X''_n)=r_n$.
\end{proof}
Our analysis is based on two different blow up sequences, one having uniformly bounded
H\"older quotient, the other satisfying a suitable problem. Let $\{P_n\}_{n\in\N}\subset
\overline{B^+}$, $|P_n|<1$, be a sequence of points, to be chosen later. We write
\[
    \tau_n B^+ := \frac{B^+ - P_n}{r_n},
\]
remarking that $\tau_n B^+$ is a hemisphere, not necessarily centered on the hyperplane
 $\{y=0\}$. We introduce the sequences
\[
    w_{i,n}(X) := \eta(P_n) \frac{v_{i,n}(P_n + r_n X)}{L_n r_n^{\alpha}} \quad \text{and} \quad \bar{w}_{i,n}(X) := \frac{(\eta v_{i,n})(P_n + r_n X)}{L_n r_n^{\alpha}}, 
\]
where $X \in \tau_n B^{+}$. With this choice, on one hand it follows immediately that, for every $i$
\begin{align*}
    \max_{X'\neq X'' \in \overline{\tau_n B^{+}} } \frac{|\bar{w}_{i,n}(X')-\bar{w}_{i,n}(X'')|}{
    |X'-X''|^{\alpha}} \leq & \left|\bar{w}_{1,n}\left(\frac{X'_n-P_n}{r_n}\right)-
    \bar{w}_{1,n}\left(\frac{X''_n-P_n}{r_n}\right)\right| = 1,
\end{align*}
in such a way that the functions $\{\bar{\bw}_n\}_{n \in \N}$ share an uniform bound on
H\"older seminorm, and at least their first components are not constant.
On the other hand, since $\eta(P_n)>0$,
each $\bw_{n}$ solves
\begin{equation}\label{eqn: w_sol}
    \begin{cases}
    - \Delta w_{i,n} = 0 & \text{ in }\tau_nB^{+}\\
    \partial_{\nu} w_{i,n} = f_{i,n}(w_{i,n}) - M_n w_{i,n} \tsum_{j \neq i}
    w_{j,n}^2 & \text{ on } \tau_n \partial^0 B^+,
    \end{cases}
\end{equation}
with $f_{i,n}(s) = \eta(P_n) r_n^{1-\alpha} L_n^{-1} f_{i,\beta_n}(L_n
r_n^{\alpha} s/\eta(P_n))$ and $M_n = \beta_n L_n^2 r_n^{2\alpha + 1}/\eta(P_n)^2$.
\begin{remark}\label{rem:f_i to 0}
The uniform bound of $\| \bv_{\beta} \|_{L^{\infty}}$
imply that
\[
\sup_{\tau_n \partial^0 B^+}|f_{i,n}(w_{i,n})| = \eta(P_n) r_n^{1-\alpha} L_n^{-1}
\sup_{\partial^0 B^+}|f_{i,\beta_n}
\left(v_{i,n}\right)|\leq C(\bar m)r_n^{1-\alpha} L_n^{-1}\to 0
\]
as $n\to\infty$.
\end{remark}
A crucial property is that the two blow up sequences defined above have
asymptotically equivalent behavior, as enlighten in the following lemma.
\begin{lemma}\label{lem:_uniform_convergence_w_and_bar_w}
Let $K\subset\R^{N+1}$ be compact. Then
\begin{enumerate}
  \item  $\displaystyle\max_{X \in K\cap\overline{\tau_n B^+}} | \bw_{n}(X)- \bar{\bw}_{n}(X)| \to 0$;
   \item there exists $C$, only depending on $K$, such that $|\bw_{n} (X)- \bw_{n}(0)| \leq C$, for every $x\in K$.
\end{enumerate}
\end{lemma}
\begin{proof}
Again, this is a consequence of the Lipschitz continuity of $\eta$ and of the uniform boundedness of $\{\bv_{\beta}\}_{\beta}$. Indeed we have, for every $i = 1, \dots, k$,
\[
    | w_{i,n}(X)- \bar{w}_{i,n}(X)| \leq \bar m r_n^{- \alpha} L_n^{-1} | \eta(X_n + r_n X)- \eta(X_n)| \leq \ell\bar m r_n^{1- \alpha} L_n^{-1} |X|
\]
and the right hand side vanishes in $n$, implying the first part. Moreover, by definition,
$\bw_{n} (0)= \bar \bw_{n}(0)$, and
$|\bar{\bw}_{n} (X)- \bar{\bw}_{n}(0)| \leq C |X|^{\alpha}$ for every $X\in \tau_n B^+$. But then we can conclude noticing that
\[
|\bw_{n} (X)- \bw_{n}(0)|\leq |\bw_{n} (X)- \bar \bw_{n}(X)| + |\bar \bw_{n} (X)- \bar \bw_{n}(0)|
\]
and applying the first part.
\end{proof}
\begin{lemma}\label{lem:_uniform_convergence_of W}
Let, up to subsequences, $\Omega_{\infty} := \lim\tau_n B^+$ and let
\[
    \bW_n(X) : = \bw_{n}(X)-\bw_{n}(0) \quad \text{ and } \quad \bar\bW_n(X) : = \bar\bw_{n}(X)-\bar\bw_{n}(0).
\]
Then there exists a function $\bW \in \C^{0,\alpha}(\Omega_{\infty})$ which is harmonic and such that $\bW_n \to \bW$ and $\bar\bW_n \to \bW$ uniformly in every compact set $K\subset \Omega_{\infty}$. Moreover, if we choose $\{P_n\}_{n \in \N}$ such that $|X'_n-P_n| < Cr_n$ for some constant $C$ and for every $n$, then $\bW$ is non constant.
\end{lemma}
\begin{proof}
Let $K \subset \Omega_{\infty}$ be any fixed compact set. Then, by definition, $K$ is contained in
the half sphere $\tau_n B^+$, for every $n$ sufficiently large.
By definition, $\{\bar\bW_n\}_{n\in\N}$ is a sequence of functions
which share the
same $\C^{0,\alpha}$-seminorm and are uniformly bounded in $K$, since $\bar\bW_n(0) = 0$. By the
Ascoli-Arzel\`a theorem, there exists a function $\bW \in C(K)$ which, up to a subsequence, is the
uniform limit of $\{\bar\bW_n\}_{n\in\N}$: taking a countable compact exhaustion of
$\Omega_{\infty}$ we
find that $\bar\bW_n \to \bW$ uniformly in every compact set. By Lemma
\ref{lem:_uniform_convergence_w_and_bar_w}, we also find that $\bW_n \to \bW$ and, since the
uniform limit of harmonic function is harmonic, we conclude that $\bW$ is harmonic. Let $X, Y \in
\Omega_{\infty}$ be any pair of points. By definition, there exists $n_0 \in \N$ such that $X, Y
\in \tau_n B^+$ for every $n \geq n_0$, and so
\[
    |\bar \bW_{n}(X) - \bar\bW_{n}(Y)| \leq \sqrt{k} |X-Y|^\alpha \quad \text{ for every } n \geq n_0.
\]
Passing to the limit in $n$ the previous expression, we obtain $\bW \in \C^{0,\alpha}(\Omega_{\infty})$. Let now $C>0$ be fixed, and let us choose $\{P_n\}_{n \in \N}$ be such that $|X'_n-P_n| < Cr_n$. It follows that, up to a subsequence,
\[
    \frac{X'_n-P_n}{r_n} \to X' \quad \text{and} \quad \frac{X''_n-P_n}{r_n} \to X'',
\]
where $X', X'' \in \overline{B_{C+1} \cap \Omega_{\infty}}$. Therefore, by equicontinuity and uniform convergence,
\[
    \left|\bar W_{1,n}\left(\frac{X'_n-P_n}{r_n}\right) - \bar W_{1,n}\left(\frac{X''_n-P_n}{r_n}\right)\right| = 1 \implies |W_{1}(X') - W_{1}(X'')| = 1
\]
and the lemma follows.
\end{proof}
In Lemma \ref{lem: acc non a part+} we have shown that $X'_n$, $X''_n$ can not accumulate
too fast towards $\partial^+B^+$. Now we can prove that they converge to $\partial^0B^+$.
\begin{lemma}\label{lem: shift up_local}
There exists $C>0$ such that, for every $n$ sufficiently large,
\[
\frac{\dist(X'_n, \partial^0 B^+) + \dist(X''_n, \partial^0 B^+)}{r_n} \leq C.
\]
\end{lemma}
\begin{proof}
We argue by contradiction. Taking into account the second part of Lemma \ref{lem: acc non a part+},
this forces
\[
    \frac{\dist(X'_n, \partial B^+) + \dist(X''_n, \partial B^+)}{r_n} \to \infty.
\]
Choosing $P_n=X'_n$ in the definition of $\bw_n$, $\bar\bw_n$, we can apply Lemma
\ref{lem:_uniform_convergence_of W}. First of all, we notice that $\tau_n B^+ \to
\Omega_\infty=\R^{N+1}$. But then $\bW$ as in the aforementioned lemma is harmonic,
globally H\"older continuous on $\R^{N+1}$ and non constant, in contradiction with
Liouville theorem.
\end{proof}
We are in a position to choose $P_n$ in the definition of $\bw_n$, $\bar\bw_n$:
from now one let us define
\[
P_n:=(x'_n,0),
\]
where as usual $X'_n=(x'_n,y'_n)$. With this choice, it is immediate to see that
$\tau_nB^+\to\Omega_\infty = \R^{N+1}_+$, and that all the above results, and in particular
Lemma \ref{lem:_uniform_convergence_of W}, apply. This last fact follows from Lemma
\ref{lem: shift up_local}, since
\[
Cr_n \geq \dist(X'_n, \partial^0 B^+) = |X'_n-P_n|.
\]
Our next aim is to prove that $\left\{\bw_n\right\}_{n\in\N}$, $\left\{\bar\bw_n\right\}_{n\in\N}$ are uniformly bounded. This will be done
by contradiction, in a series of lemmas.
\begin{lemma}\label{lem: bound in Mn not infinitesimal local}
    Under the previous blow up configuration, if $\bar{w}_{i,n}(0) \to \infty$ for some $i$, then
\[
    M_n w_{i,n}^2(0) = M_n \bar{w}_{i,n}^2(0) \leq C
\]
for a constant $C$ independent of $n$. In particular, $M_n\to0$.
\end{lemma}
\begin{proof}
Let $r > 0$ be fixed, and let $B_{2r}^+$ be the half ball of radius $2r$: by Lemma \ref{lem: shift up_local}, for $n$ sufficiently large we have that $B_{2r}^+ \subset \tau_nB^+$. Since the sequence $\{\bar{\bw}_n\}_{n\in \N}$ is made of continuous functions which share the same $\C^{0,\alpha}$-seminorm, we have that $\inf_{B_{2r}^+} |\bar w_{i,n}| \to \infty$. Furthermore, by Lemma \ref{lem:_uniform_convergence_w_and_bar_w} $\inf_{B_{2r}^+} |w_{i,n}| \to \infty$ as well.

Proceeding by contradiction, we assume that
\[
    I_{n} := \inf_{\partial^0 B_{2r}^+ } M_n w_{i,n}^2 \rightarrow \infty.
\]
We first show that for $j \neq i$, both the sequence $\{ w_{j,n} \}_{n \in \N}$ and $\{ \bar{w}_{j,n} \}_{n \in \N}$ are bounded in $B_{2r}^+$. We recall that $|w_{j,n}|$ is a subsolution of problem
\eqref{eqn: w_sol}. More precisely, by Lemma \ref{lem: measure inequality}, we have that
\begin{equation}\label{eqn: inequality system of j neq 1 local}
\int\limits_{B^+_{2r}} \nabla |w_{j,n}| \cdot \nabla \varphi \, \de{x} \de{y}
-  \int \limits_{\partial^0 B_{2r}^+}\left(\|f_{j,n}\|_{L^{\infty}(B_{2r})} - I_n |w_{j,n}|
\right)\varphi  \, \de{x}\leq0,
\end{equation}
for every $\varphi \in H^{1}_0(B_{2r})$, $\varphi\geq0$.
Letting $\eta \in \C_0^{\infty}(B_{2r})$, we can choose $\varphi=\eta^2 |w_{j,n}|$ in the
above equation, obtaining
\begin{multline*}
    \int\limits_{B^+_{2r}} \left(|\nabla (\eta |w_{j,n}|)|^2 - |\nabla \eta|^2 |w_{j,n}|^2\right) \, \de{x} \de{y} + I_n \int\limits_{\partial^0 B^+_{2r}} \eta^2 |w_{j,n}|^2 \, \de{x} \\
    \leq \|f_{j,n}\|_{L^{\infty}} \int \limits_{\partial^0 B_{2r}^+} \eta ^2|w_{j,n}| \, \de{x}.
\end{multline*}
As a consequence
\begin{multline}\label{eqn: upperbound Mn not 0 local}
    I_n \int\limits_{\partial^0 B^+_{2r}} \eta^2 |w_{j,n}|^2 \de{x} \leq
\int\limits_{B^+_{2r}} |\nabla \eta|^2 |w_{j,n}|^2 \de{x} \de{y} +
\|f_{j,n}\|_{L^{\infty}} \int \limits_{\partial^0 B_{2r}^+} \eta ^2|w_{j,n}| \, \de{x}\\
 \leq \sup_{B^+_{2r}}
|w_{j,n}|^2 \int\limits_{B^+_{2r}} |\nabla \eta|^2\de{x} \de{y} + \|f_{j,n}\|_{L^{\infty}} \int\limits_{\partial^0 B_{2r}^+} \eta ^2 \frac12 \left( 1+ |w_{j,n}|^2\right) \, \de{x}\\
\leq \sup_{B^+_{2r}} |w_{j,n}|^2 \left( \int\limits_{B^+_{2r}} |\nabla \eta|^2\de{x} \de{y} + C(r) \|f_{j,n}\|_{L^{\infty}} \right)+ C(r) \|f_{j,n}\|_{L^{\infty}},
\end{multline}
where, by Remark \ref{rem:f_i to 0}, $C(r) \|f_{j,n}\|_{L^{\infty}} \to 0$. On the other hand,
using again the uniform H\"older bounds of the sequence $\{\bar{\bw}_n\}_{n \in \N}$ and the uniform control given by Lemma \ref{lem:_uniform_convergence_w_and_bar_w}, we infer
\begin{equation}\label{eqn: lowerbound Mn not 0 local}
\begin{split}    
I_n \int\limits_{\partial^0 B^+_{2r}} \eta^2 |w_{j,n}|^2 \de{x} &\geq I_n
\inf_{\partial^0 B_{2r}^+} |w_{j,n}|^2 \int\limits_{\partial^0 B^+_{2r}}
\eta^2\de{x} \\&\geq  C I_n \left(\sup_{B^+_{2r}} |w_{j,n}|^2 - (2r)^{2\alpha}
\right) \int\limits_{\partial^0 B^+_{2r}} \eta^2\de{x}.
\end{split}
\end{equation}
Combining \eqref{eqn: upperbound Mn not 0 local} with \eqref{eqn: lowerbound Mn not 0 local}
we deduce the uniform boundedness of $\sup_{\partial^+ B^+_{2r}} |w_{j,n}|$ for
$j \neq i$. Equation \eqref{eqn: inequality system of j neq 1 local} fits into (the variational
counterpart of) Lemma \ref{eqn: decay with perturbations}, which implies
\begin{equation}\label{eqn: decay supersol}
    |w_{j,n}| \leq \frac{C}{I_n} \sup_{\partial^+ B_{2r}^+} |w_{j,n}|  \qquad \text{ on } \partial^0 B_r^+
\end{equation}
for a constant $C$ independent of $n$. From the uniform bound it follows that
$w_{j,n} \rightarrow 0$ uniformly in $\partial^0 B^+_{r}$ for every $r>0$, and the same is true for $\bar{w}_{j,n}$, $j\neq i$: in
particular, since $|\bar w_{1,n}(\tau_nX'_n)-\bar w_{1,n}(\tau_nX''_n)| = 1$, we deduce that, necessarily, $i = 1$.

Now, $w_{1,n}$ satisfies
\[
\int\limits_{B^+_{r}} \nabla w_{1,n} \cdot \nabla \varphi \, \de{x} \de{y}
=  \int \limits_{\partial^0 B_{r}^+}\left(f_{1,n} - M_n w_{1,n} \tsum_{j \neq 1} w_{j,n}^2
\right)\varphi  \, \de{x}
\]
for every $\varphi \in H^{1}_0(B_{r})$. From the previous estimates and
the definition of $I_n$ we find
\[
\begin{split}
    \left|f_{1,n} - M_n w_{1,n} \tsum_{j \neq 1} w_{j,n}^2\right|
    &\leq \|f_{1,n}\|_{L^\infty} + M_n (|w_{1,n}|^2 + 1) \tsum_{j \neq 1} w_{j,n}^2\\
    &\leq \|f_{1,n}\|_{L^\infty} +C\frac{I_n + M_n (r^{2\alpha} + 1) }{I_n^2} \to 0
\end{split}
\]
on $\partial^0 B^+_r$, and this holds for every $r$. As a consequence, we can define
$\{\bW_n\}_{n \in \N}$ as in Lemma \ref{lem:_uniform_convergence_of W}, obtaining that
$W_{1,n}$ converges to $W_1$, which is a nonconstant, globally H\"older continuous function
on $\R^{N+1}_+$, which satisfies
\[
    \begin{cases}
    - \Delta W_1 = 0 &\text{ in } \R^{N+1}_+ \\
    \partial_{\nu} W_1 = 0 &\text{ on } \R^{N}.
    \end{cases}
\]
But then the even extension of $W_1$ through $\{y=0\}$ contradicts the Liouville theorem.
\end{proof}

\begin{lemma}\label{lem: bound when wi is unbounded local}
Under the previous blow up setting, if there exists $i$ such that $\bar{w}_{i,n}(0) \to \infty$, then
\[
   M_n |\bar{w}_{i,n}(0)| \tsum_{j \neq i} \bar{w}_{j,n}^2(0) \leq C
\]
for a constant $C$ independent of $n$.
\end{lemma}
\begin{proof}
Let $r>1$ be any fixed radius. Multiplying equation \eqref{eqn: w_sol} by $w_{i,n}$ and integrating on $B_r^+$ we obtain the
identity
\[
    \int\limits_{B_r^+} |\nabla w_{i,n}|^2 \, \de{x} \de{y} + \int\limits_{\partial^0
B_r^+}\left( - f_{i,n}w_{i,n} + M_n w_{i,n}^2 \tsum_{j \neq i} w_{j,n}^2 \right)
\de{x} =  \int\limits_{\partial^+ B_r^+} w_{i,n} \partial_{\nu} w_{i,n} \, \de{\sigma}.
\]
Defining
\[
    E_{i}(r) := \frac{1}{r^{N-1} } \left( \int\limits_{B_r^+} |\nabla w_{i,n}|^2 \,
\de{x} \de{y} + \int\limits_{\partial^0 B_r^+}\left( - f_{i,n}w_{i,n} + M_n w_{i,n}^2
\tsum_{j \neq i} w_{j,n}^2\right) \de{x} \right)
\]
and
\[
    H_i(r) := \frac{1}{r^{N}} \int\limits_{\partial^+ B_r^+} w_{i,n}^2 \de{\sigma},
\]
a straightforward computation shows that $H_i \in AC(r/2,r)$ and
\[
    H_i'(r) = \frac{2}{r} E_{i}(r).
\]
In particular, integrating from $r/2$ to $r$, we obtain the following identity
\begin{equation}\label{eqn: identity on H and E local}
    H_i(r) - H_i\left(\frac{r}{2}\right) = \int\limits_{r/2}^r \frac{2}{s} E_{i}(s)
\, \mathrm{d}s.
\end{equation}
If $r$ is suitably chosen, and $n$ is large, after a scaling in the definition of $H_i$, we have that the left hand side of \eqref{eqn: identity on H and E local} writes as
\[
    \begin{split}
    H_i(r) - H_i\left(\frac{r}{2}\right) &= \int\limits_{\partial^+ B^+} \left[ w_{i,n}^2(rx) - w_{i,n}^2\left(\frac{r}{2}x\right) \right] \de{\sigma} \\
    &= \int\limits_{\partial^+ B^+} \left[ w_{i,n}(rx) - w_{i,n}\left(\frac{r}{2}x\right) \right] \left[ w_{i,n}(rx) + w_{i,n}\left(\frac{r}{2}x\right) \right] \de{\sigma}\\
    &\leq C(r) (|w_{i,n}(0)| +1),
    \end{split}
\]
where we used the first part of Lemma \ref{lem:_uniform_convergence_w_and_bar_w} to estimate the difference in the integral above, and the second part of the same lemma for the
sum. In a similar way, we obtain a lower bound of the right hand side of equation \eqref{eqn: identity on H and E local}
\begin{multline*}
    \frac{1}{r}\int\limits_{r/2}^{r} \frac{2}{s} E_{i}(s)  \mathrm{d}s \geq \min_{s \in [r/2,r]} \frac1s E_{i}(s) \\
    \geq M_n \min_{s \in [r/2,r]} \frac{1}{s^N}\int\limits_{\partial^0 B_s^+} \tsum_{j \neq i} w_{i,n}^2 w_{j,n}^2 \, \de{x} - \max_{s \in [r/2,r]} \frac{1}{s^{N}} \int\limits_{\partial^0 B_s^+}  |f_{i,n}w_{i,n}| \, \de{x}  \\
    \geq M_n C(r)\left( \tsum_{j \neq i} w_{i,n}^2(0) w_{j,n}^2(0) - 1\right) -  C(r) \|f_{j,n}\|_{L^{\infty}} (|w_{i,n}(0)| +1),
\end{multline*}
where $C(r) \|f_{j,n}\|_{L^{\infty}} \to 0$ as $n \to \infty$.
Putting the two estimates together and recalling that $M_n$ is bounded, we find
\[
    M_n  \tsum_{j \neq i} w_{i,n}^2(0) w_{j,n}^2(0) \leq C(r)( |w_{i,n}(0)| + 1).
\]
The conclusion follows dividing by $|w_{i,n}(0)|$, using the uniform control of the sequence $\{w_{i,n}\}_{n \in \N}$ and $\{\bar w_{i,n}\}_{n \in \N}$ and the assumption that $|\bar w_{i,n}(0)| \to \infty$.
\end{proof}
\begin{lemma}
If $\{\bar{\bw}_{n}(0)\}_{n \in \N}$ is unbounded, then also $\{\bar{w}_{1,n}(0)\}_{n \in \N}$ is.
\end{lemma}
\begin{proof}
Suppose, by contradiction, that $\{\bar{w}_{1,n}(0)\}_{n \in \N}$ is bounded while, by Lemma
\ref{lem: bound in Mn not infinitesimal local}, $M_n \to 0$. Then, reasoning
as in the proof of Lemma \ref{lem:_uniform_convergence_of W}, we have that both
$w_{1,n}$ and $\bar w_{1,n}$ converge to some $w_1$, uniformly on compact sets; moreover, $w_1$ is
harmonic, globally H\"older continuous and non constant. We claim that there
exists a constant $\lambda\geq0$ such that
\[
    \partial_{\nu} w_{1,n} = f_{1,n} - M_n  w_{1,n}\tsum_{i\neq 1}  w_{i,n}^2 \to - \lambda w_1
\]
uniformly in every compact set. This would contradict Proposition \ref{prp: global eigenfunction},
proving the lemma.

To prove the claim:
\begin{itemize}
    \item let $i \neq 1$ be an index such that $\bar{w}_{i,n}(0)$ is unbounded: from Lemma 
    \ref{lem: bound when wi is unbounded local} we see that $M_n \bar w_{i,n}^2(0)$ is bounded. Moreover, by uniform H\"older bounds,
\[
 M_n\left|
\bar{w}_{i,n}^2(x,0) - \bar{w}_{i,n}^2(0,0)\right| \leq M_n \osc_{\partial^0 B^+_r} \bar{w}_{i,n}^2 \to 0,
\]
implying $M_n\bar{w}_{i,n}^2(x,0)\to\lambda_i\geq0$;
\item let now $j$ be an index such that $\bar{w}_{j,n}(0)$ is bounded. Then, again by uniform convergence,
\[
    M_n \bar w_{1,n} \bar w_{j,n}^2  \to 0
\]
uniformly in every compact set.
\end{itemize}
It follows that
\[
     f_{1,n} - M_n \bar w_{1,n}\tsum_{i\neq 1} \bar w_{i,n}^2 \to - \lambda w_1
\]
uniformly in every compact set, and the same limit holds for $\{w_{1,n}\}_{n \in
\N}$ by uniform convergence.
\end{proof}
%
%
%
\begin{lemma}\label{lem: w(0) bdd}
The sequence  $\{\bar{\bw}_{n}(0)\}_{n \in \N}$ is bounded.
\end{lemma}
\begin{proof}
By contradiction, let $\{\bar{\bw}_{n}(0)\}_{n \in \N}$ be unbounded. Then, by the above lemmas, $M_n\to0$, $\{\bar w_{1,n}(0)\}_{n \in \N}$ is unbounded, and
\[
M_n |\bar{w}_{1,n}(0)| \tsum_{i \neq 1} \bar{w}_{i,n}^2(0) \leq C,
\]
independent of $n$. Taking into account the uniform bound on
$\osc_{\partial^0 B^+_r} \bar{w}_{1,n} \bar{w}_{i,n}^2$, we deduce the existence of
a constant $\lambda$ such that, at least up to a subsequence,
\[
  f_{1,n}  -M_n \bar w_{1,n} \tsum_{h \neq 1} \bar w_{h,n}^2 \to \lambda
\]
uniformly on every compact subset of $\R^N$, and the same holds true for the sequence $\{w_{1,n}\}_{n \in \N}$. Thus, as usual, $W_{1,n} = w_{1,n}-w_{1,n}(0)$ converges to $W_1$, a nonconstant, globally H\"older continuous solution to
\[
    \begin{cases}
    - \Delta W_1= 0 &\text{ in } \R^{N+1}_+ \\
    \partial_{\nu} W_1 = \lambda &\text{ on } \R^N.
    \end{cases}
\]
Appealing to Proposition \ref{prp: prescribed constant normal derivative}, we obtain a contradiction.
\end{proof}
The uniform bound on $\{\bar{\bw}_{n}(0)\}_{n \in \N}$ allows to prove the following
convergence result.
\begin{lemma}\label{lem: uniform implies strong convergence local}
Under the previous blow up setting, there exists $\bw \in
(H^1_{\loc}\cap \C^{0,\alpha})\left(\overline{\R^{N+1}_+}\right)$ such that, up to a subsequence,
\[
	\bw_{n} \to \bw \text{ in }(H^1\cap C)(K)
\]
for every compact $K\subset\overline{\R^{N+1}_+}$
.
\end{lemma}
\begin{proof}
Reasoning as in the proof of Lemma \ref{lem:_uniform_convergence_of W}
we can easily obtain that,
up to subsequences, both $\{\bar{\bw}_{n}\}_{n \in \N}$ and $\{\bw_{n}\}_{n \in \N}$ converge uniformly on compact sets to the same limit $\bw \in \C^{0,\alpha}\left(\overline{\R^{N+1}_+}\right)$, hence we are left to show the $H^1_{\loc}$ convergence
of the latter sequence.

Let $K$ be compact, $r$ be such that $K\subset \overline{B_r^+}$, and let us consider
$\eta \in \C^{\infty}_0(B_r^+)$ any smooth cutoff function, such that $0\leq \eta\leq 1$ and
$\eta\equiv 1$ on $K$. Testing the equation for $w_{i,n}$ by $w_{i,n}
\eta^2$, we obtain
\begin{multline*}
    0 \leq \int \limits_{K} |\nabla w_{i,n}|^2  \de{x} \de{y} + M_n \int \limits_{\partial^0 K} w_{i,n}^2 \tsum_{j \neq i} w_{j,n}^2  \de{x} \\
    \leq \int \limits_{B_r^+} |\nabla w_{i,n}|^2 \eta^2 \de{x} \de{y} + M_n \int \limits_{\partial^0 {B_r^+}} w_{i,n}^2 \tsum_{j \neq i} w_{j,n}^2 \eta^2 \de{x} \\
    \leq \frac{1}{2} \int \limits_{B_r^+} w_{i,n}^2 |\Delta \eta^2| \de{x} \de{y} + \frac12\int \limits_{\partial^0 {B_r^+}} \left(w_{i,n}^2 |\partial_{\nu} \eta^2| + f_{i,n} w_{i,n} \eta^2\right) \de{x}.
\end{multline*}
Since the right hand side is bounded uniformly in $n$ (recall Lemmas \ref{lem: w(0) bdd} and \ref{lem:_uniform_convergence_w_and_bar_w}), we deduce that, up to subsequences,
$\{\bw_n \}_{n \in\N}$ weakly converges in $H^1(K)$. Since this holds for every $K$, we deduce that $\bw_n \rightharpoonup \bw$ in $H^1_{\loc}\left(\overline{\R^{N+1}_+}\right)$. To prove the strong convergence,
let us now test the equation by $\eta^2 (w_{i,n} - w_i)$. We obtain
\begin{equation}\label{eqn: strong_conv_blowup}
    \int \limits_{{B_r^+}} \nabla w_{i,n} \cdot \nabla\left[ \eta^2 (w_{i,n} - w_i) \right] \,
\de{x} \de{y} = \int \limits_{\partial^0 {B_r^+}} \eta^2
(w_{i,n}-w_{i})\partial_{\nu} w_{i,n} \, \de{x}.
\end{equation}
We can estimate the right hand side as
\begin{multline*}
    \int \limits_{\partial^0 {B_r^+}} \eta^2  (w_{i,n}-w_{i})\partial_{\nu} w_{i,n} \,\de{x} \\
    \leq \sup_{x \in {B_r^+}} | w_{i,n}-w_{i}|  \int \limits_{\partial^0 B^+_r} \eta^2 \left[M_n |w_{i,n}| \tsum_{i,j < i} w_{j,n}^2   + |f_{i,n}| \right] \, \de{x}\\
    \leq  C(r) \sup_{x \in {B_r^+}} | w_{i,n}-w_{i}|,
\end{multline*}
where the last step holds since the inequality for $|w_{i,n}|$
(Lemma \ref{lem: measure inequality}) tested by $\eta^2$ yields
\begin{multline*}
    \int \limits_{\partial^0 B^+_r} \eta^2  M_n |w_{i,n}| \tsum_{i,j < i} w_{j,n}^2
   \, \de{x}
    \\ \leq
\int \limits_{\partial^0 {B_r^+}}\left( |f_{i,n}| \eta^2 + |w_{i,n} \partial_{\nu} \eta^2| \right)
\de{x} +
\int \limits_{B_r^+} |w_{i,n} \Delta \eta^2| \,\de{x} \de{y}
 \leq
C(r).
\end{multline*}

Resuming, equation \eqref{eqn: strong_conv_blowup} implies
\[\begin{split}
\int \limits_{{B_r^+}} \left| \nabla(\eta w_{i,n}) \right|^2 \,\de{x} \de{y} \leq & \int \limits_{{B_r^+}}
\left( \eta^2 \nabla w_{i,n} \cdot \nabla w_{i} + 2\eta w_i \nabla w_{i,n} \cdot
\nabla \eta+ |\nabla \eta|^2 w_i^2 \right) \, \de{x} \de{y}  \\ &+ C(r) \sup_{x \in
{B_r^+}} | w_{i,n}-w_{i}|.
\end{split}
\]
Using both the weak $H^1$ convergence and the uniform one, we have that
$$
\limsup_{n \rightarrow \infty} \int \limits_{{B_r^+}} \left| \nabla(\eta w_{i,n}) \right|^2 \,
\de{x} \de{y} \leq \int \limits_{{B_r^+}} \left| \nabla(\eta w_{i}) \right|^2 \, \de{x} \de{y}
$$
and we conclude the strong convergence in $H^1({B_r^+})$ of $\{\eta \bw_{n}\}_{n
\in \N}$ to $\eta \bw$, that is, since $\eta$ was arbitrary, the strong
convergence of $\bw_{n}$ to $\bw$ in $H^1_{\loc}\left(\overline{\R^{N+1}_+}\right)$.
\end{proof}
\begin{proof}[End of the proof of Theorem \ref{thm:_local_holder}]
Summing up, we have that $\bw_{n} \to \bw$ in $(H^1\cap C)_{\loc}$,
and that the limiting blow up profile $\bw$ is a nonconstant vector
of harmonic, globally H\"older continuous functions. To reach the final
contradiction, we distinguish, up to subsequences, between
the following three cases.

\textbf{Case 1:} $M_n\to0$. In this case also the equation on the
boundary passes to the limit, and the nonconstant component
$w_1$ satisfies $\partial_{\nu} w_{1} \equiv 0$ on $\R^{N}$,
so that its even extension through $\{y=0\}$ contradicts Liouville theorem.

\textbf{Case 2:} $M_n\to C>0$. Even in this case the equation on the boundary passes to the limit, and
$\bw$ solves
\[
    \begin{cases}
    - \Delta w_{i} = 0 & x \in \R^{N+1}_+\\
    \partial_{\nu} w_{i} = - C w_i \tsum_{j \neq i} w_j^2 & \text{on } \R^{N}
\times \{0\}
    \end{cases}
\]
The contradiction is now reached using Proposition \ref{prp: liouville
system}, since $\bw \in \classG_c \cap \C^{0,\alpha}(\R^{N+1}_+)$ and $\alpha < \nuACF$.

\textbf{Case 3:} $M_n\to \infty$. By Proposition \ref{prp: approx classG_s}, we infer
$\bw \in \classG_s \cap \C^{0,\alpha}(\R^{N+1}_+)$ with $\alpha < \nuACF$, in contradiction with
Proposition \ref{prp: liouville inequalities}.

As of now, the contradictions we have obtained imply that $\{\bv_\beta\}_{\beta>0}$ is
uniformly bounded in $\C^{0,\alpha}\left(\overline{B^+_{1/2}}\right)$, for every $\alpha<\nuACF$. But
then the relative compactness in $\C^{0,\alpha}\left(\overline{B^+_{1/2}}\right)$ follows by
Ascoli-Arzel\`a Theorem, while the one in $H^1({B^+_{1/2}})$ can be shown by reasoning as
in the proof of Lemma \ref{lem: uniform implies strong convergence local}.
\end{proof}
\begin{remark}\label{rem: basta gamma<1}
It is worthwhile noticing that, in proving Theorem \ref{thm:_local_holder}, the only part in
which we used the assumption $\alpha<\nuACF$ is the concluding argument, while in the rest of
the proof it is sufficient to suppose $\alpha<1$.
\end{remark}
As we mentioned, even though we are not able to show that $\nuACF=1/2$, nonetheless we will prove
that the uniform H\"older bound hold for any $\alpha<1/2$. In view of the previous remark, this
can be done by means of some sharper Liouville type results, which will be obtained in the next section.
To conclude the present discussion, we observe that a result analogous to Theorem \ref{thm:_local_holder}
holds, when entire profiles of segregation are considered, instead of solutions to $\problem{\beta}$.

\begin{proposition}\label{prp: compactness in Gs}
Let $\{\bv_n\}_{n \in \N}$ be a subset of $\classG_s\cap \C^{0,\alpha}\left(\overline{B^+_1}\right)$,
for some $0<\alpha\leq\nuACF$, such that
\[
    \| \bv_{n} \|_{L^{\infty}(B^+_1)} \leq \bar m,
\]
with $\bar m$ independent of $n$. Then for every $\alpha' \in (0,\alpha)$ there exists a constant
$C = C(\bar m,\alpha')$, not depending on $n$, such that
\[
    \| \bv_n\|_{\C^{0,\alpha'}\left(\overline{B^+_{1/2}}\right)} \leq C.
\]
Furthermore, $\{\bv_n\}_{n\in\N}$ is relatively compact in $H^1(B^+_{1/2}) \cap \C^{0,\alpha'}\left(\overline{B^+_{1/2}}\right)$ for every $\alpha' < \alpha$.
\end{proposition}
\begin{proof}
The proof follows the line of the one of Theorem \ref{thm:_local_holder}, being in fact easier, since we
do not have to handle any competition term. We proceed by contradiction, assuming that, up to subsequences,
\[
    L_n := \max_{i = 1, \dots, k} \sup_{X', X'' \in \overline{B^+}}
    \frac{|\eta(X') v_{i,n}(X')-\eta(X'') v_{i,n}(X'')|}{|X'-X''|^{\alpha'}} \rightarrow \infty,
\]
where again $\eta$ is a smooth cutoff function of the ball $B_{1/2}$ and $\alpha' < \alpha$. If $L_n$
is achieved by $(X'_n, X''_n)$, we introduce the sequences
\[
    w_{i,n}(X) := \eta(X_n) \frac{v_{i,n}(P_n + r_n X)}{L_n r_n^{\alpha'}} \quad \text{and} \quad \bar{w}_{i,n}(X) := \frac{(\eta v_{i,n})(P_n + r_n X)}{L_n r_n^{\alpha'}},
\]
for $X \in \tau_nB^{+}$, where, as usual, on one hand $\bar\bw_n$ has H\"older seminorm (and oscillation) equal to 1, while on the other hand
$\bw_n$ belongs to $\classG_s$.
All the preliminary properties of $(X'_n, X''_n)$,  up to Lemma \ref{lem: shift up_local}, are still valid, since they depend only on the harmonicity of $\{\bw_n\}_{n\in\N}$. It follows that the choice $P_n = (x_n',0)$ for every $n\in\N$ guarantees the convergence of the rescaled domains $\tau_nB^+$ to $\R^{N+1}_+$, while on any compact set the sequences $\{\bw_n\}_{n \in \N}$ and $\{\bar\bw_n\}_{n \in \N}$ shadow each other. Up to relabelling and subsequences, we are left with two alternatives:
\begin{enumerate}
\item either for any compact set $K\in\R^{N}$ we have $\bw_{1,n}(x,0) \neq 0$ for every $n \geq n_0(K)$ and $x \in K$;
\item or there exists a bounded sequence $\{x_n\}_{n\in\N}\subset\R^N$ such that $\bw_{n}(x_n,0) = 0$ for every $n$.
\end{enumerate}

In the first case, if we define $\bW_{n} = \bw_{n} - \bw_{n}(0)$ and $\bar\bW_{n} = \bar\bw_{n} - \bar\bw_{n}(0)$, we obtain that the sequence $\{\bar\bW_n\}_{n \in \N}$ is uniformly bounded in $\C^{0,\alpha'}$, and hence $\{\bW_n\}_{n \in \N}$ converges uniformly on compact sets
to a non constant, globally H\"older continuous function $\bW$, with $\partial_\nu W_{1}(x,0)\equiv0$
and $W_i(x,0)\equiv0$ for $i>1$, on $\R^N$. Extending properly the vector $\bW$ to the whole
$\R^{N+1}$, we find a contradiction with the Liouville theorem.

Coming to the second alternative, this time $\{\bw_n\}_{n \in \N}$ itself converges,
uniformly on compact sets, to a non constant, globally H\"older continuous function $\bw$.
We want to show that the convergence is also strong in $H^1_{\loc}$: this will imply that also
$\bw \in\classG_s$ (recall also the end of the proof of Proposition \ref{prp: approx classG_s}), in contradiction with Proposition \ref{prp: liouville inequalities}. To prove the strong convergence
let us consider, for any $i$, the even extension of $|w_{i,n}|$ through $\{y=0\}$, which we denote again with $|w_{i,n}|$. We have that there exists a non negative Radon measure $\mu_{i,n}$ such that
\[
    - \Delta |w_{i,n}| = - \mu_{i,n} \qquad\text{ in } \mathcal{D}'(\tau_nB):
\]
indeed, on one hand, for $X \in \{w_{i,n} \neq 0\}$, there exists a radius $r>0$ such that the
even extension of $w_{i,n}$ through $\{y=0\}$ is harmonic in $B_r(X)$, providing
\[
    |w_{i,n}|(X) \leq \frac{1}{|B_r|}\int\limits_{B_r(X)} |w_{i,n}|(Y) \de{Y};
\]
on the other hand $X \in  \{w_{i,n} = 0\}$ immediately implies the same inequality, and the consequent subharmonicity of $|w_{i,n}|$. At this point, we can reason as in \cite{tt}, showing that the
$L^\infty$ uniform bounds on compact sets of $|w_{i,n}|$ implies that the  measures $\mu_{i,n}$ are bounded on compact sets \cite[Lemma 3.7]{tt}; and that this, together with the uniform convergence of  $\{|\bw_n|\}_{n\in\N}$, implies its strong $H^1_\loc$-convergence \cite[Lemma 3.11]{tt}.

As a consequence of the previous contradiction argument, we deduce both the uniform
bounds and the pre-compactness of $\{\bv_n\}_{n\in\N}$ in $\C^{0,\alpha'}(B_{1/2})$.
Once we have (the uniform $L^\infty$ bounds and) the uniform convergence of $\{\bv_n\}_{n\in\N}$,
the strong $H^1$ pre-compactness can be obtained exactly as in the last part of the proof, replacing $|w_{i,n}|$ with $|v_{i,n}|$.
\end{proof}
\section{Liouville type theorems, reprise: the optimal growth}
In Section \ref{section:_uniform_local} we proved that non existence results of Liouville type imply uniform bounds in corresponding H\"older norms. This section is devoted to the study of the optimal Liouville exponents, which will allow to enhance the regularity estimates. Our aim is to prove the following result.
\begin{theorem}\label{thm: sharp Liouville}
Let $\nu\in\left(0,\frac12\right)$. If
\begin{enumerate}
 \item either $\bv\in\classG_s\cap \C^{0,\nu}\left(\overline{\R^{N+1}_+}\right)$,
 \item or $\bv\in\classG_c$ and $|\bv(X)| \leq C(1 + |X|^{\nu})$ for every $X$,
\end{enumerate}
then $\bv$ is constant.
\end{theorem}
The rest of the section is devoted to the proof of the above theorem. As of now, we already know by Propositions \ref{prp: liouville system} and \ref{prp: liouville
inequalities} that such theorem holds true whenever $\nu$ is smaller than $\nuACF$. In order to
refine such result, we will prove that it holds for $\nu$ smaller than $\nuLio$, according
to the following definition.
\begin{definition}
For $\nu>0$ and for every dimension $N$, we define the class
\[
\classH(\nu,N) := \left\{\bv \in \classG_s :
\begin{array}{l}
\bv \in \C^{0,\alpha}_{\loc}\left(\overline{\R^{N+1}_+}\right),\text{ for some }\alpha>0\\
\bv \text{ is non trivial 
and $\nu$-homogeneous}
\end{array}
\right\},
\]
and the critical value
\[
\begin{split}
    \nuLio(N) &= \inf\{\nu>0: \classH(\nu,N) \text{ is non empty} \}
\end{split}
\]
\end{definition}
\begin{remark}\label{rem:nuLio_leq_1}
Since $(y,0,\dots,0)\in\classH(1,N)$, for every $N$, we have that $\nuLio(N)\leq1$.
\end{remark}
\begin{remark}
By Corollary \ref{cor:4.5} we have that, if $\bv$ is non constant and
satisfies assumption (1) in Theorem \ref{thm: sharp Liouville} for some $\nu$,
then $\bv\in\classH(\nu,N)$.
\end{remark}
To prove Theorem \ref{thm: sharp Liouville}, we start by showing
that, given any non constant $\bv$ satisfying assumption (2) in Theorem
\ref{thm: sharp Liouville} for some $\nu$, we can construct a function
$\bar\bv\in\classH(\nu',N)$, for a suitable $\nu'\leq\nu$. This, together
with the previous remark, will imply the equivalence between
Theorem \ref{thm: sharp Liouville} and the inequality
\[
\nuLio(N)\geq\frac12.
\]

To construct such $\bar\bv$, we will use the blow down method. For any (non trivial)
$\bv\in\classG_c$ we denote with $N_{\bv}(x_0,r)$, $H_{\bv}(x_0,r)$ the
related quantities involved in the Almgren frequency
formula, defined in Section \ref{sec:almgren}. For any $r>0$, let us define
\[
    \bv_r(X) := \frac{1}{\sqrt{H_{\bv}(0,r)}} \bv(rX).
\]
Since $H$ is a strictly positive increasing function in $\R_+$ (recall Theorem
\ref{thm:_Almgren_for_classG_c}), $\bv_r$ is well defined. We have the following.
\begin{lemma}[Blow down method]\label{lem:_blowdown_compactness}
Let $\bv$ be a non constant function, satisfying assumption (2) in Theorem \ref{thm: sharp Liouville} for some $\nu$, and let
\[
0<\nu' = \lim_{r \to \infty} N_{\bv}(r)\leq\nu.
\]
Then there exists $\bar\bv \in \classH(\nu',N)$ such that,
for a suitable sequence $r_n\to\infty$,
\[
	\bv_{r_n} \to \bar\bv \text{ in }(H^1\cap C)(K)
\]
for every compact $K\subset\overline{\R^{N+1}_+}$.
\end{lemma}
\begin{proof}
First of all, by construction, we have that
\[
    \|\bv_r\|_{L^2(\partial^+ B^+)} = 1 \qquad \text{so that} \qquad \|v_{i,r}\|_{L^2(\partial^+ B^+)} \leq 1 \quad \text{for } i = 1,\dots,k.
\]
Each $\bv_r$ is solution to the system
\[
    \begin{cases}
    - \Delta v_{i,r} = 0 & \text{in } B^+\\
    \partial_{\nu} v_{i,r} + rH(r) v_{i,r} \tsum_{j \neq i} v_{j,r}^2 = 0& \text{on } \partial^0 B^+,
    \end{cases}
\]
where $r H(r) \to \infty$ monotonically as $r \to \infty$. Reasoning as in Lemma \ref{lem: measure inequality}, we have that the even reflection of $|v_{i,r}|$ through $\{y=0\}$ satisfies
\[
    \begin{cases}
    - \Delta |v_{i,r}| \leq 0 & \text{in } B\\
    \|v_{i,r}\|_{L^2(\partial B)} \leq \sqrt{2}.
    \end{cases}
\]
By the Poisson representation formula, it follows that there exists a constant $C$, not depending
on $r$, such that
\[
    \| \bv_r \|_{L^{\infty}(B_{3/4}^+)} \leq C
\]
for every $r$. Thus we are in a position to apply Theorem \ref{thm:_local_holder} in order to see
that the family $\{\bv_r\}_{r > 1}$ is relatively compact in $(H^1\cap \C^{0,\alpha})(B^+_{1/2})$,
for all $\alpha < \nuACF$. Furthermore, Proposition \ref{prp: approx classG_s} implies that
any limiting point of the family is an element of $\classG_s$ on $B^+_{1/2}$. In order to find a non trivial
limiting point,
we claim that there exists a sequence of radii $\{r_n\}_{n \in \N}$, $r_n \to \infty$, and a
positive constant $C$ such that
\[
    H(r_n) \leq C H(r_n/2) \qquad \forall r_n > 0.
\]
Indeed, we can argue by contradiction, assuming that there exists $r_0 > 0$ such that
\[
    H(r) \geq 3^{2\nu} H(r/2) \qquad \forall r \geq r_0.
\]
Using the diadic sequence of radii $\{2^jr_0\}_{j \in \N}$ we see that
\[
    3^{2\nu j} H(r_0) \leq H(2^j r_0) \leq C (2^{j})^{2\nu},
\]
by assumption. The above inequality provides a contradiction for $j$ sufficiently large,
yielding the validity of the claim. If we denote with $\bar\bv$ a limiting point
of the sequence $\{\bv_{r_n}\}_{n \in \N}$, it follows that
\[
    \|{\bv}_{r_n}\|_{L^2(\partial^+ B^+_{1/2})}=\sqrt{\frac{H(r_n/2)}{H(r_n)}}
    \geq \sqrt{\frac{1}{C}},
\]
implying, in particular, that $\bar\bv$ is a nontrivial element of $\classG_s$.
Moreover, its Almgren quotient $N_{\bar\bv}(0,r)$ is constant for all $r \in (0, 1/2)$, indeed
\[
    N_{\bar\bv}(0,r) = \lim_{r_n \to \infty} N_{\bv_{r_n}}(0,r) =
    \lim_{r_n \to \infty} N_{\bv}(0,r_n r) = \lim_{r \to \infty} N_{\bv}(0,r) = \nu',
\]
where the latter limit exists by the monotonicity of $N$ (Theorem
\ref{thm:_Almgren_for_classG_c}); moreover, since $\bv$ is not constant we have that $\nu'>0$, while $\nu'\leq\nu$ by Lemma \ref{lem:_first consequence classG_c}.
Since $N(0,r)$ is constant, we conclude by Theorem \ref{thm:_Almgren_for_classG_s} that $\bar\bv$ is homogeneous of degree $\nu'$,  and then it can be extended on the whole $\R^{N+1}_+$ to a $\C_{\loc}^{0,\alpha}$ function, for every $\alpha<\nuACF$.
\end{proof}
%
%
By the previous lemma, if we show that $\nuLio(N)\geq 1/2$ then Theorem \ref{thm: sharp Liouville}
will follow. The next step in this direction consists in reducing such problem to the one of
estimating $\nuLio(1)$.
\begin{lemma}[Dimensional descent]\label{lem: dimesional descent}
For any dimension $N\geq 2$, it holds
\[
    \nuLio (N) \geq \nuLio (N-1).
\]
\end{lemma}
\begin{proof}
For every $\nu>0$ such that there exists $\bv \in \classH(\nu,N)$, we will prove that
$\nuLio (N-1)\leq \nu$. Let $\nu$, $\bv$ as above.
By homogeneity, we have that $\bv(0,0) = 0$ and $N(0,r) = \nu$ for all $r > 0$.
Since the function $\bv$ is homogeneous, its boundary nodal set
\[
    \mathcal{Z} = \{x \in \R^N: \bv(x,0) = 0\}
\]
is a cone at $(0,0)$. We can easily rule out two degenerate situations:
\begin{enumerate}
\item $\mathcal{Z} = \R^N$, in which case all the components of $\bv$ have trivial trace
on  $\R^N$. As a consequence, the odd extension of $\bv$ through $\{y=0\}$ is a nontrivial vector of
harmonic functions on $\R^{N+1}$, forcing $\nu\geq1\geq\nuLio(N-1)$ by Remark \ref{rem:nuLio_leq_1};
\item $\mathcal{Z} = \{ (0,0) \}$, in which case all the components of $\bv$ but one have trivial
trace, and the last one has necessarily a vanishing normal derivative in $\{ y = 0\}$. As before,
extending the former functions oddly and the latter evenly through $\{y=0\}$, we obtain
again that $\nu \geq1\geq\nuLio(N-1)$.
\end{enumerate}
We are left to analyze the third and most delicate case, namely the one in which the boundary
$\partial \mathcal{Z}$ is non trivial.
Let $x_0 \in \partial \mathcal{Z}\setminus\{(0,0)\}$, and let us introduce the following blow up family (here $r \to 0$)
\[
    \bv_r(X) = \frac{1}{\sqrt{H(x_0,r)}} \bv \left((x_0,0) +rX\right).
\]
We want to apply Proposition \ref{prp: compactness in Gs} to (a subsequence of) $\{\bv_r\}_{r}$:
the only assumption non trivial to check is the uniform $L^{\infty}$ bound. To prove it, we observe that
the even extension of $|v_{i,r}|$ through $\{y=0\}$ (denoted with the same writing) is subharmonic, indeed the inequality
\[
    |v_{i,r}|(X) \leq \frac{1}{|B_\rho|}\int\limits_{B_\rho(X)} |v_{i,r}|(Y) \de{Y}
\]
holds true, if $\rho$ is sufficiently small, both when $v_{i,r}(X)=0$ and when $v_{i,r}(X)\neq0$;
once we know that each $|v_{i,r}|$ is non negative and subharmonic, arguing as in the first part of the proof of Lemma \ref{lem:_blowdown_compactness}, one can show that $w_{i,r}$ is uniformly bounded in $L^{\infty}(B_{3/4})$. Applying Proposition \ref{prp: compactness in Gs} we obtain that, up to
subsequences, $\bv_{r}$ converges uniformly and strongly in $H^1$ to $\bar \bv$, an element $\classG_s(N)$ on $B^+_{1/2}$. Reasoning as in the end of the proof of Lemma \ref{lem:_blowdown_compactness}, we
infer that $\bar\bv$ is non trivial, locally $\C^{0,\alpha}$, and that
\[
    N_{\bar\bv}(0,\rho) = \lim_{r \to 0} N_{\bv_{r}}(0,\rho) =
    \lim_{r \to } N_{\bv}(x_0,r\rho) = \lim_{r \to 0} N_{\bv}(x_0,r) =: \nu',
\]
where
\[
    \alpha\leq N_{\bv}(x_0, 0^+) = \nu' \leq N_{\bv}(x_0, \infty) = \nu
\]
by Lemmas \ref{lem: first consequence classG_s}, \ref{lem: second consequence classG_s} and the
monotonicity of $N$. In particular, $\bar\bv$ is homogeneous of degree $\nu'$.

To conclude the proof, we will show that $\bar \bv$ is constant along the
direction parallel to $(x_0,0)$, and that its restriction on the orthogonal half plane belongs to
$\classG_s(N-1)$. Let $(x,y) \in \R^{N+1}_+$ and $h \in \R$ be fixed. By the homogeneity of $\bv$ we have
\begin{multline*}
     \big|\bv_r(x+h (x_0+rx),(1+hr)y) - \bv_r(x,y)\big| \\
    = \frac{\big|\bv((1+hr)(x_0 +rx,ry)) - \bv(x_0 +rx, ry)\big|}{\sqrt{H(x_0,r)}}  \\
   = |(1+ h r)^\nu -1|\,\frac{|\bv(x_0 +rx, ry)| }{\sqrt{H(x_0,r)}}
   = |(1+ h r)^\nu -1|\,|\bv_r(x,y)|.
\end{multline*}
As $r\to0$ (up to subsequences) we infer, by  uniform convergence,
\[
| \bar\bv(x+hx_0,y) - \bar\bv(x,y)| = 0,\qquad\text{ for every }h\in\R.
\]
Let us denote by $\hat \bv$ a section of $\bar\bv$ with respect to the direction $\{h(x_0,0)\}_{h \in \R}$: we claim that $\hat \bv \in \classH(\nu',N-1)$. It is a direct check to verify that
$\hat\bv$ is nontrivial, $\nu'$-homogeneous, and $\C^{0,\alpha}_\loc$. In order to show that $\hat\bv\in\classG_s(N-1)$, we observe that the equations and the segregation conditions are trivially satisfied, therefore we only need to prove the Pohozaev identities on cylindrical domains (recall the discussion before Definition \ref{def:segr ent prof}). To this aim, let $C'$ denote one
of such domains in $\R^{N}_+$, and $C''$ the corresponding domain in $\R^{N+1}_+$ having $C'$ as $N$-dimensional section,  and the further axis parallel to $(x_0,0)$. But then the Pohozaev identity for $\hat \bv$
on $C'$ immediately follows by the one for $\bar \bv$ on $C''$, using Fubini theorem.
\end{proof}
We are ready to obtain the proof of Theorem \ref{thm: sharp Liouville} as a byproduct of the
following classification result, which completely characterize the elements of $\classH(\nu,1)$ and
shows  that $\nuLio (1)=1/2$.

\begin{proposition}\label{prp: classification on plane}
Let $\nu>0$. The following holds.
\begin{enumerate}
 \item $\classH(\nu,1)=\emptyset \iff 2\nu\not\in\N$;
 \item if $\nu\in\N$ any element of $\classH(\nu,1)$ consists in homogeneous polynomial, and
 only one of its components may have non trivial trace on $\{y=0\}$;
 \item if $\nu = k + 1/2$, $k\in\N$, any element of $\classH(\nu,1)$ has exactly two non
 trivial components, say $v$ and $w$, and there exists $c \neq0$ such that
\[
v(\rho, \theta) = c\, \rho^{\frac{1}{2} + k} \cos \left(\frac{1}{2} + k\right) \theta,
\qquad 
w(\rho, \theta) = \pm c\, \rho^{\frac{1}{2} + k} \sin\left(\frac{1}{2} + k\right) \theta
\]
\end{enumerate}
%
%
%
(here $(\rho,\theta)$ denote polar coordinates in $\R^2_+$ around the homogeneity pole).
\end{proposition}
\begin{proof}
Let $\nu>0$ be such that $\classH(\nu,1)$ is not empty, and $\bv \in \classH(\nu,1)$. Since, by
assumption, $\bv$ is homogeneous, the Almgren quotient $N(0,r)$ is equal to $\nu$ for every $r>0$.
Moreover, for topological reasons, no more than two components of $\bv$ can have non trivial
trace on $\{y=0\}$. We will classify $\bv$, and hence $\nu$, according to the number of such
components.

As a first case, let us suppose that two components of $\bv$, say $v$ and $w$, have non trivial
trivial trace, in such a way that they solve
\[
    \begin{cases}
    - \Delta v = 0 & \text{ in } \R_+^2\\
    v (x,0)= 0 & \text{ on } x<0 \\
    \partial_{\nu} v (x,0)= 0 & \text{ on } x>0,
    \end{cases}
    \qquad \text{and} \qquad
    \begin{cases}
    - \Delta w = 0 & \text{ in } \R_+^2\\
    w (x,0)= 0 & \text{ on } x>0\\
    \partial_{\nu} w (x,0)= 0 & \text{ on } x<0.
    \end{cases}
\]
By homogeneity, we can easily find $v$ and $w$; indeed, for instance, $v$ must be of the form
$v(\rho,\theta)=\rho^{\nu} g(\theta)$ with $\nu$ and $g$ solutions to
\[
    \begin{cases}
    \nu^2 g + g'' = 0 & \text{ in } (0,\pi)\\
    g(\pi) = 0, \, g'(0) = 0, \\
    \end{cases}
\]
and an analogous argument holds for $w$. We conclude that
\[
    v(\rho, \theta) = c \rho^{\frac{1}{2} + k} \cos \left(\frac{1}{2} + k\right) \theta,  \qquad w(\rho, \theta) = d \rho^{\frac{1}{2} + k} \sin\left(\frac{1}{2} + k\right) \theta,
\]
with $c, d \neq0$ and $k \in \N$, forcing $\nu = k+1/2$. All the other components of $\bv$ must satisfy
\[
    \begin{cases}
    - \Delta v_i = 0 & \text{ in } \R_+^2\\
    v_i = 0 & \text{ on } \R \times \{ 0 \},
    \end{cases}
\]
with homogeneity degree equal to $k+1/2$, which is impossible unless they are null. Let $\bar{v}$ be the function
\[
    \bar{v}( \rho, \theta ) = \rho^{\frac{1}{2} + k} \cos\left(\frac{1}{2} + k\right) \theta,
\]
in such a way that $v(x,y) = c \bar{v}(x,y)$, while $w(x,y) = d \bar{v}(-x,y)$. Since $\bv$ must
satisfy the Pohozaev identities for the elements of $\classG_s$, we infer that
\[
    \int\limits_{\partial^+ B^+_r(x_0,0)} |\nabla v|^2 + |\nabla w|^2 \, \de\sigma = 2 \int\limits_{\partial^+ B^+_r(x_0,0)} |\partial_{\nu} v|^2 + |\partial_{\nu} w|^2 \, \de\sigma,
\]
for every $x_0 \in \R$ and $r>0$. Considering the choices $x_0 = 1$ and $x_1 = -1$,
and using the symmetries, one has
\begin{align*}
    A_+ c^2  + A_- d^2 &= 2 B_+ c^2  + 2 B_- d^2 ,\\
    A_- c^2  + A_+ d^2 &= 2 B_- c^2  + 2 B_+ d^2 ,
\end{align*}
where
\[
A_\pm=\int\limits_{\partial^+ B^+_r(\pm1,0)} |\nabla \bar{v}|^2  \, \de\sigma,\qquad
B_\pm=\int\limits_{\partial^+ B^+_r(\pm1,0)} |\partial_{\nu} \bar{v}|^2  \, \de\sigma.
\]
Since $A_\pm - 2 B_\pm\neq0$, at least for some $r$, the above equalities force $c^4 - d^4 = 0$,
that is $d = \pm c$. We want to show that
this condition is also sufficient for $(v,w,0,\dots,0)$ to belong to $\classH(\nu,1)$.
To this aim, we only need to prove the actual validity of the Pohozaev identity for any $x_0$ and
$R$. We begin by observing that $v$ and $w$ are conjugated harmonic functions, thus in
particular it holds
\[
    \nabla v \cdot \nabla w = 0 \quad \text{and} \quad |\nabla v| = |\nabla w | \quad \text{in } \R^{2}_+.
\]
Hence, for any unitary vector $\mathbf{n} \in \R^2$ we have
\[
    |\nabla v|^2 = |\nabla w|^2 = |\nabla v \cdot \mathbf{n}|^2 + |\nabla w \cdot \mathbf{n}|^2
    = |\partial_{\mathbf{n}} v|^2 + |\partial_{\mathbf{n}} w|^2,
\]
and the Pohozaev identity follows by integrating over half circles, and choosing
$\nu$ as the outer normal. Resuming, the case in which $\bv$ has two components with non trivial
trace on $\{y=0\}$ always fall into alternative (3) of the statement.

Secondly, let us assume that only one component, say $v$, has non trivial trace on
$\{ y = 0 \}$. Then $\{ v(x,0) > 0 \}$ is either a half line or the entire real line. The first case never happens, since $v$ would solve
\[
    \begin{cases}
    - \Delta v = 0 & \text{ in } \R_+^2\\
    v(x,0) = 0 & \text{ on } x < 0 \\
    \partial_{\nu} v = 0 & \text{ on } x > 0.
    \end{cases}
\]
Reasoning as before, we would obtain that $v$ is of the form
\[
    v(\rho, \theta) = c \rho^{\frac{1}{2} + k} \cos \left(\frac{1}{2} + k\right) \theta,
\]
with $c \in \R$ and $k \in \N$, while all the (odd extensions of the) other components
should be harmonic on $\R^2$ and homogeneous of degree $k+1/2$, that is null;
the Pohozaev identity would force $c = 0$, and $\bv$ would be trivial.
In the second case, if $v(x,0) \neq 0$ for every  $x\neq0$, then $v$ is of the form
\[
    v(\rho, \theta) = c \rho^{k}\cos( k \theta )
\]
with $c \in \R \setminus \{0\}$ and $k \in \N$, while all the other components of $\bv$ are of the form
\[
    v_i(\rho, \theta) = c_i \rho^{k} \sin( k \theta )
\]
for some $c_i \in \R$. Then the case of one non trivial
trace on $\{y=0\}$ always falls into alternative (2) of the statement.

As the last case, let us suppose that $v_i(x,0)\equiv0$ for every $i$. Then each of them is a $\nu$-homogeneous solution to
\[
    \begin{cases}
    - \Delta v_i = 0 & \text{ in } \R_+^2\\
    v_i = 0 & \text{ on } \R \times \{ 0 \}
    \end{cases}
\]
that is, for some $k \in \N$ and $c_i \in \R$, $\nu=1+k$ and
\[
    v_i(\rho, \theta) = c_i \rho^{1+k} \sin( (1+k) \theta ).
\]
Also this case always falls into alternative (2) of the statement, and the proposition follows.
\end{proof}
\section{$\C^{0,\alpha}$ uniform bounds, $\alpha<1/2$}\label{section:_uniform_global}
This section is devoted to the proof of the uniform H\"older bounds, with every exponent
less that $1/2$, for the problem with exterior boundary Dirichlet data. In this direction,
let us consider the problem
\[
    \begin{cases}
    - \Delta v_i = 0 & \text{in } B^+\\
    \partial_{\nu} v_i = f_{i, \beta}(v_i) - \beta v_i \tsum_{j \neq i} v_j^2 &
    \text{on } \partial^0 B^+\cap\Omega\\
    v_i = 0 &     \text{on } \partial^0 B^+\setminus\Omega,
    \end{cases} \eqno \probbd{\beta}
\]
where $\Omega$ is a smooth bounded domain of $\R^N$ and the functions $f_{i,\beta}$ are continuous and uniformly bounded, with respect to $\beta$, on bounded sets.
\begin{remark}
For $\probbd{\beta}$ it is known that, if $\Omega$ is of class $\C^3$, then any $L^\infty$ solution
is in fact $\C^{0,\alpha}$ for every $\alpha<1/2$, see \cite{Shamir}. Furthermore, a uniform
bound holds when $\beta$ is bounded, similarly to Remark \ref{rem: troianiello}. Actually, the assumption
on the smoothness of $\Omega$ can be weakened, at least when considering global problems for
$u(\cdot)=v(\cdot,0)$, as done in the recent paper \cite{serra}.
\end{remark}
We prove the following.
\begin{theorem}\label{thm:_global_holder}
Let $\{\bv_{\beta}\}_{\beta > 0}$ be a family of solutions to problem $\probbd{\beta}$ on $B^+_1$
such that
\[
    \| \bv_{\beta} \|_{L^{\infty}(B^+_1)} \leq \bar m,
\]
with $\bar m$ independent of $\beta$. Then for every $\alpha \in (0,1/2)$ there exists a constant
$C = C(\bar m,\alpha)$, not depending on $\beta$, such that
\[
    \| \bv_\beta\|_{\C^{0,\alpha}\left(\overline{B^+_{1/2}}\right)} \leq C.
\]
Furthermore, $\{\bv_\beta\}_{\beta > 0}$ is relatively compact in $H^1(B^+_{1/2}) \cap \C^{0,\alpha}\left(\overline{B^+_{1/2}}\right)$ for every $\alpha < 1/2$.
\end{theorem}
Actually, two particular cases of the above theorem can be obtained
in a rather direct way.
\begin{remark}
If $\partial^0B^+\cap\Omega=\emptyset$ then Theorem \ref{thm:_global_holder} holds true.
Indeed, the family of functions obtained from $\{\bv_{\beta}\}_{\beta > 0}$ by odd reflection
across $\{y=0\}$ consists in harmonic, $L^\infty$ uniformly bounded functions on $B_1$.
\end{remark}
\begin{remark}[Proof of Theorem \ref{thm: intro_local}]\label{rem:bourbaki}
If $\partial^0B^+\subset\Omega$ then Theorem \ref{thm:_global_holder} holds true.
This is indeed the content of Theorem \ref{thm: intro_local}, that is the one of Theorem
\ref{thm:_local_holder} with $\nuACF$ replaced by $1/2$. In order to prove this result, one can
reason as in the proof of such theorem, by using Theorem \ref{thm: sharp Liouville} instead
of Propositions \ref{prp: liouville system} and \ref{prp: liouville
inequalities} (also recall Remark \ref{rem: basta gamma<1}).
\end{remark}
\begin{proof}[Proof of Theorem \ref{thm:_global_holder}]
The outline of the proof follows the one of Theorem \ref{thm:_local_holder}, to which we refer the reader
for further details. To start with, let $\eta$ be a smooth cutoff function as in equation
\eqref{eqn: eta_blowup}, and let $\alpha \in (0,1/2)$ be fixed. We assume by contradiction that
\[
\begin{split}
    L_n :=& \max_{i = 1, \dots, k} \max_{X'\neq X'' \in \overline{B^+}} \frac{|(\eta v_{i,n})(X')-(\eta v_{i,n})(X'')|}{|X'-X''|^{\alpha}}\\
     =&  \frac{|(\eta v_{1,n})(X'_n)-(\eta v_{i,n})(X''_n)|}{r_n^{\alpha}} \rightarrow \infty,
\end{split}
\]
where, as usual, $\bv_n$ solves $\probbd{\beta_n}$, $\beta_n \rightarrow \infty$, and $r_n := |X'_n-X''_n|\to0$. Furthermore, reasoning as in Lemmas \ref{lem: acc non a part+} and
\ref{lem: shift up_local}, one can prove that the sequences $\{X'_n\}_{n\in\N}$ and
$\{X''_n\}_{n\in\N}$ accumulate near $\partial^0 B^+$ and far away from $\partial^+ B^+$, at least in the scale of $r_n$.

Under the previous notations, we define the blow up sequences
\[
    w_{i,n}(X) := \eta(P_n) \frac{v_{i,n}(P_n + r_n X)}{L_n r_n^{\alpha}} \quad \text{and} \quad \bar{w}_{i,n}(X) := \frac{(\eta v_{i,n})(P_n + r_n X)}{L_n r_n^{\alpha}},
\]
where
\[
P_n := (x'_n,0)
\quad\text{ and }\quad
X\in\tau_n B^+ := \frac{B^+ - P_n}{r_n}.
\]
Such sequences satisfy the following properties:
\begin{itemize}
 \item $\{\bar\bw_n\}_{n\in\N}$ have uniformly bounded H\"older quotient on
 $\tau_n \overline{B^+} $, and $\osc w_{1,n}=1$ for every $n$ on a suitable compact
 set;
 \item each $\bw_n$ solves
 \begin{equation*}
    \begin{cases}
    - \Delta w_{i,n} = 0 & \text{ in }\tau_nB^{+}\\
    \partial_{\nu} w_{i,n} = f_{i,n}(w_{i,n}) - M_n w_{i,n} \tsum_{j \neq i}
    w_{j,n}^2 & \text{on } \tau_n(\partial^0 B^+\cap\Omega)\\
    w_{i,n} = 0 &     \text{on } \tau_n(\partial^0 B^+\setminus\Omega),
    \end{cases}
 \end{equation*}
 where $\sup |f_{i,n}(w_{i,n})| \to 0$ as $n\to\infty$;
 \item $ | \bw_{n}- \bar{\bw}_{n}| \to 0$ uniformly, as $n\to\infty$, on every compact set.
\end{itemize}
By the regularity assumption on $\partial\Omega$ we infer that, up to translations, rotations and subsequences, one of the following three cases must hold.

\textbf{Case 1:} $\tau_n(\partial^0 B^+\setminus\Omega)\to\R^N$. In particular, we have that
$\bw_n(0)=\bar\bw_n(0)=0$ for $n$ large. Reasoning as in Section \ref{section:_uniform_local}
we obtain that both $\bw_n$ and $\bar\bw_n$ converge, uniformly on compact sets, to the same $\bw$
which is harmonic and globally H\"older continuous on $\R^{N+1}_+$, vanishing on $\R^N$ and nonconstant.
But then the odd extension of $\bw$ across $\{y=0\}$ contradicts Liouville theorem.

\textbf{Case 2:} $\tau_n(\partial^0 B^+\cap\Omega)\to\R^N$. In this case, for every compact set $K\subset\overline{\R^{N+1}_+}$, we have that $\{\bw_n|_K\}_{n\in\N}$ and $\{\bar\bw_n|_K\}_{n\in\N}$, for $n$ large, fit in the setting of Section \ref{section:_uniform_local}. Consequently we can argue
exactly in the same way, recalling that the regularity for every $\alpha<1/2$ is obtained by means
of Theorem \ref{thm: sharp Liouville} (see also Remark \ref{rem:bourbaki}).

\textbf{Case 3:} $\tau_n(\partial^0 B^+\cap\Omega)\to\{x\in\R^N:x_1>0\}$. As in the first case,
we have that $\bw_n(0)=\bar\bw_n(0)=0$ for $n$ large, implying that $w_{1,n}\to w_1$, uniformly
on compact sets of $\overline{\R^{N+1}_+}$, with $w_1$ non constant, harmonic, and such that
$w_1(x,0)=0$ for $x_1\leq0$. Finally, reasoning as in Lemma
\ref{lem: uniform implies strong convergence local}, we have that $w_{1,n}\to w_1$ also strongly
in $H^1_\loc$, thus $w_{1}\partial_\nu w_1\leq0$. We are in a position to apply Proposition
\ref{prp: liouville_boundary} to $w_1$ and reach a contradiction.
\end{proof}
Using the above result, we can prove the following global theorem.
\begin{theorem}\label{thm: global_global}
Let
$\{\bv_{\beta}\} \in H^1_{\loc}(\R^N\times(0,1))$ solve
\[
    \begin{cases}
    - \Delta v_{i,\beta} = 0 & \text{in } \R^N\times(0,1)\\
    \partial_{\nu} v_{i,\beta} = f_{i, \beta}(v_{i,\beta}) - \beta v_{i,\beta} \tsum_{j \neq i} v_{j,\beta}^2 &\text{on } \Omega\\
    v_{i,\beta} = 0 &     \text{on } \R^N\setminus\Omega.
    \end{cases}
\]
If there exists a constant $\bar m$ such that
\[
    \| v_{i,\beta}\|_{L^{\infty}(\R^N\times(0,1))} \leq \bar m
\]
then for any $\alpha \in (0,1/2)$
\[
    \| \bv_\beta\|_{\C^{0,\alpha}(\R^N\times[0,1/3])} \leq C(\bar m,\alpha).
\]
Furthermore, $\{\bv_\beta\}_{\beta > 0}$ is relatively compact in $(H^1 \cap \C^{0,\alpha})_{\loc}$
for every $\alpha < 1/2$.
\end{theorem}
\begin{proof}
The proof easily follows by a covering argument. Indeed, we can cover $\R^N\times[0,1/3]$ with a countable number of half-balls of radius $1/2$, centered on $\R^N$, and apply Theorem \ref{thm:_global_holder}
to each of the corresponding half-ball of radius $1$.
\end{proof}
\begin{proof}[Proof of Theorem \ref{thm: intro_global}]
This is actually a corollary of Theorem \ref{thm: global_global}: indeed, if $u\in (H^{1/2}\cap L^\infty)(\R^N)$, and $v\in H^1(\R^{N+1}_+)$ is its unique harmonic extension satisfying
\[
(-\Delta)^{1/2} u(\cdot) = -\partial_y v(\cdot,0),
\]
then $v$ is uniformly bounded in $L^\infty$.
\end{proof}
\begin{remark}
Analogous results can be proved, with minor changes, when the fractional operator considered is the spectral square root of the laplacian, as studied in \cite{ct}. Indeed, in such situation, the corresponding extension problem is given by
\[
    \begin{cases}
    - \Delta v_{i,\beta} = 0 & \text{in } \Omega\times(0,\infty)\\
    \partial_{\nu} v_{i,\beta} = f_{i, \beta}(v_{i,\beta}) - \beta v_{i,\beta} \tsum_{j \neq i} v_{j,\beta}^2 &\text{on } \Omega\times\{0\}\\
    v_{i,\beta} = 0 &     \text{on } \partial\Omega\times(0,\infty),
    \end{cases}
\]
and the starting regularity for $\beta$ bounded is even finer. As a consequence, one can consider the extension of $v$ which is trivial outside $\Omega\times(0,\infty)$, and conclude by using
a modified version of Proposition \ref{prp: liouville_boundary}, suitable for subharmonic functions.
\end{remark}

\section{$\C^{0,1/2}$ regularity of the limiting profiles}

In this section we consider the regularity of the limiting profiles, that is, the
accumulation points of solutions to problem $\problem{\beta}$ as $\beta \to
\infty$. In Section \ref{section:_uniform_local} we proved that, if
$\{\bv_\beta\}_{\beta>0}$ is a family of solutions to problem $\problem{\beta}$,
and $\|\bv_\beta\|_{L^{\infty}(B^+)} \leq \bar m$ for a constant
$\bar m$ independent of $\beta$,
then there exists a sequence $\bv_n:=\bv_{\beta_n}$ such that $\beta_n\to\infty$ and
\[
    \bv_n \to \bv \qquad \text{ in }(H^1\cap \C^{0,\alpha})(K\cap B^+),
\]
for every compact set $K\subset B$ and every $\alpha \in (0,1/2)$.
Now we turn to the proof of Theorem \ref{thm: intro_limiting_prof}, that is, we show that
$\bv \in \C^{0,1/2}_\loc(B^+\cup\partial^0B^+)$. Actually, we will prove such theorem under
a more general assumption: from now on we will assume that the reaction terms in problem
$\problem{\beta}$ satisfy
\[
    \lim_{n \to \infty} f_{i,n} = f_i \quad \text{uniformly in every compact set},
\]
where $(f_1, \dots, f_k)$ are locally Lipschitz, and such that, for some $\eps > 0$,
\begin{equation}\label{eqn: technical assumption}
2F_i(s)-sf_i(s)\geq -C |s|^{2+\eps} \text{ for $s$ sufficiently small,}
\end{equation}
for every $i$, where $F_i(s)=\int\limits_0^s f_i(t)\de t$ (in particular, $f_i(0) = 0$).
\begin{remark}
If $f_i\in \C^{1,\eps}$ in a neighborhood of $0$, and $f_i(0)=0$, then assumption \eqref{eqn: technical assumption} holds true. Indeed this implies that $2F_i(s)-sf_i(s)=O(s^{2+\eps})$ as $s\to0$.
\end{remark}
We will obtain Theorem \ref{thm: intro_limiting_prof} as a byproduct of a stronger result,
in the form of the following proposition.
\begin{proposition}\label{prp: 1/2_reg_lim_prof}
Let $\bv \in H^1(B^+)$ be such that
\begin{enumerate}
 \item $\bv\in(H^1\cap \C^{0,\alpha})(K\cap B^+)$,
 for every compact set $K\subset B$ and every $\alpha \in (0,1/2)$;
 \item $v_i v_j |_{\partial^0 B^+} = 0$ for every $j\neq i$ and
 \begin{equation*}
    \begin{cases}
    - \Delta v_i = 0 & \text{in } B^+\\
    v_i \partial_{\nu} v_i = v_if_{i}(v_i) & \text{on } \partial^0 B^+,
    \end{cases}
 \end{equation*}
 where $f_i$ is locally Lipschitz continuous and satisfies \eqref{eqn: technical assumption}, for every $i=1,\dots,k$;
 \item for every $x_0 \in \partial^0 B^+$ and a.e. $r > 0$ such that $B_r^+(x_0,0) \subset B^+$,
 the following Pohozaev identity holds
\begin{multline*}
    (1-N) \int\limits_{B^+_r } \tsum_{i} |\nabla v_i|^2 \,\de{x}\de{y}  +
    r\int\limits_{\partial^+B^+_r } \tsum_{i} |\nabla v_i|^2 \,\de{\sigma} + \\
    + 2N \int\limits_{\partial^0 B^+_r } \tsum_{i} F_{i}(v_i) \, \de{x}
    - 2r \int\limits_{S_r^{N-1} } \tsum_{i} F_{i}(v_i) \, \de{\sigma}
    = 2r \int\limits_{\partial^+B^+_r }  \tsum_{i} |\partial_{\nu} v_i|^2 \,\de{\sigma}.
\end{multline*}
\end{enumerate}
Then $\bv\in  \C^{0,1/2}(K\cap B^+)$, for every compact $K\subset B$.
\end{proposition}
As we mentioned, Theorem \ref{thm: intro_limiting_prof} will follow from the above proposition
by virtue of the following result.
\begin{lemma}
Let $\beta_n\to\infty$ and $\bv_n$ solve problem $\problem{\beta_n}$, for every $n$,
be such that
\[
    \bv_n \to \bv \qquad \text{ in }(H^1\cap \C^{0,\alpha})(K\cap B^+),
\]
for every compact set $K\subset B$ and every $\alpha \in (0,1/2)$. Moreover,
let the corresponding reaction terms $f_{i,n}$ converge, uniformly on compact sets, to the
locally Lipschitz functions $f_i$ satisfying \eqref{eqn: technical assumption}.  Then $\bv$ fulfills the assumptions of the Proposition \ref{prp: 1/2_reg_lim_prof}.
\end{lemma}
\begin{proof}
The proof follow the line of the one of Proposition  \ref{prp: approx classG_s}, with
minor changes.
\end{proof}
In view of the previous lemma, with a slight abuse of terminology, we will denote as limiting profiles
also functions which simply satisfy the assumptions of Proposition \ref{prp: 1/2_reg_lim_prof}. For
the rest of this section we will denote with $\bv$ a fixed limiting profile.

In the proof of Proposition \ref{prp: 1/2_reg_lim_prof} we shall use a further monotonicity formula of Almgren type. For every $x_0 \in \partial^0 B^+$ and $r > 0$ such that $B_r^+(x_0,0) \subset B^+$, we introduce the functions
\[
    \begin{split}
    E(x_0, r) &:= \frac{1}{r^{N-1}} \left( \int\limits_{B^+_r(x_0,0)}
\tsum_{i} |\nabla v_i|^2 \, \de{x} \de{y} - \int\limits_{\partial^0
B^+_r(x_0,0)} \tsum_{i} f_{i}(v_i) v_i \, \de{x} \right)\\
    H(x_0, r) &:= \frac{1}{r^{N}} \int\limits_{\partial^+ B^+_r(x_0,0)}
\tsum_{i} v_i^2 \, \de\sigma.
    \end{split}
\]
As usual, the function $E(x_0,r)$ admits an equivalent
expression: indeed, multiplying the equation in assumption (2)
by $v_i$, integrating over $B_r^+(x_0,0)$ and summing over $i=1,
\dots, k$ we obtain
\begin{equation}\label{eqn: E and H prime}
    E(x_0,r) = \frac{1}{r^{N-1}} \int\limits_{\partial^+ B^+_r(x_0,0)}
\tsum_{i} v_i \partial_{\nu}v_i \de{\sigma} = \frac{2}{r} H'(x_0, r).
\end{equation}
The presence of internal reaction terms in the definition of the function $E$ has to be dealt with. To this end, the next two lemmas will provide a crucial estimate in order to bound the Almgren quotient. Before we state them, let us recall the following Poincar\'e inequality: for every $p \in [2,p^\#]$,
where $p^\#=2N/(N-1)$ denotes the critical Sobolev exponent for trace embedding (or simply $p\geq2$ in dimension $N=1$), there exists a constant $C_P = C_P(N,p)$ such that, for every $w \in H^{1}(B^+_r)$,
\begin{equation}\label{eqn: modified Poincare}
    \left[\frac{1}{r^{N}} \int\limits_{\partial^0 B_r^+} |w|^p \, \de{x}\right]^{\frac{2}{p}}
\leq C_P \left[\frac{1}{r^{N-1}} \int\limits_{B_r^+} |\nabla w|^2 \, \de{x}
\de{y} + \frac{1}{r^N} \int\limits_{\partial^+ B_r^+} w^2 \, \de{\sigma}
\right]
\end{equation}
(such an inequality follows by the one on $B^+$ by scaling arguments).
\begin{lemma}\label{lem: first lower bound on E+aH}
For every $p \in [2,p^\#]$ there exist constants $C >0$, $\bar{r} >0$ such that
\[
 \left[\frac{1}{r^{N}} \int\limits_{\partial^0 B_r^+} \tsum_i |v_i|^p \, \de{x}\right]^{\frac{2}{p}} \leq C\left[  E(r) + H(r) \right]\quad\text{for every $r \in (0,\bar{r})$}.
\]
\end{lemma}
\begin{proof}
Since $\bv\in L^\infty(B^+)$, and each $f_i$ is locally Lipschitz continuous with $f_i(0)= 0$,
we have
\begin{equation*}
\begin{split}
    \left|\frac{1}{r^{N-1}}  \int\limits_{\partial^0 B^+_r} \tsum_{i}f_{i}(v_i) v_i \, \de{x} \right|
    &\leq C  \frac{1}{r^{N-1}}\int\limits_{\partial^0 B^+_r} \tsum_{i} v_i^2 \, \de{x}\\
    &\leq C'r \left[\frac{1}{r^{N-1}} \int\limits_{B_r^+} \tsum_{i} |\nabla v_i|^2 \, \de{x}
\de{y} + \frac{1}{r^N} \int\limits_{\partial^+ B_r^+} \tsum_{i} v_i^2 \, \de{\sigma}
\right],
\end{split}
\end{equation*}
where we used inequality \eqref{eqn: modified Poincare} with $p=2$. As a consequence,
\begin{equation}\label{eqn: equazione non eliminabile}
    E(r)+ H(r)\geq
     (1-Cr) \left[\frac{1}{r^{N-1}} \int\limits_{B_r^+} \tsum_{i} |\nabla v_i|^2 \, \de{x}
\de{y} + \frac{1}{r^N} \int\limits_{\partial^+ B_r^+} \tsum_{i} v_i^2 \, \de{\sigma}
\right],
\end{equation}
and the lemma follows by taking into account equation \eqref{eqn: modified Poincare}
and choosing
\(
\bar{r}
\)
sufficiently small.
\end{proof}
For the following lemma we introduce, for  $p \in (2,p^\#]$, the auxiliary function
\[
	\psi(x_0,r) := \left(\frac{1}{r^N} \int \limits_{\partial^0 B^+_r(x_0,0)} \tsum_i |v_i|^p \, \de x \right)^{1-\frac{2}{p}},
\]
which is bounded for $r$ small. We have the following.
\begin{lemma}\label{lem: bound by AC function}
For every $p \in (2,p^\#]$ there exist constants $C >0$, $\bar{r} >0$ such that
\[
   \frac{1}{r^{N-1}} \int \limits_{S_r^{N-1}} \tsum_i |v_i|^p \, \de{\sigma} \leq C  \left[ E(r) + H(r)\right]  \cdot \frac{\de}{\de r}(r \psi(r))\quad\text{for every $r \in (0,\bar{r})$}.
\]
\end{lemma}
\begin{proof}
A direct computation yields the identity
\[
\begin{split}
\frac{\de}{\de r} \psi(r)  &= \left(1-\frac{2}{p}\right) \psi^{-2/(p-2)}
\left(\frac{1}{r^N} \int \limits_{\partial^0 B^+_r(x_0,0)} \tsum_i |v_i|^p \, \de x \right)'\\
&= \left(1-\frac{2}{p}\right) \psi(r) \frac{
\left(r^{-N} \int \limits_{\partial^0 B^+_r} \tsum_i |v_i|^p \, \de x \right)'}{
r^{-N} \int \limits_{\partial^0 B^+_r} \tsum_i |v_i|^p \, \de x }.
\end{split}
\]
As a consequence we infer
\[
   \frac{\de}{\de r}(r \psi(r)) = \psi(r) \left[ r \left(1-\frac{2}{p}\right)\frac{ \int \limits_{S_r^{N-1}} \tsum_i |v_i|^p \, \de{\sigma}}{  \int \limits_{\partial^0 B^+_r} \tsum_i |v_i|^p \, \de{\sigma}}+\left(1-N\left(1-\frac{2}{p}\right)\right)\right].
\]
Now, $p \leq p^\#$ implies $N\left(1-\frac{2}{p}\right) \leq 1$, so that
\[
    \frac{\de}{\de r}(r \psi(r)) \geq r \psi(r) \left(1-\frac{2}{p}\right)\frac{ \int \limits_{S_r^{N-1}} \tsum_i |v_i|^p \, \de{\sigma}}{  \int \limits_{\partial^0 B^+_r} \tsum_i |v_i|^p \, \de{\sigma}}.
\]
Recalling the definition of $\psi$ and using Lemma \ref{lem: first lower bound on E+aH}, we finally obtain
\[
   ( E(r) + H(r)) \frac{\de}{\de r}(r \psi(r)) \geq C \frac{1}{r^{N-1}} \int \limits_{S_r^{N-1}} \tsum_i |v_i|^p \, \de{\sigma},
\]
where, since $p>2$, $C>0$.
\end{proof}
As a matter of fact, we need to estimate the Almgren quotient only on the zero set of $\bv$ (which is well defined since $\bv$ is continuous).
\begin{definition}
We define the boundary zero set of the limiting profile $\bv$ as
\[
    \mathcal{Z} = \{x \in \partial^0B^+: \bv(x,0) = 0 \}.
\]
\end{definition}

\begin{remark}
A natural notion of free boundary, associated to a limiting profile $\bv$, is the set in which the boundary condition of assumption (2) does not reduce to
\[
    \partial_{\nu} v_i = f_i(v_i), v_j \equiv 0 \quad \text{for some } j \neq i,
\]
that is, \emph{a posteriori}, the support of the singular part of the measure $\partial_{\nu} \mathbf{v}$. It is then clear that the free boundary is a subset of $\mathcal{Z}\subset\R^N$.
\end{remark}
We are now in a position to state the Almgren type result which we use in this framework. As we mentioned, we prove it only at points of $\mathcal{Z}$; furthermore, it concerns boundedness of a (modified)
Almgren quotient, rather than its monotonicity. More precisely, let us consider the function
\[
N(x_0,r) := \frac{E(x_0,r)}{H(x_0, r)} + 1.
\]
We have the following.
\begin{lemma}\label{lem: Almgren monotonicity for the limiting profiles}
There exist constants $C>0$, $\bar{r}>0$ such that, for every $x_0 \in {\mathcal{Z}}$, $r \in (0, \bar r)$ and $B^+_r(x_0,0) \subset B^+$, we have:
\begin{enumerate}
 \item $H(r)>0$, $N(r)>0$  on $(0, \bar r)$;
 \item the function $r \mapsto e^{Cr(1+\psi(r))}N(x_0,r)$ is monotone non decreasing;
 \item $N(x_0,0^+)\geq 1+\dfrac12$.
\end{enumerate}
\end{lemma}
\begin{proof}
The proof is similar to the one of Theorem \ref{thm:_Almgren_for_classG_s}, but in this case the internal reaction terms  do not vanish. Let $x_0 \in {\mathcal{Z}}$ and let $\bar{r}$ be such that both Lemma \ref{lem: first lower bound on E+aH} and Lemma \ref{lem: bound by AC function} hold. First, we ensure that the Almgren quotient, where defined, is non negative. Indeed, by Lemma \ref{lem: first lower bound on E+aH},
\[
    E(r) + H(r) \geq 0  \implies N(r) = \frac{E}{H} + 1 \geq 0,
\]
whenever $H(r) \neq 0$. By continuity of $H$ we can consider, as in the proof of Theorem \ref{thm:_Almgren_for_classG_s}, a neighborhood of $r$ where $H$ does not vanish.
We compute the derivative of $E$ and we use the Pohozaev identity (assumption
(3) of Proposition \ref{prp: 1/2_reg_lim_prof}), to obtain
\begin{multline*}
    E'(r)= \frac{1-N}{r^{N}} \left( \int\limits_{B^+_r}
\tsum_{i} |\nabla v_i|^2 \, \de{x} \de{y} - \int\limits_{\partial^0
B^+_r} \tsum_{i} v_if_{i}(v_i) \, \de{x} \right)\\ + \frac{1}{r^{N-1}} \left( \int\limits_{\partial^+ B^+_r}
\tsum_{i} |\nabla v_i|^2 \, \de{x} \de{y} - \int\limits_{ S^{N-1}_r} \tsum_{i} v_i f_{i}(v_i) \, \de{x} \right)\\
= \underbrace{\frac{2}{r^{N-1}} \int\limits_{\partial^+B^+_r }  \tsum_{i} |\partial_{\nu} v_i|^2 \,\de{\sigma}}_{T} + \underbrace{\frac{1}{r^{N}} \int\limits_{\partial^0
B^+_r} \left[(N-1) \tsum_{i} v_if_{i}(v_{i}) - 2N \tsum_{i}
F_{i}(v_{i})\right] \, \de{x}}_{I} \\ + \underbrace{\frac{1}{r^{N-1}} \int\limits_{S_r^{N-1}} \left[- \tsum_{i}
v_i f_{i}(v_{i}) + 2 \tsum_{i} F_{i}(v_{i})\right] \, \de{\sigma}}_{Q}.
\end{multline*}
Since $\bv \in L^{\infty}$, $f_i$ are locally Lipschitz and $f_i(0) = 0$, there exists a positive constant $C$, such that
\begin{equation*}
    |f(v_i)v_i| \leq C v_i^2 \text{ and }  |F(v_i)| \leq C v_i^2.
\end{equation*}
The direct application of Lemma \ref{lem: first lower bound on E+aH} (with $p=2$) provides
\[
I \geq - C (E+H).
\]
On the other hand, by assumption \eqref{eqn: technical assumption} and Lemma \ref{lem: bound by AC function} (it is sufficient to choose $p = \min\{2+\eps, p^\#\}$), we obtain
\[
     Q \geq -C(E+ H) (r \psi)'.
\]
The two estimates yield
\[
  E'\geq T - C\left[ 1 +  (r \psi)' \right](E+H).
\]
Therefore, differentiating the Almgren quotient and using the Cauchy-Schwarz inequality, we obtain
\[
    \frac{N'}{N} =  \frac{E'+H'}{E+H} - \frac{H'}{H} \geq  \frac{TH - E H'}{H(E+H)}- C\left[ 1 +
    (r \psi)' \right] \geq - C\left[ 1 +  (r \psi)' \right],
\]
which implies that the function $e^{Cr(1+\psi(r))} N(r)$ is non decreasing as far as $H(r)\neq0$.
Equation \eqref{eqn: E and H prime} directly implies
\[
    \frac{\mathrm{d}}{\mathrm{d}r} \log H(r) = \frac{H'(r)}{H(r)} = \frac{2E(r)}{rH(r)} = \frac{2(N(r)-1)}{r};
\]
reasoning as  in the proof of Theorem  \ref{thm:_Almgren_for_classG_s}, we can use this formula,
together with the bound
\[
N(r) \leq e^{Cr^*(1+\psi(r^*))}N(r^*)\quad\text{ for every  }r \leq r^*,
\]
in order to obtain the strict positivity of
$H$ for $r \in (0,\bar r)$ (for a possibly smaller $\bar r$). Finally, reasoning as in the proof of
Lemma \ref{lem: first consequence classG_s}, part (2), let us assume by contradiction that, for some
$r^*<\bar r$ and $\eps>0$, $e^{Cr^*(1+\psi(r^*))}N(r^*)\leq \frac32-\eps$. By the above bound
we obtain that
\[
\begin{split}
 \frac{\mathrm{d}}{\mathrm{d}r} \log H(r)
       \leq \frac{2(e^{Cr^*(1+\psi(r^*))}N(r^*) - 1)}{r} \leq \frac{1-2\eps}{r}
\end{split}
\]
for every $r\in(0,r^*)$. But this is in contradiction with the fact that $\bv$ is in
$\C^{0,\alpha}$ for $\alpha=(1-\eps)/2$.
\end{proof}
The proof Proposition \ref{prp: 1/2_reg_lim_prof} is based on a contradiction argument, involving Morrey inequality. Indeed, let $K \subset B$ be compact, and let us define, for every $X\in K\cap \{y\geq0\}$ and every $r<\dist(K, \partial B)$, the function
\[
    \Phi(X,r):=\frac{1}{r^N} \int\limits_{B_{r}(X) \cap \{y > 0\}} \tsum_{i} |\nabla v_i|^2 \, \de{x} \de{y} .
\]
It is well known that if $\Phi$ is bounded then $\bv\in \C^{0,1/2}(K\cap B^+)$.

As a consequence of Lemma \ref{lem: Almgren monotonicity for the
limiting profiles}, we can prove a first estimate on $\Phi$.
\begin{lemma}\label{lem:_upperbound_Morrey}
For every compact $K\subset B$ there exists constants $C>0$, $\bar r>0$, such that for every $x_0 \in
{\mathcal{Z}\cap K}$ and $r \in (0, \bar r)$, it holds
\[
    \Phi(x_0,r) \leq C.
\]
\end{lemma}
\begin{proof}
If $\bar r$ is sufficiently small, from Lemma \ref{lem: Almgren monotonicity for the limiting profiles} we know that
\[
 \frac32 e^{-Cr(1+\psi(r))} \leq  N(r)  \leq  C
\]
for every $r \in (0,\bar r)$. Since $E+H=NH$, equation \eqref{eqn: equazione non eliminabile} implies that
\begin{equation*}
    \frac{1}{r^{N}} \int\limits_{B^+_r(x_0,0)} \tsum_{i} |\nabla v_i|^2 \,
\de{x} \de{y} \leq C \frac{H(r)}{r}.
\end{equation*}
On the other hand, by the lower estimate on $N$,
\[
    \frac{\mathrm{d}}{\mathrm{d} r} \log \frac{H(r)}{r} 
    \geq  3 \frac{e^{-Cr(1+\psi(r))}-1}{r} \geq -3C(1+\psi(r)) \geq - C.
\]
Integrating, we obtain
\[
\frac{H(r)}{r} \leq e^{C\bar r}\frac{H(\bar r)}{\bar r},
\]
and the lemma follows.
\end{proof}
The above result can be complemented by the following lemma.
\begin{lemma}\label{lem:_upperbound_Morrey 2}
For every compact $K\subset B$ there exist constants, $C>0$, $\bar r>0$, such that for every $x_0 \in
{(K\cap\{y=0\})\setminus\mathcal{Z}}$ and
\[
0< r <d:=\min\{\dist(x_0,\mathcal{Z}), \bar r\},
\]
it holds
\[
   \Phi(x_0,d)\geq C \Phi(x_0,r).
\]
\end{lemma}
\begin{proof}
Since $x_0\not\in\mathcal{Z}$ and $r \leq \dist(x_0,\mathcal{Z})$, we can assume that
$v_j \equiv 0$ on $\partial^0 B^+_r{(x_0,0)}$ for, say, $j\geq 2$. As a consequence,
the odd extension of $v_j$ across $\{y=0\}$ is harmonic on $B_r(x_0,0)$, and the
mean value property applied to the subharmonic function $|\nabla v_j|^2$ provides
\begin{equation}\label{eqn: phi fuori da Z}
\frac{1}{r^N} \int\limits_{B^+_{r}(x_0,0)} |\nabla v_j|^2 \, \de{x} \de{y}
\leq  \frac{r}{d}
\frac{1}{d^N} \int\limits_{B^+_{d}(x_0,0)} |\nabla v_j|^2 \, \de{x} \de{y},
\quad\text{ for every }j\geq2.
\end{equation}
We now show that a similar estimate holds true also for $v_1$.
Indeed, let $ u :=|\nabla v_1|^2$; by a straightforward computation, we have that
\[
\begin{cases}
    - \Delta  u  \leq 0 &\text{ in } B_{d}^+\\
    \partial_{\nu}  u  \leq a  u  &\text{ in }\partial^0 B_{d}^+,
\end{cases}
\]
where $a := 2 ||f_1'(v_1)||_{L^{\infty}(B^+)}$ is bounded by assumption. Now, by scaling, one can show
that if $\bar r = \bar r(a)$ is sufficiently small, then the equation
\begin{equation*}
\begin{cases}
    - \Delta \varphi = 0 &\text{ in }B_{\bar{r}}^+\\
    \partial_{\nu} \varphi = a \varphi &\text{ on } \partial^0 B_{\bar{r}}^+
\end{cases}
\end{equation*}
admits a strictly positive (and smooth) solution. By the definition of $d$ we deduce that
\[
\begin{cases}
    - \div\left( \varphi^2 \nabla\frac{ u }{\varphi} \right) \leq 0 &B_{d}^+\\
    \varphi^2 \partial_{\nu} \frac{ u }{\varphi} \leq 0  &\partial^0 B_{d}^+,
\end{cases}
\]
so that the even extension of $ u $ is a solution to
\[
  - \div\left( \varphi^2 \nabla\frac{ u }{\varphi} \right) \leq 0 \quad \text{in } B_{d}.
\]
Integrating such equation on any ball $B_r$, we obtain
\[
  \int\limits_{\partial B_r} \varphi^2 \partial_{\nu} \frac{u}{\varphi} \de{\sigma} \geq 0
\]
If we introduce the function
\[
  H(r) = \frac{1}{r^{N}} \int\limits_{\partial B_r} \varphi u \de{\sigma} =  \int\limits_{\partial B} \varphi^2(rx) \frac{u(rx)}{\varphi(rx)} \de{\sigma}  ,
\]
a straightforward computation shows that
\[
  H'(r) = \frac{2}{r^{N}} \int\limits_{\partial B_r} u \varphi \frac{\partial_{\nu} \varphi}{ \varphi } \de{\sigma} + \frac{1}{r^{N}} \int\limits_{\partial B_r} \varphi^2 \partial_{\nu} \frac{u}{\varphi} \de{\sigma} \geq - 2 \left\| \frac{\partial_{\nu} \varphi}{\varphi} \right\|_{L^{\infty}(B)} H(r) \geq -C H(r),
\]
that is, the function $r \mapsto e^{Cr} H(r)$ is monotone non decreasing in $r$. Hence, for every
$0 < r_1 \leq r_2 \leq d$, we obtain that $H(r_1)\leq C H(r_2)$. Multiplying by
$r_1^N r_2^N$ and integrating in $(0,r) \times (r,d)$, with $r \leq d$, we obtain
\[
	\left(1-\frac{r^{N+1}}{d^{N+1}}\right)\frac{1}{r^{N+1}} \int\limits_{ B_r}\varphi u  \, \de{x} \de{y} \leq \frac{C}{d^{N+1}}\int\limits_{B_d\setminus B_r}\varphi u  \,  \de{x} \de{y} .
\]
Adding $C d^{-N-1} \int\limits_{B_r}\varphi u  \,  \de{x} \de{y}$, we infer
\[
  \frac{1}{r^{N+1}} \int\limits_{ B_{r}} \varphi u \de x\de y \leq C \frac{1}{d^{N+1}}
  \int\limits_{ B_{d}} \varphi u \de{x}\de y.
\]
Recalling that $\varphi$ is positive and bounded, and that $u=|\nabla v_1|^2$, we finally obtain that
\begin{equation*}
\frac{1}{r^N} \int\limits_{B^+_{r}(x_0,0)} |\nabla v_1|^2 \, \de{x} \de{y}
\leq  C\frac{r}{d}
\frac{1}{d^N} \int\limits_{B^+_{d}(x_0,0)} |\nabla v_1|^2 \, \de{x} \de{y}.
\end{equation*}
The lemma now follows by summing up with inequality \eqref{eqn: phi fuori da Z}, for $j=2,\dots,k$, and recalling that $d/r\geq1$.
\end{proof}
\begin{proof}[End of the proof of Proposition \ref{prp: 1/2_reg_lim_prof}]
Let us assume by contradiction that there exists a sequence $\{(X_n,r_n)\}_{n \in \N}$ such that $X_n=(x_n,y_n)\in K\cap \{y\geq0\}$, $r_n<\dist(K, \partial B)$, and
\[
    \Phi(X_n,r_n)\to + \infty,\quad\text{ as }n\to\infty.
\]
It is immediate to prove that $r_n \to 0$ and $y_n \to 0$: indeed, $\bv$ is $H^1$ and harmonic for $\{y>0\}$. In particular, the sequence $\{X_n\}_{n\in \N}$ accumulates at $\partial^0 K$. First we observe that,  thanks to the subharmonicity of $\tsum_{i} |\nabla v_i|^2$, if $r_n< y_n$ then
\[
 \Phi(X_n,y_n)\geq\frac{y_n}{r_n}\Phi(X_n,r_n) \geq \Phi(X_n,r_n) ;
\]
as a consequence we can assume without loss of generality that $r_n\geq y_n$. Analogously, once
$r_n\geq y_n$, we have that
\[
\Phi((x_n,0),2r_n) \geq \frac{1}{2^N} \Phi(X_n,r_n),
\]
and again, without loss of generality, we can assume that $y_n=0$ for every $n$, and drop it from
our notation.

Now, by the result of Lemma \ref{lem:_upperbound_Morrey 2}, the sequence $(x_n,r_n)$ can be replaced by a sequence of points in ${\mathcal{Z}}$. Indeed, if $\dist(x_n,\mathcal{Z}) > \bar r$ for every $n \in \N$, then
\[
	\Phi(x_n,r_n) \leq C \Phi(x_n,\bar r)
\]
and the right hand side is bounded since $\bv \in H^1(B^+)$. Consequently, it must be $\dist(x_n,\mathcal{Z}) \leq \bar r$, and then
\[
	\Phi(x_n,r_n) \leq C \Phi(x_n, \dist(x_n,\mathcal{Z})).
\]
Since the set ${\mathcal{Z}}$ is locally closed and $\dist(K,\partial^+B) > 0$, for $n$ sufficiently large, to each $x_n$ we can associate $x'_n \in \mathcal{Z}$ such that $\dist(x_n,\mathcal{Z}) = |x_n-x'_n| \leq \frac12 \dist(x_n,\partial^+B)$ and we can substitute the sequence $(x_n, \dist(x_n,\mathcal{Z}))$ with $(x'_n, 2\dist(x_n,\mathcal{Z}))$. We are in position to apply Lemma \ref{lem:_upperbound_Morrey} and find a contradiction to the unboundedness of the Morrey quotient.
\end{proof}

%


\end{document}